\documentclass[reqno,a4paper,11pt]{amsart}
\usepackage{fullpage}

\usepackage[x11names]{xcolor} 

\usepackage[T1]{fontenc}
\usepackage[utf8]{inputenc} 

\usepackage{amsfonts,amsmath,amssymb}
\usepackage[fleqn]{mathtools}
\usepackage{float}
\usepackage{bm}
\usepackage{upgreek}
\usepackage{slashed}
\usepackage{accents}
\usepackage{youngtab}

\usepackage[ddmmyyyy]{datetime}

\usepackage{tikz}
\usepackage{graphicx}
\usepackage[all]{xy}
\usepackage{array}
\usepackage{scalerel}
\usepackage{multirow}
\usepackage[shortlabels]{enumitem}
\usepackage{caption}
\usepackage{hyperref}

\usepackage{amsthm}
\theoremstyle{plain}
\newtheorem{thm}{Theorem}[section]
\newtheorem*{thm*}{Theorem}
\newtheorem{lem}[thm]{Lemma}
\newtheorem{prop}[thm]{Proposition}
\newtheorem{cor}[thm]{Corollary}
\theoremstyle{definition}
\newtheorem{defn}[thm]{Definition}
\newtheorem{conj}[thm]{Conjecture}
\newtheorem{exa}[thm]{Example}

\theoremstyle{remark}
\newtheorem{rem}[thm]{Remark}

\numberwithin{equation}{section}
\newcommand{\parderv}[2] {\frac{\partial #1}{\partial #2}}

\renewcommand{\d} {\mathrm{d}}

\newcommand{\ol} [1] {\overline{#1}}
\newcommand{\wh} [1]{\widehat{#1}}
\newcommand{\wt} [1]{\widetilde{#1}}
\newcommand{\mr}[1] {\mathring{#1}}
\newcommand{\ul} [1]{\underline{#1}}

\newcommand{\mbf} [1]{\mathbf{#1}}
\newcommand{\mbb} [1]{\mathbb{#1}}
\newcommand{\mc} [1]{\mathcal{#1}}

\renewcommand{\i} {\mathrm{i}}
\newcommand{\e} {\mathrm{e}}

\newcommand{\Id} {\mathrm{Id}}

\renewcommand{\bigwedge}{\scaleobj{1.2}{\wedge}}

\renewcommand{\bigodot}{\scaleobj{1.2}{\odot}}

\newcommand{\hook}{\makebox[7pt]{\rule{6pt}{.3pt}\rule{.3pt}{5pt}}\,}

\newcommand{\Riem} {\mathsf{R}}
\newcommand{\Weyl} {\mathsf{W}}
\newcommand{\Ric} {\mathsf{Ric}}
\newcommand{\Sc} {\mathsf{Sc}}
\newcommand{\WbA} {\mathsf{A}}

\newcommand{\Rho} {\mathsf{P}}
\newcommand{\Bach} {\mathsf{B}}
\newcommand{\Cot} {\mathsf{Y}}
\newcommand{\Nh} {\mathsf{N}}

\newcommand{\CM} {\mathsf{S}}
\newcommand{\FG} {\mathsf{F}\mkern-4mu\mathsf{G}}

\newcommand{\Ann} {\mathrm{Ann}}

\newcommand{\Z} {\mathbf{Z}}

\newcommand{\R} {\mathbf{R}}
\newcommand{\C} {\mathbf{C}}

\newcommand{\CP} {\mathbf{C}\mathbb{P}}

\begin{document}

	\title
{Perturbations of Fefferman spaces\\
	over almost CR manifolds}

	\begin{abstract}
 We construct a one-parameter family of Lorentzian conformal structures on the canonical circle bundle of a partially integrable contact almost Cauchy--Riemann manifold. This builds on previous work by Leitner, who generalised Fefferman's construction associated to a CR manifold to the non-involutive case. We provide characterisations of these conformal structures and show that they admit distinguished pure spinor fields.

We introduce exact `perturbations' of such Fefferman spaces by a semi-basic one-form, which can be suitably interpreted as a tuple of weighted tensors on the almost CR manifold. The resulting perturbed conformal space is an instance of a so-called nearly Robinson manifold introduced recently by Fino, Leistner and the present author.

 We investigate the existence of metrics in these conformal classes which satisfy appropriate subsystems of the Einstein equations. These metrics are defined only off cross-sections of Fefferman's circle bundle, and are conveniently expressed in terms of almost Lorentzian densities, which include Gover's almost Einstein scales as a special case. In particular, in dimensions greater than four, almost Einstein scales always arise from so-called CR--Einstein structures, almost CR analogues of Einstein metrics. We derive necessary and sufficient conditions for a perturbed Fefferman space to be conformally flat on the zero set of an almost Einstein scale.

We construct an explicit example of a CR--Einstein structure on a strictly almost CR five-manifold based on a strictly almost K\"{a}hler--Einstein four-manifold due to Nurowski and Przanowski.
	\end{abstract}
	
	\date{\today\ at \xxivtime}
	
	\author{Arman Taghavi-Chabert}\address{Institute of Physics, {\L}\'{o}d\'{z} University of Technology, W\'{o}lcza\'{n}ska 217/221, 90-005 {\L}\'{o}d\'{z}, Poland}\email{arman.taghavi-chabert@p.lodz.pl}
	
	\thanks{}
	
	\maketitle
	\tableofcontents

\section{Introduction}
	This article is concerned with the interaction between Lorentzian conformal geometry and Cauchy--Riemann geometry, where the latter arises as the leaf space of a foliation by null geodesics on the former. This geometric setting can be found in two distinct and independent contexts. The first of these, due to Fefferman \cite{Fefferman1976a,Fefferman1976}, associates a non-degenerate, or contact, CR structure of hypersurface type to a conformal structure of Lorentzian signature on a canonical circle bundle. There are various approaches to the construction: in the original version, Fefferman constructed a biholomorphically invariant conformal structure on a circle bundle over the boundary of a bounded strictly convex domain of the complex space $\C^{m+1}$ by means of a formal solution to the Monge--Amp\`{e}re equation. Later formulations allow for the CR manifold $(\mc{M},H,J)$ to be abstract, that is, not necessarily embeddable. Of particular relevance in this article is the construction due to Lee \cite{Lee1986}, who showed that Fefferman's conformal structure lives on the total space of a circle bundle over $(\mc{M},H,J)$, which is a certain quotient of the canonical bundle $\mc{C}$ of $\mc{M}$. \v{C}ap and Gover \cite{Cap2008} later substituted $\mc{C}$ for a certain root thereof, which allows them to work within the framework of parabolic geometry and tractor calculus. In all of these versions, one associates to each contact form $\theta$ a so-called \emph{Fefferman metric}
	\begin{align}\label{eq:Fefmet}
		\wt{g}_{\theta} & = 4 \theta \odot \wt{\omega}^{W}_{\theta} + h \, ,
	\end{align}
	where $\wt{\omega}^{W}_{\theta}$ is the induced Weyl connection one-form compatible with $\theta$, and $h$ is the (degenerate) metric induced from the Levi form of $\theta$. Conformally related contact forms yield conformally related Fefferman metrics. The conformal structure thus constructed is characterised by, among other conditions, the existence of a certain \emph{null conformal Killing field}, and the CR manifold is recovered as the leaf space of its integral curves.
	
	Fefferman's construction was later adapted by Leitner \cite{Leitner2010} to the case where $(\mc{M},H,J)$ is a partially integrable non-degenerate almost CR manifold --- when $\mc{M}$ has dimension greater than three --- and also to the case where the Weyl connection $\wt{\omega}^{W}_{\theta}$ in \eqref{eq:Fefmet} is gauged by some horizontal one-form.

	The other notion connecting Lorentzian geometry and almost Cauchy--Riemann geometry is known as \emph{almost Robinson geometry}, an idea introduced by Nurowski and Trautman in \cite{Nurowski2002} in the involutive case --- see also \cite{Hughston1988} --- and developed further by Fino, Leistner and the present author in \cite{Fino2023} in the more general case. An almost Robinson manifold is essentially a Lorentzian analogue of an almost Hermitian manifold. This simply consists of a totally null complex $(m+1)$-plane distribution $\wt{N}$ on a $(2m+2)$-dimensional Lorentzian conformal manifold $(\wt{\mc{M}},\wt{\mbf{c}})$. This distribution intersects its complex conjugate in the complexification of a real null line distribution $\wt{K}$. When $\wt{N}$ is preserved along any nowhere vanishing section $\wt{k}$ of $\wt{K}$, the integral curves of $\wt{k}$ are null geodesics, and $\wt{N}$ gives rise to an almost CR structure on the leaf space of the congruence generated by $\wt{k}$. We then say that $\wt{N}$ is \emph{nearly Robinson}. In dimension four, this setting is equivalent to a so-called \emph{optical geometry} with \emph{non-shearing congruence of null geodesics}. The terminology \emph{shear-free}, instead of `non-shearing', is frequently employed. When the CR structure is contact, the congruence is said to be \emph{twisting}. These congruences play a fundamental part in mathematical relativity, notably in the discovery of exact solutions to the vacuum Einstein field equations, such as the Kerr black hole \cite{Kerr1963}. Aspects of Penrose's twistor theory were highly influenced by these ideas \cite{Penrose1967,Mason1998}. At the same time, Robinson and Trautman \cite{Robinson1986} initiated a fruitful programme of research in that direction \cite{Tafel1985,Lewandowski1990,Lewandowski1991a,Hill2008}.

In dimensions greater than four, there is a priori no connection between non-shearing congruences and nearly Robinson structures. The present author, however, recently \cite{TaghaviChabert2022} gave a mild curvature condition on the Weyl tensor for a twisting non-shearing congruence of null geodesics to give rise to a nearly Robinson structure. In the same reference, all local Einstein metrics on Lorentzian conformal manifolds of even dimension admitting a twisting non-shearing congruence of null geodesics were found to be contained in a three-parameter family of metrics, which include the \emph{Taub--NUT--(A)dS metric} and the \emph{Fefferman--Einstein metric}. In this case, the leaf space of the congruence is characterised by the existence of a \emph{CR--Einstein structure}, which corroborates previous results in \cite{Leitner2007} for CR geometries. Reference \cite{Cap2008} by \v{C}ap and Gover sheds more light on these aspects, especially in connection with the r\^{o}le played by so-called \emph{almost Einstein scales} and their zero sets, which can be suitably identified with cross-sections of the Fefferman bundle.

Finally, there is an obvious analogy when it comes to the usefulness of Cauchy--Riemann methods in solving the Einstein equations on a Lorentzian manifold: if one wishes to find Einstein metrics on a \emph{Riemannian} manifold, then more powerful resources become available once we assume that our manifold is \emph{K\"{a}hler} --- see for instance \cite{Yau2000} for a review of these ideas.

The overall aim of this paper is to recast some classes of nearly Robinson manifolds as Fefferman spaces of some kind. More precisely, we set ourselves the following goals:
\begin{enumerate}
	\item To provide a formulation of CR--Einstein structures for non-involutive almost CR manifolds by means of the Webster calculus. These structures underlie a certain class of Einstein Lorentzian metrics to be considered here. Our investigation will also be illustrated by an example of a five-dimensional strictly almost CR manifold arising from a strictly almost K\"{a}hler--Einstein manifold.
	\item To develop Leitner's generalisation of Fefferman's construction for almost CR manifolds by using the squared norm of the Nijenhuis torsion tensor. This will lead to a one-parameter family of such conformal structures, and we will provide geometric characterisations for them, which generalise Sparling's characterisation of Fefferman spaces for CR structures \cite{Graham1987,Cap2008}.
	\item To introduce a new class of Lorentzian conformal structures arising from almost CR manifolds, as exact `perturbations' of Fefferman conformal spaces: a `perturbed' Fefferman metric takes the form $\wt{g}_{\theta,\wt{\xi}} = \wt{g}_{\theta} + 4 \theta \odot \wt{\xi}$, where $\wt{g}_{\theta}$ is a Fefferman metric for some pseudo-Hermitian form $\theta$, and $\wt{\xi}$ is a semi-basic one-form.
	\item To seek almost Einstein scales for such perturbed Fefferman spaces in dimensions greater than four, and investigate the geometric properties of their zero sets.
\end{enumerate} 

Let us re-emphasise that the focus of this article will be on conformal geometries of even dimensions \emph{greater than four}, although some results may still apply in dimension four. It turns out that the four-dimensional story is somewhat richer and more technical, and for this reason, it is treated separately in \cite{TaghaviChabert2023}.

The plan of the article is as follows. We start in Section \ref{sec:CR_geom} with a review of almost CR geometry, focussing in particular on so-called \emph{CR--Einstein} structures. In Theorem \ref{thm:alCRE}, these pseudo-Hermitian analogues of Einstein metrics are shown to be equivalent to solutions to a system of invariant differential equations on a CR scale, that is, a nowhere vanishing real-valued density corresponding to a pseudo-Hermitian structure. In Theorem \ref{thm:spalCRE}, a variation of these equations, now involving a \emph{complex-valued} density, is proposed, and a solution thereof defines a CR--Einstein scale, which turns out to be a constant multiple of the squared norm of the Nijenhuis tensor --- Proposition \ref{prop:spalCRE-N2}.

Section \ref{sec:conf_geom} contains a review of conformal geometry and so-called optical and almost Robinson geometries which are at the heart of non-shearing congruences of null geodesics. In Definitions \ref{defn:redEins}, \ref{defn:purad}, \ref{defn:aESc} and \ref{defn:puradSc} we introduce formal definitions of certain metrics with prescribed Ricci curvature, which we term \emph{weakly half-Einstein}, \emph{half-Einstein} and \emph{pure radiation} metrics, and their \emph{almost Lorentzian scale} analogues as generalisations of almost Einstein scales --- Propositions \ref{prop:scale2metric} and \ref{prop:scale2metricB}.

In Section \ref{sec:Fefferman}, we construct a one-parameter family of Lorentzian conformal structures from a strictly almost CR manifold, by modifying the original Fefferman construction --- see Definitions \ref{defn-alpha-Fefferman} and \ref{defn-alpha-Fefferman-space}. In Proposition \ref{prop:non_tw-sp} we show that these Fefferman spaces admit distinguished spinor fields satisfying a conformally invariant differential equation that reduces to the twistor equation whenever the underlying almost CR structure is involutive. We then state the characterisation of these conformal structures that arise from such Fefferman spaces in Theorem \ref{thm:Fefferman-CRA}.

Section \ref{sec:perturbed_Fefferman} begins with the formal Definition \ref{defn:perturbed_Fefferman} of perturbations of Fefferman spaces. We then provide a technical Lemma \ref{lem:NSh2Feff}, which allows to formalise the relation between perturbed Fefferman spaces and geometries endowed with twisting non-shearing congruences of null geodesics. Proposition \ref{prop:PetrovIII} provides sufficient conditions for an optical geometry to be locally conformally isometric to a perturbed Fefferman space. The section ends with Conjecture \ref{conj}, which puts forward another characterisation of certain classes of perturbed Fefferman spaces.

The focus of Section \ref{sec:al_Lor_sc} is to examine the consequences of the existence of almost Lorentzian scales on optical geometries with twisting non-shearing congruences of null geodesics. Among others, Proposition \ref{prop:density_Ric_NS} provides a geometric description of the zero set of some classes of such scales, while Propositions \ref{prop:alwkhalfE} and \ref{prop:al_half-Einstein} provide characterisations of almost (weakly) half-Einstein scales in terms of their underlying CR geometries. In Corollary \ref{cor:alt_char} we give sufficient conditions for an optical geometry to be locally conformally isometric to a perturbed Fefferman space. The CR data of perturbed Fefferman spaces admitting almost (half-)Einstein scales is described in Theorems \ref{thm:hlfEinsc} and \ref{thm:pertFeffpurad}.

Section \ref{sec:asymptotics} looks into the properties of the zero set of an almost half-Einstein scale on a perturbed Fefferman space, more particularly, Proposition \ref{prop:redE_hypsrf} in relation to its causal property, and Theorem \ref{thm:conformal_flat_asym} in relation to conformal flatness.

We comment briefly, in Section \ref{sec:sym}, on conformal symmetries and their interpretation in terms of CR invariant differential equations.

The article ends with two appendices: In Appendix \ref{app:NP} we give an example of a five-dimensional strictly almost CR manifold endowed with a CR--Einstein structure, and Appendix \ref{app:proofchi} contains the lengthy proof of Theorem \ref{thm:Fefferman-CRA}.

\noindent\textbf{Acknowledgments:} The research leading to these results has received funding from the Norwegian Financial Mechanism 2014-2021 UMO-2020/37/K/ST1/02788.

	\section{Almost CR geometry}\label{sec:CR_geom}
We first recall some background on almost CR geometry following \cite{Tanaka1975,Webster1978,Gover2005,Cap2008,Cap2010,Matsumoto2016,Case2020,Matsumoto2022} before delving into more recent notions in Section \ref{sec:gauged_Webster}, regarding gauged Webster connections, and in Section \ref{sec:CRE} regarding CR--Einstein structures for almost CR manifolds.
	\subsection{Preliminaries}
	An \emph{almost Cauchy--Riemann (CR) structure} on a smooth manifold $\mc{M}$ of dimension $2m+1$, for $m \geq 1$, consists of a pair $(H,J)$ where $H$ is a rank-$2m$ distribution and ${J}$ a bundle complex structure on ${H}$, i.e.\ ${J} \circ {J} = - {\Id}$, where ${\Id}$ is the identity map on ${H}$. The complexification ${}^\C {H}$ of ${H}$ thus splits as ${}^\C {H} = {H}^{(1,0)} \oplus {H}^{(0,1)}$, where ${H}^{(1,0)}$ and ${H}^{(0,1)}$ are the $+\i$-eigenbundle and $-\i$-eigenbundle of ${J}$ respectively, each having complex rank $m$. Throughout this article, the following additional assumptions will be made:
	\begin{itemize}
		\item $(H,J)$ is \emph{partially integrable}, that is, $[H^{(1,0)},H^{(1,0)}] \subset {}^\C H$;
		\item $H$ is \emph{contact}, that is,\ for any nowhere vanishing $\theta \in \Gamma(\Ann(H))$, the $(2m+1)$-form $\theta \wedge (\d \theta)^m$ is nowhere vanishing.
	\end{itemize}
We shall refer to the triple $(\mc{M},H,J)$ thus defined as an \emph{almost CR manifold}. When $H^{(1,0)}$ is involutive, $(H,J)$ is said to be a \emph{CR structure} and $(\mc{M},H,J)$ a \emph{CR manifold}. In the case where $H^{(1,0)}$ is not involutive, we shall refer to $(H,J)$ as a \emph{strictly almost CR structure} and $(\mc{M},H,J)$ as a \emph{strictly almost CR manifold}. In dimension three, $m=1$, an almost CR manifold is necessarily a CR manifold.

The bundle complex structure $J$ endows $H$ with an orientation. We shall assume that $\mc{M}$ is oriented so that we can fix an orientation on $\Ann(H)$. A choice of contact form $\theta \in \Gamma(\Ann(H))$, positive with respect to this orientation, will be referred to as a \emph{pseudo-Hermitian structure} for $(H,J)$. The restriction of $\d \theta$ to $H$ is non-degenerate, and is the imaginary part of a Hermitian bilinear form called the \emph{Levi form associated to $\theta$}. Its signature $(p,q)$, where $p+q=m$, is an invariant of $(H,J)$.

	We shall also assume that the canonical bundle ${\mc{C}} := \bigwedge^{m+1} \Ann({H}^{(0,1)})$ of $(H,J)$ admits a $-(m+2)$-nd root, which we shall denote ${\mc{E}}(1,0)$. More generally, we define \emph{density bundles} ${\mc{E}}(w,w') := {\mc{E}}(1,0)^w \otimes \overline{{\mc{E}}(1,0)}{}^{w'}$ for any $w, w' \in \R$ such that $w - w' \in \Z$. We note that $\overline{{\mc{E}}(w,w')} = {\mc{E}}(w',w)$. In particular, $\mc{C} = \mc{E}(-m-2,0)$. 
	
	There are nowhere vanishing canonical sections ${\bm{\theta}}$ of $T^* {\mc{M}} \otimes {\mc{E}}(1,1)$ and ${\bm{h}}$ of $({H}^{(1,0)})^* \otimes ({H}^{(0,1)})^* \otimes {\mc{E}}(1,1)$ with the property that for each real-valued $0<s \in \Gamma({\mc{E}}(1,1))$, ${\theta} = s^{-1} {\bm{\theta}}$ is a pseudo-Hermitian structure with Levi form ${h} = s^{-1}  {\bm{h}}$. We shall also call the weighted form $\bm{h}$ the \emph{Levi form} of $(H,J)$, and we refer to $s$ as a \emph{CR scale}. We shall also identify ${H}^{(1,0)}$ with $({H}^{(0,1)})^*\otimes {\mc{E}}(1,1)$ and ${H}^{(0,1)}$ with $({H}^{(1,0)})^*\otimes {\mc{E}}(1,1)$ by means of $\bm{h}$ and its inverse.
	
	A density $\sigma$ of weight $(1,0)$ determines a unique CR scale $s=\sigma \ol{\sigma}$, which in turn gives rise to a pseudo-Hermitian structure ${\theta} = s^{-1} {\bm{\theta}}$. Conversely, at any point, a CR scale determines a circle of densities of weight $(1,0)$. If $\sigma$ is such a density, then the section $\zeta := \sigma^{-(m+2)}$ of the canonical bundle satisfies
	\begin{align}\label{eq:volnorm}
				\theta \wedge (\d \theta)^m & = \i^{m^2} m! (-1)^q \theta \wedge (\ell \hook \zeta) \wedge (\ell \hook \ol{\zeta}) \, .
	\end{align}
	and is said to be \emph{volume-normalised}.	
	
	We introduce plain and barred  minuscule Greek abstract indices in the following way:
	\begin{align*}
		{\mc{E}}^{\alpha} &:= {H}^{(1,0)} \, , & {\mc{E}}^{\bar{\alpha}} & := {H}^{(0,1)}\, , & {\mc{E}}_0 & := \Ann(H) \, , \\
		{\mc{E}}_{\alpha} & := ({H}^{(1,0)})^* \, , & {\mc{E}}_{\bar{\alpha}} & := ({H}^{(0,1)})^* \, ,
		& {\mc{E}}^0 & := \left(\Ann(H)\right)^* \, .
	\end{align*}
	In effect, indices can be raised and lowered using ${\bm{h}}_{\alpha \bar{\beta}}$, e.g.\ ${v}_{\alpha} = {\bm{h}}_{\alpha \bar{\beta}} {v}^{\bar{\beta}}$. Clearly, complex conjugation on ${}^\C {H}$ changes the index type, so we shall write ${v}^{\bar{\alpha}}$ for $\overline{{v}^{\alpha}}$, and so on.

	Symmetrisation will be denoted by round brackets, and skew-symmetrisation by square brackets, e.g.\ $\lambda_{(\alpha \beta)} = \frac{1}{2} \left( \lambda_{\alpha \beta} + \lambda_{\beta \alpha} \right)$ and $\lambda_{[\alpha \beta]} = \frac{1}{2} \left( \lambda_{\alpha \beta} - \lambda_{\beta \alpha} \right)$. The trace-free part of tensors of mixed types with respect to the Levi form will be denoted by a ring, e.g. $(\lambda_{\alpha \bar{\beta}})_\circ = \lambda_{\alpha \bar{\beta}} - \frac{1}{m} \bm{h}_{\alpha \bar{\beta}} \lambda_{\gamma \bar{\delta}} \bm{h}^{\gamma \bar{\delta}}$.

	\subsection{Webster connections}
	For each choice of pseudo-Hermitian structure ${\theta}$, there is a unique vector field ${\ell}$, known as the \emph{Reeb vector field}, which satisfies ${\theta} ({\ell}) = 1$ and $\d {\theta} ( {\ell} , \cdot ) = 0$. In particular, it induces a splitting
	\begin{align*}
		{}^\C T {\mc{M}} =  {}^\C \langle \ell \rangle \oplus {H}^{(1,0)} \oplus {H}^{(0,1)} \, .
	\end{align*}
	Any choice of frame $(e_{\alpha})_{\alpha=1,\ldots,m}$ of $H^{(1,0)}$ can thus be completed to an \emph{adapted frame} $\left( \ell , {e}_\alpha , \overline{{e}}_{\bar{\beta}} \right)$, $\alpha, \bar{\beta} = 1,\ldots,m$ for ${}^\C T \mc{M}$. Similarly, we call its dual $\left( {\theta} , {\theta}^\alpha , \overline{{\theta}}{}^{\bar{\alpha}} \right)$ an \emph{adapted coframe} for ${}^\C T^* \mc{M}$. The complex-valued forms $(\theta^{\alpha})$ are then said to be \emph{admissible} for $\theta$.
	
	The structure equations obtained from this adapted coframe can then be written as
	\begin{subequations}\label{eq:structure_CR}
	\begin{align}
	\d {\theta} & = \i {h}_{\alpha \bar{\beta}} {\theta}^{\alpha} \wedge \overline{{\theta}}{}^{\bar{\beta}}  \, , \\
	\d {\theta}^{\alpha} & = {\theta}^{\beta} \wedge {\Gamma}_{\beta}{}^{\alpha} +  {\WbA}^{\alpha}{}_{\bar{\beta}}   {\theta} \wedge \overline{{\theta}}{}^{\bar{\beta}} - \tfrac{1}{2} {\Nh}_{\bar{\beta} \bar{\gamma}}{}^{\alpha}\overline{{\theta}}{}^{\bar{\beta}} \wedge \overline{{\theta}}{}^{\bar{\gamma}} \, , \label{eq:structure_CR_A}\\
	\d \overline{{\theta}}{}^{\bar{\alpha}} & = \overline{{\theta}}{}^{\bar{\beta}} \wedge {\Gamma}_{\bar{\beta}}{}^{\bar{\alpha}} + {\WbA}^{\bar{\alpha}}{}_{\beta}  {\theta} \wedge {\theta}^{\beta} - \tfrac{1}{2} {\Nh}_{\beta \gamma}{}^{\bar{\alpha}} \theta^{\beta} \wedge \theta^{\gamma} \, ,
	\end{align}
\end{subequations}
	where
	\begin{itemize}
		\item ${h}_{\alpha \bar{\beta}}$ is the Levi form of ${\theta}$, viewed as Hermitian matrix --- abstractly, this is $h_{\alpha \bar{\beta}} = s^{-1} \bm{h}_{\alpha \bar{\beta}}$ where $s$ is the CR scale of $\theta$;
	 	\item ${\Gamma}_{\beta}{}^{\alpha}$ is the connection $1$-form of the \emph{Webster connection} ${\nabla}$ on $T {\mc{M}}$ that preserves $\theta$ and $\d \theta$, i.e.\
	 	\begin{align*}
	 	{\nabla} {\theta} & = 0 \, , & {\nabla} {\theta}^\alpha & = - {\Gamma}_{\beta}{}^{\alpha} \otimes {\theta}^{\beta} \, ,  & {\nabla} \ol{{\theta}}{}^{\bar{\alpha}} & = - \Gamma_{\bar{\beta}}{}^{\bar{\alpha}} \otimes \ol{\theta}{}^{\bar{\beta}} \, ,
	 	\end{align*}
	 	with $\d {h}_{\alpha \bar{\beta}} - \Gamma_{\alpha \bar{\beta}} - \Gamma_{\bar{\beta} \alpha} = 0$;
	 	\item ${\WbA}_{\alpha \beta}$ is the \emph{Webster torsion tensor} of $\nabla$, which satisfies ${\WbA}_{\alpha \beta} = {\WbA}_{(\alpha \beta)}$;
	 	\item and ${\Nh}_{\alpha \beta \gamma}$ is the \emph{Nijenhuis torsion tensor} of $\nabla$, which satisfies ${\Nh}_{\alpha \beta \gamma} = {\Nh}_{[\alpha \beta] \gamma}$, ${\Nh}_{[\alpha \beta \gamma]} = 0$.
	 \end{itemize}
The Webster connection, also known as \emph{Webster--Tanaka} or \emph{Webster--Stanton} connection, is uniquely determined by its compatibility with the contact form and its Levi form, and the prescription of symmetry on its torsion.

Clearly, the Nijenhuis tensor ${\Nh}_{\alpha \beta \gamma}$ vanishes identically if and only if $H^{(1,0)}$ is involutive. This is a CR invariant condition. For convenience, we set
\begin{align*}
	\| \Nh \|^2 & := \Nh_{\alpha \beta \gamma} \Nh^{\alpha \beta \gamma} \, .
\end{align*}
This is a density of weight $(-1,-1)$. If the Levi form has definite signature, i.e.\ $pq=0$, then its vanishing is equivalent to the vanishing of $\Nh_{\alpha \beta \gamma}$, and thus to the involutivity of $(H,J)$.

On the other hand, the Reeb vector field $\ell$ of $(H,J,\theta)$ is a transverse infinitesimal symmetry of $({H},{J})$, i.e.\ $\mathsterling_\ell v \in \Gamma(H^{(1,0)})$ for all $v \in \Gamma(H^{(1,0)})$, if and only if ${\WbA}_{\alpha \beta}$ vanishes identically.

The curvature tensors of $\nabla$ are given by, for any section $v^\alpha$ of ${H}^{(1,0)}$, 
\begin{align*}
	({\nabla}_{\alpha} {\nabla}_{\bar{\beta}} - {\nabla}_{\bar{\beta}} {\nabla}_{\alpha} ) v^{\gamma} + \i \bm{{h}}_{\alpha \bar{\beta}} v^\gamma & =: {\Riem}_{\alpha \bar{\beta} \delta}{}^{\gamma} v^{\delta}  \, , \\
	({\nabla}_{\alpha} {\nabla}_{0} - {\nabla}_{0} {\nabla}_{\alpha} ) v^{\gamma} - {\WbA}_{\alpha}{}^{\bar{\beta}} \nabla_{\bar{\beta}} v^\gamma & =: {\Riem}_{\alpha 0 \delta}{}^{\gamma} v^{\delta}  \, , \\
	({\nabla}_{\alpha} {\nabla}_{\beta} - {\nabla}_{\beta} {\nabla}_{\alpha} ) v^{\gamma} - {\Nh}_{\alpha \beta}{}^{\bar{\delta}} \nabla_{\bar{\delta}} v^\gamma & =: {\Riem}_{\alpha \beta \delta}{}^{\gamma} v^{\delta}  \, ,
\end{align*}
together with their complex conjugates. The Bianchi identities allows us to express  ${\Riem}_{\alpha 0 \delta}{}^{\gamma}$ and ${\Riem}_{\alpha \beta \delta}{}^{\gamma}$ in terms of the torsion and its covariant derivatives as given in \cite{TaghaviChabert2022}.
The more important piece ${\Riem}_{\alpha \bar{\beta} \gamma}{}^{\delta}$ can be used to define, for $m>1$, the \emph{Chern--Moser tensor}
\begin{align*}
	{\CM}_{\alpha}{}^{\gamma}{}_{\beta}{}^{\delta} & := \left( {\Riem}_{(\alpha}{}^{(\gamma}{}_{\beta)}{}^{\delta)} \right)_\circ \, ,
\end{align*}
and, for $m \geq 1$, the \emph{Webster--Ricci tensor} ${\Ric}{}_{\gamma}{}^{\delta} := {\bm{h}}{}^{\alpha \bar{\beta}} {\Riem}{}_{\alpha \bar{\beta} \gamma}{}^{\delta}$ and  the \emph{Webster--Ricci scalar} ${\Sc} := {\Ric}{}_{\gamma}{}^{\gamma}$. We conveniently introduce the \emph{Webster--Schouten tensor} and the \emph{Webster--Schouten scalar}
\begin{align*}
	{\Rho}_{\alpha \bar{\beta}} & := \tfrac{1}{m+2}\left( {\Ric}_{\alpha \bar{\beta}} - \tfrac{1}{2m+2} {\Sc} \, {\bm{h}}_{\alpha \bar{\beta}} \right) \, , & {\Rho} :=  {\Rho}_{\alpha \bar{\beta}} {\bm{h}}^{\alpha \bar{\beta}} \, ,
\end{align*}
respectively. It then follows that $\Rho = \tfrac{1}{2(m+1)} \Sc$.

The Chern--Moser tensor is a CR invariant. In dimension greater than three, the vanishing of both ${\Nh}_{\alpha \beta \gamma}$ and ${\CM}_{\alpha \bar{\gamma} \beta \bar{\delta}}$ is equivalent to the almost CR structure being locally \emph{CR flat} --- if $pq=0$, this means that $({\mc{M}},H,J)$ is locally CR equivalent to the CR $(2m+1)$-sphere. In dimension three, the corresponding invariant is the fourth-order \emph{Cartan tensor}, but will play no r\^{o}le in this article.

The Webster connection extends to a connection on the density bundles $\mc{E}(w,w')$. If $\sigma \in \Gamma(\mc{E}(w,0))$ is such that $\sigma^{-\frac{m+2}{w}} = \theta \wedge \theta^1 \wedge \ldots \wedge \theta^m$ for some adapted coframe $(\theta,\theta^\alpha, \ol{\theta}{}^{\bar{\alpha}})$. Then
\begin{align}\label{eq:conn1_density}
	\nabla \sigma & = \tfrac{w}{m+2} \Gamma_{\alpha}{}^{\alpha} \sigma \, ,
\end{align}
where $\Gamma_{\alpha}{}^{\beta}$ is the connection one-form of $\nabla$ corresponding to $\theta$. By complex conjugation, one obtains a similar formula for densities of weight $(0,w)$.

Note that $\nabla$ preserves the weighted contact form $\bm{\theta}$, i.e.\ $\nabla \bm{\theta} = 0$. Thus if $\theta = s^{-1} \bm{\theta}$ for some CR scale $s$, $\nabla$ also preserves $s$. This implies that for any $\sigma \in \Gamma(\mc{E}(1,0))$ such that $s = \sigma \ol{\sigma}$, we must have $\sigma^{-1}\nabla \sigma = - \ol{\sigma}^{-1}\nabla \ol{\sigma}$.

For any smooth density $f \in \Gamma(\mc{E}(w,w'))$, the commutation relations are given by
	\begin{subequations}\label{eq:CR_com}
		\begin{align}
			({\nabla}_{\alpha} {\nabla}_{\bar{\beta}} - {\nabla}_{\bar{\beta}} {\nabla}_{\alpha} ) {f} & = \tfrac{w -w'}{m+2} \left( \Ric_{\alpha \bar{\beta}} - \Nh_{\gamma \delta \alpha} \Nh^{ \gamma \delta}{}_{\bar{\beta}} + \Nh_{\alpha \gamma \delta} {\Nh}_{\bar{\beta}}{}^{\gamma \delta} \right) f - \i \bm{{h}}_{\alpha \bar{\beta}} {\nabla}_{0} {f} \, , \label{eq:CR_com1} \\
			({\nabla}_{\alpha} {\nabla}_{0} - {\nabla}_{0} {\nabla}_{\alpha} ) {f} & = \tfrac{w -w'}{m+2} \left( {\nabla}^{\beta}{\WbA}_{\alpha \beta} -  {\WbA}^{\beta \gamma} {\Nh}_{\alpha \beta \gamma} \right) f + {\WbA}_{\alpha}{}^{\bar{\beta}} 
			{\nabla}_{\bar{\beta}} {f} \, , \label{eq:CR_com2} \\
			({\nabla}_{\alpha} {\nabla}_{\beta} - {\nabla}_{\beta} {\nabla}_{\alpha} ) {f} & = \tfrac{w -w'}{m+2} {\nabla}^{\gamma} {\Nh}_{\alpha \beta \gamma} f + {\Nh}_{\alpha \beta}{}^{\bar{\gamma}} {\nabla}_{\bar{\gamma}} {f} \, . \label{eq:CR_com3}
		\end{align}
	\end{subequations}
This follows from the trace of the curvature $2$-form $\Omega_{\gamma}{}^{\gamma} = \d \Gamma_{\gamma}{}^{\gamma}$, which we find to be
\begin{multline}\label{eq:density_curv2form}
\Omega_{\gamma}{}^{\gamma}
= \left( \Ric_{\alpha \bar{\beta}} - \Nh_{\gamma \delta \alpha} \Nh^{ \gamma \delta}{}_{\bar{\beta}} + \Nh_{\alpha \gamma \delta} {\Nh}_{\bar{\beta}}{}^{\gamma \delta} \right) {\theta}^{\alpha} \wedge \overline{{\theta}}{}^{\bar{\beta}} \\
+ \tfrac{1}{2} {\nabla}^{\alpha} {\Nh}_{\gamma \delta \alpha} {\theta}^{\gamma} \wedge {\theta}^{\delta} - \tfrac{1}{2} {\nabla}^{\bar{\alpha}} {\Nh}_{\bar{\gamma} \bar{\delta} \bar{\alpha}} \overline{{\theta}}{}^{\bar{\gamma}} \wedge \overline{{\theta}}{}^{\bar{\delta}} \\
+ \i \left( (m+2) \mathsf{T}_{\gamma} - \nabla_{\gamma} \Rho \right) {\theta}^{\gamma} \wedge \theta - \i \left( (m+2) \mathsf{T}_{\bar{\gamma}} - \nabla_{\bar{\gamma}} \Rho \right) \overline{{\theta}}{}^{\bar{\gamma}} \wedge \theta \, ,
\end{multline}
where
\begin{align}\label{eq:CRT}
	\mathsf{T}_{\alpha} & = \tfrac{1}{m+2} \left( {\nabla}_{\alpha} {\Rho} - \i {\nabla}^{\gamma} {\WbA}_{\gamma \alpha} + \i {\WbA}^{\beta \gamma} {\Nh}_{\alpha \beta \gamma} \right) \, .
\end{align}

\subsection{Transformation rules}
Under a change of contact form, the Webster connection is subject to transformation laws. These can be found in e.g.\ \cite{Matsumoto2022}, and do not differ from the involutive case, see e.g.\ \cite{Gover2005}. It will suffice to state the transformation rules of the Webster torsion, Webster--Schouten tensor and Webster--Schouten scalar as
\begin{align}
		\wh{{\WbA}}_{\alpha \beta} & = {\WbA}_{\alpha \beta} + \i {\nabla}_{(\alpha} {\Upsilon}_{\beta)}  - \i {\Upsilon}_{\alpha} {\Upsilon}_{\beta} +  \i \Upsilon^{\gamma} \Nh_{\gamma (\alpha \beta)} \, , \label{eq:trsfA} \\
\wh{\Rho}_{\alpha \bar{\beta}} & = \Rho_{\alpha \bar{\beta}} - \tfrac{1}{2} \left( \nabla_{\alpha} \Upsilon_{\bar{\beta}} + \nabla_{\bar{\beta}} \Upsilon_{\alpha} \right) - \tfrac{1}{2} \Upsilon^{\gamma} \Upsilon_{\gamma} \bm{h}_{\alpha \bar{\beta}} \, , \label{eq:trsfPt} \\
\wh{{\Rho}} & = {\Rho} - \tfrac{1}{2} \left( {\nabla}^{\alpha} {\Upsilon}_{\alpha} + {\nabla}_{\alpha} {\Upsilon}^{\alpha} \right) - \tfrac{m}{2} {\Upsilon}^{\gamma} {\Upsilon}_{\gamma} \, , \label{eq:trsfP}
\end{align}
respectively. For future use, we also compute
\begin{align*}
	\wh{\nabla}_{\bar{\delta}} \Nh_{\gamma \alpha \beta} & = \nabla_{\bar{\delta}} \Nh_{\gamma \alpha \beta} + \Upsilon_{\bar{\delta}} \Nh_{\gamma \alpha \beta} + \bm{h}_{\gamma \bar{\delta}} \Nh_{\varepsilon \alpha \beta} \Upsilon^{\varepsilon} + \bm{h}_{\alpha \bar{\delta}} \Nh_{\gamma \varepsilon \beta} \Upsilon^{\varepsilon} + \bm{h}_{\beta \bar{\delta}} \Nh_{\gamma \alpha \varepsilon} \Upsilon^{\varepsilon} \, ,
\end{align*}
and hence
\begin{align}
	\wh{\nabla}^{\gamma} \Nh_{\gamma [\alpha \beta]} 
	& = \nabla^{\gamma} \Nh_{\gamma [\alpha \beta]} + (m + 2 )\Upsilon^{\gamma} \Nh_{\gamma [\alpha \beta]}   \, , &
	\wh{\nabla}^{\gamma} \Nh_{\gamma (\alpha \beta)} 
	& = \nabla^{\gamma} \Nh_{\gamma (\alpha \beta)} + m \Upsilon^{\gamma} \Nh_{\gamma (\alpha \beta)}  \, . \label{eq:trsfDNh}
\end{align}

	\subsection{Gauged Webster connections}\label{sec:gauged_Webster}
		In what follows, $\mc{E}_{\bullet}$ will denote any multiple tensor products of $\mc{E}_{\alpha}$ and $\mc{E}_{\bar{\alpha}}$. Choose a Webster connection $\nabla$. Then, for any choice of one-form $\xi$ on $\mc{M}$, we define the connection
	\begin{align}\label{eq:gconn}
		\accentset{\xi}{\nabla} & : \Gamma(\mc{E}_{\bullet}(w,w')) \longrightarrow \Gamma( T^* \mc{M} \otimes \mc{E}_{\bullet}(w,w')) \quad : \bm{f} \mapsto  \accentset{\xi}{\nabla} \bm{f} := \nabla \bm{f} - \i (w - w') \xi \otimes \bm{f} \, .
	\end{align}
	Clearly, $\accentset{\xi}{\nabla}$ coincides with $\nabla$ on any section of $\mc{E}_{\bullet}(w,w)$. We shall often use the short-hand $\accentset{\xi}{\nabla} = \nabla + \xi$ to mean \eqref{eq:gconn}.
	
	\begin{defn}\label{defn:GPconn}
		We shall call the connection defined by \eqref{eq:gconn} the \emph{Webster connection gauged by $\xi$}. If $\xi = \i \nabla \log z$ for some complex-valued function $z$, we refer to the gauge $\xi$ as being \emph{exact}.
	\end{defn}
	
	If $\xi = \xi_{\alpha} \theta^{\alpha} + \xi_{\bar{\alpha}} \ol{\theta}{}^{\bar{\alpha}} + \xi_{0} \theta$
	for some adapted coframe $(\theta, \theta^{\alpha}, \ol{\theta}{}^{\bar{\alpha}})$, then
		\begin{align*}
		F & := \d \xi = F_{\alpha \beta}  \theta^{\alpha} \wedge \theta^{\beta} + F_{\bar{\alpha} \bar{\beta}} \ol{\theta}{}^{\bar{\alpha}} \wedge \ol{\theta}{}^{\bar{\beta}} + 2 F_{\alpha \bar{\beta}} \theta^{\alpha} \wedge \ol{\theta}{}^{\bar{\beta}} 
		+ 2 F_{0 \beta} \theta \wedge \theta^{\beta} 
		+ 2 F_{0 \bar{\beta}} \theta \wedge \ol{\theta}{}^{\bar{\beta}} \, ,
	\end{align*}
	where, using the structure equations \eqref{eq:structure_CR},
	\begin{align*}
		F_{\alpha \beta} & = \nabla_{[\alpha} \xi_{\beta]} - \tfrac{1}{2} \Nh_{\alpha \beta}{}^{\bar{\gamma}} \xi_{\bar{\gamma}}\, , & F_{\alpha \bar{\beta}} & = \tfrac{1}{2} \left( \nabla_{\alpha} \xi_{\bar{\beta}} - \nabla_{\bar{\beta}} \xi_{\alpha} + \i \xi_{0} h_{\alpha \bar{\beta}} \right) \, , \\
		F_{0 \beta} & = \tfrac{1}{2} \left( \nabla_{0} \xi_{\beta} - \nabla_{\beta} \xi_{0} + \xi_{\bar{\alpha}} {\WbA}^{\bar{\alpha}}{}_{\beta} \right) \, .
	\end{align*}
	
	It immediately follows that:
	\begin{lem}\label{lem:closedg}
		Let $\xi$ be a one-form on $(\mc{M},H,J)$ so that $\xi = \xi_{\alpha} \theta^{\alpha} + \xi_{\bar{\alpha}} \ol{\theta}{}^{\bar{\alpha}} + \xi_{0} \theta$
		for some adapted coframe $(\theta, \theta^{\alpha}, \ol{\theta}{}^{\bar{\alpha}})$. A necessary and sufficient condition for $\xi$ to be closed is that
\begin{align*}
	\nabla_{[\alpha} \xi_{\beta]} - \tfrac{1}{2} \Nh_{\alpha \beta}{}^{\bar{\gamma}} \xi_{\bar{\gamma}} & = 0 \, , &  \nabla_{\alpha} \xi_{\bar{\beta}} - \nabla_{\bar{\beta}} \xi_{\alpha} + \i \xi_{0} h_{\alpha \bar{\beta}} & = 0 \, , 
	& \nabla_{0} \xi_{\beta} - \nabla_{\beta} \xi_{0} + \xi_{\bar{\alpha}} {\WbA}^{\bar{\alpha}}{}_{\beta} & = 0 \, .
\end{align*}		
	\end{lem}

	\subsection{CR--Einstein structures}\label{sec:CRE}
A \emph{CR--Einstein structure}\footnote{Such a structure was introduced by the author in \cite{TaghaviChabert2022} under the name \emph{almost CR--Einstein structure}. However, the use of the word `almost' clashes with that in the notion of \emph{almost CR--Einstein structure} given in \cite{Cap2008}. For this reason, it was deemed appropriate to drop the adverb `almost' to avoid confusion. That $(H,J)$ is not necessarily involutive should be clear from the context.} on an almost CR manifold $({\mc{M}},{H},{J})$ is a pseudo-Hermitian structure, whose Webster torsion tensor ${\WbA}_{\alpha \beta}$, Webster--Schouten tensor ${\Rho}_{\alpha \bar{\beta}}$ and Nijenhuis tensor ${\Nh}_{\alpha \beta \gamma}$ satisfy \cite{TaghaviChabert2022}
\begin{subequations}\label{eq:CR-E}
\begin{align}
		& \WbA_{\alpha \beta} = 0 \, , \label{eq:CR-E:A} \\
		& \nabla^{\gamma} \Nh_{\gamma (\alpha \beta)} = 0 \, , \label{eq:CR-E:DN} \\
		& \left( \Rho_{\alpha \bar{\beta}} - \tfrac{1}{m+2} \Nh_{\alpha \gamma \delta} \Nh_{\bar{\beta}}{}^{\gamma \delta} \right)_\circ = 0 \, . \label{eq:CR-E:Rho}
	\end{align}
\end{subequations}
In the same reference, it is proved  that condition \eqref{eq:CR-E:Rho} can be replaced by 
\begin{align*}
	\Ric_{\alpha \bar{\beta}} - \Nh_{\alpha \gamma \delta} \Nh_{\bar{\beta}}{}^{\gamma \delta} & = \Lambda h_{\alpha \bar{\beta}} \, , & \mbox{for some constant $\Lambda$.}
\end{align*}
In particular,
\begin{align}\label{eq:LScN2}
	\Lambda & = \tfrac{1}{m} \left( \Sc - \| \Nh \|^2 \right) \, .
\end{align}
Here, we view, with a slight abuse of notation, $\Sc$ and $\| \Nh \|^2$ as unweighted functions.

When $(H,J)$ is involutive, equations \eqref{eq:CR-E} reduce to
	\begin{align}\label{eq:CR-Einstein}
		& \WbA_{\alpha \beta} = 0 \, ,
		& \left( \Rho_{\alpha \bar{\beta}} \right)_\circ = 0 \, ,
	\end{align}
which were first investigated in \cite{Leitner2007}, where the corresponding pseudo-Hermitian structure is then referred to as a \emph{transversally symmetric pseudo-Einstein structure}. These were later revisited in tractor language in \cite{Cap2008} under the name CR--Einstein structure. There, building on a result by \cite{Lee1988}, the authors showed that the CR--Einstein equations \eqref{eq:CR-Einstein} are equivalent to the existence of a nowhere vanishing density $\sigma$ of weight $(1,0)$ satisfying the system of CR-invariant differential equations
\begin{subequations}\label{eq:A-CR-EI}
	\begin{align}
		\nabla_{\alpha} \nabla_{\beta} \sigma + \i \WbA_{\alpha \beta} & = 0 \, , \label{eq:soliton} \\
				\nabla_{\bar{\alpha}} \sigma & = 0 \, . \label{eq:sgm-hol} 
	\end{align}
\end{subequations}
If $\sigma$ is allowed to have a non-empty zero set, the resulting structure is called \emph{almost CR--Einstein}. The prolongation of the system \eqref{eq:A-CR-EI} leads to the construction of the so-called \emph{CR tractor bundle}, that is, a complex rank-$(m+2)$ vector bundle, equipped with the so-called \emph{normal tractor connection}. In fact, any solution to \eqref{eq:A-CR-EI} can be equivalently encoded as a parallel section of the CR tractor bundle --- see \cite{Cap2008} for details.
	
\begin{rem}
In \cite{TaghaviChabert2023}, where $(\mc{M},H,J)$ is assumed to be involutive, densities satisfying condition \eqref{eq:sgm-hol} are referred to \emph{CR densities} in analogy to CR functions. These are shown to be equivalent to the existence of (nowhere vanishing) closed sections of the canonical bundle, which is always guaranteed for realisable CR manifolds.

Such an interpretation clearly does not hold in the non-involutive case however. Certainly, strictly almost CR manifolds cannot admit closed sections of the canonical bundle as the structure equations will immediately reveal. This being said, CR densities in the sense of \eqref{eq:sgm-hol} may still exist on strictly almost CR manifolds as the example in Appendix \ref{app:NP} shows.
\end{rem}

We may nevertheless re-express the CR--Einstein conditions \eqref{eq:CR-E} in terms of a CR scale:
\begin{thm}\label{thm:alCRE}
	A CR--Einstein structure on an almost CR manifold $(\mc{M},H,J)$ determines and is determined by a unique (up to constant rescalings) CR scale $s$ that satisfies the system of CR invariant differential equations:
\begin{subequations}\label{eq:CR-E-inv}
	\begin{align}
		& \nabla_{(\alpha} \nabla_{\beta)} s + \i \WbA_{\alpha \beta} s + \Nh_{\gamma (\alpha \beta)} \nabla^{\gamma} s = 0 \, , \label{eq:AalCRE} \\
		& \nabla^{\gamma} s \Nh_{\gamma (\alpha \beta)} - \tfrac{1}{m} \nabla^{\gamma} \Nh_{\gamma (\alpha \beta)} s = 0 \, , \label{eq:NalCRE} \\
		& \left( s \nabla_{\bar{\beta}} \nabla_{\alpha} s - \nabla_{\alpha} s  \nabla_{\bar{\beta}} s + \Rho_{\alpha \bar{\beta}} s^2 - \tfrac{1}{m+2} \Nh_{\alpha \gamma \delta} \Nh_{\bar{\beta}}{}^{\gamma \delta} s^2 \right)_\circ = 0 \, . \label{eq:RalCRE}
	\end{align}
\end{subequations}
\end{thm}

\begin{proof}
	Choosing $\nabla$ such that $\nabla s = 0$ immediately gives us the implication \eqref{eq:CR-E-inv} $\Rightarrow$ \eqref{eq:CR-E}. For the converse, it is enough to use the transformation laws for $\WbA_{\alpha \beta}$, $\nabla^{\gamma}\Nh_{\gamma \alpha \beta}$ and $\Rho_{\alpha \bar{\beta}}$ given by \eqref{eq:trsfA}, \eqref{eq:trsfDNh} and \eqref{eq:trsfPt} with $\Upsilon = - s^{-1} \nabla s$. The form of \eqref{eq:RalCRE} can then be obtained by applying the commutation relation \eqref{eq:CR_com1}.
\end{proof}

\begin{defn}
	Let $(\mc{M},H,J)$ be an almost CR manifold. We shall call any real-valued density $s$ of weight $(1,1)$ that is a solution of the system \eqref{eq:CR-E-inv} an \emph{almost CR--Einstein scale}. If $s$ is a CR scale (and so, nowhere vanishing), it will be referred to simply as a \emph{CR--Einstein scale}.
\end{defn}

In this article, we shall however only be concerned with solutions that are CR scales.

\begin{rem}
	Equation \eqref{eq:AalCRE} already featured in \cite{Matsumoto2022} in the context of deformation of almost CR structures. Using \eqref{eq:CR_com3}, it can be recast as
	\begin{align*}
		\nabla_{\alpha} \nabla_{\beta} s + \i \WbA_{\alpha \beta} s + \Nh_{\gamma \alpha \beta} \nabla^{\gamma} s & = 0 \, .
	\end{align*}
\end{rem}

We now turn to the geometric interpretation of each of the conditions given in the previous theorem. That \eqref{eq:AalCRE} characterises the existence of a CR symmetry is clear, since it defines a pseudo-Hermitian structure for which \eqref{eq:CR-E:A} holds, i.e.\ the Reeb vector field is an infinitesimal CR symmetry. In addition, applying \cite[Corollary~4.1]{TaghaviChabert2022} yields the following:
\begin{prop}\label{prop:CRE-KE}
	Let $(\mc{M},H,J)$ be an almost CR manifold, and let $s \in \Gamma(\mc{E}(1,1))$ be a CR scale. Then $s$ satisfies \eqref{eq:AalCRE} if and only if the Reeb vector field of the contact form $\theta = s^{-1} \bm{\theta}$ is an infinitesimal CR symmetry. This being the case, the leaf space $\ul{\mc{M}}$ of the Reeb foliation locally inherits an almost K\"{a}hler structure $(\ul{h},\ul{J})$. In addition,
	\begin{itemize}
		\item $s$ satisfies \eqref{eq:NalCRE} if and only if the Ricci tensor of $\ul{h}$ is of type $(1,1)$ with respect to $\ul{J}$;
		\item $s$ satisfies \eqref{eq:NalCRE} and \eqref{eq:RalCRE} if and only if  $\ul{h}$ is Einstein.
	\end{itemize} 
\end{prop}

In general, a CR--Einstein structure does not arise from a density $\sigma$ of weight $(1,0)$ that satisfies \eqref{eq:sgm-hol} as illustrated by the example in Appendix \ref{app:NP}. It is nonetheless instructive to write down a system of equations describing such a structure.
\begin{thm}\label{thm:spalCRE}
	Suppose that an almost CR manifold $(\mc{M},H,J)$ admits a nowhere vanishing density $\sigma$ of weight $(1,0)$ that satisfies the system of CR invariant differential equations
	\begin{subequations}\label{eq:spalCREall}
	\begin{align}
		& \nabla_{\bar{\alpha}} \sigma = 0 \, , \label{eq:spalCRE} \\
		& \nabla_{(\alpha} \nabla_{\beta)} \sigma + \i \WbA_{\alpha \beta} \sigma + \tfrac{1}{m} \nabla^{\gamma} \Nh_{\gamma (\alpha \beta)} \sigma = 0 \, , \label{eq:AspalCRE} \\
		& \nabla_{\gamma} \sigma \Nh^{\gamma}{}_{(\bar{\alpha} \bar{\beta})} - \tfrac{1}{m} \nabla_{\gamma} \Nh^{\gamma}{}_{(\bar{\alpha} \bar{\beta})} \sigma = 0 \, , \label{eq:NspalCRE1} \\
		& \left( 2 \Nh_{\alpha \gamma \delta} \Nh_{\bar{\beta}}{}^{\gamma \delta} -  \Nh_{\gamma \delta \alpha} \Nh^{ \gamma \delta}{}_{\bar{\beta}}  \right)_\circ = 0 \, . \label{eq:NspalCRE}
	\end{align}
	\end{subequations}
	Then $s := \sigma \ol{\sigma}$ is a CR--Einstein scale, i.e.\ $s$ satisfies \eqref{eq:CR-E-inv}.
	
	Conversely, suppose $s$ is a CR--Einstein scale, and $s = \sigma \ol{\sigma}$ for some $\sigma \in \Gamma(\mc{E}(1,0))$ that satisfies \eqref{eq:spalCRE}. Then $\sigma$ also satisfies \eqref{eq:AspalCRE}, \eqref{eq:NspalCRE1} and \eqref{eq:NspalCRE}.
\end{thm}

\begin{proof}
	We note that using \eqref{eq:NspalCRE1}, we can re-express \eqref{eq:AspalCRE} equivalently as
	\begin{align}
			& \nabla_{(\alpha} \nabla_{\beta)} \sigma + \i \WbA_{\alpha \beta} \sigma + \tfrac{1}{m} \nabla^{\gamma} \Nh_{\gamma (\alpha \beta)} \sigma = 0 \, . \label{eq:AspalCRE2} 
	\end{align}
	Now, set $s=\sigma \ol{\sigma}$. Assuming \eqref{eq:spalCRE}, we have $\ol{\sigma} \nabla_{\alpha} \sigma = \nabla_{\alpha} s$. So, multiplying \eqref{eq:AspalCRE2} and \eqref{eq:NspalCRE1} through by $\ol{\sigma}$ yields
	\begin{align*}
\nabla_{(\alpha} \nabla_{\beta)} s + \i \WbA_{\alpha \beta} s + \Nh_{\gamma (\alpha \beta)} \nabla^{\gamma} s = \ol{\sigma} \left( \nabla_{(\alpha} \nabla_{\beta)} \sigma + \i \WbA_{\alpha \beta} \sigma + \Nh_{\gamma (\alpha \beta)} \nabla^{\gamma} \sigma \right) \, , \\
\nabla^{\gamma} s \Nh_{\gamma (\alpha \beta)} - \tfrac{1}{m} \nabla^{\gamma} \Nh_{\gamma (\alpha \beta)} s = \ol{\sigma} \left( \nabla^{\gamma} \sigma \Nh_{\gamma (\alpha \beta)} - \tfrac{1}{m} \nabla^{\gamma} \Nh_{\gamma (\alpha \beta)} \sigma \right) \, .
\end{align*}
Similarly, we find 
		\begin{multline*}
			\left( s \nabla_{\bar{\beta}} \nabla_{\alpha} s - \nabla_{\alpha} s  \nabla_{\bar{\beta}} s + \Rho_{\alpha \bar{\beta}} s^2 - \tfrac{1}{m+2} \Nh_{\alpha \gamma \delta} \Nh_{\bar{\beta}}{}^{\gamma \delta} s^2 \right)_\circ \\
			= s \ol{\sigma} \left( \nabla_{\bar{\beta}} \nabla_{\alpha} \sigma  + \Rho_{\alpha \bar{\beta}} \sigma - \tfrac{1}{m+2} \Nh_{\alpha \gamma \delta} \Nh_{\bar{\beta}}{}^{\gamma \delta} \sigma \right)_\circ \\
			= s \ol{\sigma} \tfrac{1}{m+2} \left( \Nh_{\gamma \delta \alpha} \Nh^{ \gamma \delta}{}_{\bar{\beta}} - 2 \Nh_{\alpha \gamma \delta} \Nh_{\bar{\beta}}{}^{\gamma \delta}  \right)_\circ \, ,
		\end{multline*}
where we have used \eqref{eq:CR_com1} for the second equality. Hence, since $\sigma$ is nowhere vanishing, we obtain an equivalence between \eqref{eq:AspalCRE} (or \eqref{eq:AspalCRE2}), \eqref{eq:NspalCRE1}, \eqref{eq:NspalCRE} and \eqref{eq:AalCRE}, \eqref{eq:NalCRE}, \eqref{eq:RalCRE}. This establishes both claims of the theorem.
\end{proof}

\begin{rem}
	As pointed out in \cite{Matsumoto2022}, the system of equations \eqref{eq:spalCRE} and \eqref{eq:AspalCRE} arises from the \emph{first BGG operator} on densities of weight $(1,0)$ defined by
	\begin{align*}
		\sigma \mapsto \left( \nabla_{\bar{\alpha}} \sigma , \nabla_{(\alpha} \nabla_{\beta)} \sigma + \i \WbA_{\alpha \beta} \sigma + \tfrac{1}{m} \nabla^{\gamma} \Nh_{\gamma (\alpha \beta)} \sigma \right) \, .
	\end{align*}
For densities $\sigma \in \Gamma( \mc{E}(1,w))$, for any $w \in \Z$, a related CR invariant equation of interest is given by
	\begin{align*}
		\nabla_{(\alpha} \nabla_{\beta)} \sigma + \i \WbA_{\alpha \beta} \sigma + \tfrac{1}{m} \nabla^{\gamma} \Nh_{\gamma (\alpha \beta)}  \sigma & = 0 \, ,
	\end{align*}
or equivalently,
	\begin{align*}
		\nabla_{\alpha} \nabla_{\beta} \sigma + \i \WbA_{\alpha \beta} \sigma + \tfrac{1}{m} \nabla^{\gamma} \Nh_{\gamma (\alpha \beta)} \sigma + \Nh_{\gamma [\alpha \beta]} \nabla^{\gamma} \sigma & = 0 \, .
	\end{align*}
If $\sigma$ is a real-valued density of weight $(1,1)$, then the CR scale determines a contact form, whose pseudo-Hermitian invariants satisfy
	\begin{align*}
		\WbA_{\alpha \beta} & = \tfrac{\i}{m}  \nabla^{\gamma} \Nh_{\gamma (\alpha \beta)} \, .
	\end{align*}
	
\end{rem}

Interestingly, solutions to \eqref{eq:spalCREall} provide a nice interpretation of $\|\Nh \|^{-2}$ in the following terms:
\begin{prop}\label{prop:spalCRE-N2}
	Let $(\mc{M},H,J)$ be an almost CR manifold. Assume that $\| \Nh \|^2$ is nowhere vanishing. Suppose there is a nowhere vanishing solution $\sigma \in \Gamma(\mc{E}(1,0))$ to \eqref{eq:spalCREall}. Then the CR--Einstein scale $s := \sigma \ol{\sigma}$ is proportional to $\| \Nh \|^{-2}$ by a constant factor.
\end{prop}

\begin{proof}
	Suppose that equations \eqref{eq:spalCREall} hold. Then the commutation relations \eqref{eq:CR_com} on $\sigma$, where $\nabla$ is taken to be the Webster connection preserving $s$, reduce to
	\begin{subequations}
		\begin{align}
			0 & = \tfrac{1}{m+2} \left( - \Nh_{\gamma \delta \alpha} \Nh^{ \gamma \delta}{}_{\bar{\beta}} + 2 \Nh_{\alpha \gamma \delta} {\Nh}_{\bar{\beta}}{}^{\gamma \delta} \right) \sigma + \bm{h}_{\alpha \bar{\beta}} \left( \tfrac{1}{m+2} \Lambda \sigma s^{-1} - \i {\nabla}_{0} \sigma \right)\, , \label{eq:N2CRE1} \\
			{\nabla}_{\alpha} {\nabla}_{0} \sigma & = 0\, , \label{eq:N2CRE2} \\
			0 & = \tfrac{1}{m+2} {\nabla}^{\gamma} {\Nh}_{\alpha \beta \gamma} \sigma  \, . \label{eq:N2CRE3}
		\end{align}
	\end{subequations}
Contracting equation \eqref{eq:N2CRE1} gives $\| \Nh \|^2 = - m \Lambda s^{-1}  + \i m (m+2) \sigma^{-1} {\nabla}_{0} \sigma$. Hence, using equation \eqref{eq:N2CRE2}, we find that $\nabla_{\alpha} \left( \| \Nh \|^2 \right) = 0$, and using the commutation \eqref{eq:CR_com1}, we find that $\| \Nh \|^2 \in \Gamma(\mc{E}(-1,-1))$ is covariantly constant, and so its inverse must be proportional to $s$ by a constant factor.
\end{proof}

\begin{rem}
	With reference to the proof of Proposition \ref{prop:spalCRE-N2}, since $\Lambda$ as defined by \eqref{eq:LScN2} is constant, $\Sc$ must also be constant with respect to the Webster connection $\nabla$ compatible with $\|\Nh \|^2$. Note also that equation \eqref{eq:N2CRE3} tells us that ${\nabla}^{\gamma} {\Nh}_{\gamma \alpha \beta} = 0$ since $\sigma \neq 0$ everywhere.
\end{rem}

\begin{rem}
	We may relax our assumption on $\sigma$ by allowing it to have a non-empty zero set. In this case, analogous results will hold only off its zero set.
\end{rem}

	\section{Conformal geometry}\label{sec:conf_geom}
	This section contains a review of conformal geometry with a focus on the properties of congruences of null geodesics, which are at the heart of this article. It is in this context that we introduce generalisations of the notion of almost Einstein scales in Section \ref{sec:confsc}.
	\subsection{Preliminaries}
	Throughout $(\wt{\mc{M}}, \wt{\mbf{c}})$ will denote an oriented and time-oriented conformal manifold of dimension $n+2$ and of Lorentzian signature $(+,\ldots,+,-)$. For consistency with the rest of the article, we shall adorn tensors on $\wt{\mc{M}}$ with a tilde. Two metrics $\wt{g}$ and $\wh{\wt{g}}$ belong to the conformal class $\wt{\mbf{c}}$ if and only if
	\begin{align}\label{eq-conf-res}
		\widehat{\wt{g}} & = \e^{2 \, \wt{\varphi}} \wt{g} \, , & \mbox{for some smooth function $\wt{\varphi}$ on $\wt{\mc{M}}$.}
	\end{align}
	Following \cite{Bailey1994}, for each $w \in \R$, there are naturally associated density bundles denoted $\wt{\mc{E}}[w]$, and the Levi-Civita connection of any metric $\wt{g}$ in $\wt{\mbf{c}}$ extends to a linear connection on $\wt{\mc{E}}[w]$. In particular, metrics in $\wt{\mbf{c}}$ are in one-to-one correspondence with sections of the bundle of \emph{conformal scales}, denoted $\wt{\mc{E}}_+[1]$, which is a choice of positive square root of $\wt{\mc{E}}[2]$. The correspondence is achieved by means of the \emph{conformal metric} $\wt{\bm{g}}$, which is a distinguished non-degenerate section of $\bigodot^2 T \wt{\mc{M}} \otimes \wt{\mc{E}} [2]$ with the property that if $\wt{\sigma}$ is a conformal scale, then the corresponding metric in $\wt{\mbf{c}}$ is given by $\wt{g} = \wt{\sigma}^{-2} \wt{\bm{g}}$.  We thus have a canonical identification of $T \wt{\mc{M}}$ with $T^* \wt{\mc{M}} \otimes \wt{\mc{E}} [2]$ via $\wt{\bm{g}}$. The Levi-Civita connection $\wt{\nabla}$ of $\wt{g}$ also preserves $\wt{\sigma}$, and it follows that $\wt{\bm{g}}$ is preserved by the Levi-Civita connection of any metric $\wt{g}$ in $\wt{\mbf{c}}$.

	We shall often use the abstract index notation, whereby sections of $T \wt{\mc{M}}$, respectively, $T^* \wt{\mc{M}}$ will be adorned with upper, respectively, lower minuscule Roman indices starting from the beginning of the alphabet, and we will write $\wt{\mc{E}}^{a} := T \wt{\mc{M}}$, and $\wt{\mc{E}}_{a} :=T^* \wt{\mc{M}}$. Symmetrisation will be denoted by round brackets, and skew-symmetrisation by square brackets as before. Indices will be lowered and raised with $\wt{\bm{g}}_{a b}$ and its inverse $\wt{\bm{g}}^{a b}$ respectively. The trace-free part of symmetric tensors will be denoted by a ring.
	
	For any two metrics $\wt{g}$ and $\wh{\wt{g}}$ in $\wt{\mbf{c}}$ related by \eqref{eq-conf-res}, their respective Levi-Civita connections $\wt{\nabla}$ and $\wh{\wt{\nabla}}$ are related by
	\begin{align*}
		\widehat{\wt{\nabla}}_a \wt{\alpha}_b & = \wt{\nabla}_a \wt{\alpha}_b + (w - 1) \wt{\Upsilon}_a \wt{\alpha}_b - \wt{\Upsilon}_b \wt{\alpha}_a + \wt{\Upsilon}_c \wt{\alpha}^c \wt{\bm{g}}_{a b} \, , & \mbox{for any $\wt{\alpha}_a \in \Gamma(\wt{\mc{E}}_a[w])$,}
	\end{align*}
	where $\wt{\Upsilon}_a := \wt{\nabla}_a \wt{\varphi}$.

	By convention, we take the Riemann tensor of a given metric $\wt{g}_{a b}$ in $\wt{\mbf{c}}$ to be defined by
	\begin{align*}
		2 \wt{\nabla}_{[a} \wt{\nabla}_{b]} \wt{\alpha}_c & = - \wt{\Riem}_{a b}{}^{d}{}_{c} \wt{\alpha}_d \, , &  \mbox{for any $\wt{\alpha}_a \in \Gamma(\wt{\mc{E}}_a[w])$,} \, .
	\end{align*}
	It decomposes as
	\begin{align*}
		\wt{\Riem}_{a b c d} & = \wt{\Weyl}_{a b c d} + 4 \wt{\bm{g}}_{[a|[c} \wt{\Rho}_{d]|b]} \, ,
	\end{align*}
	where $\wt{\Weyl}_{a b c d}$ is the \emph{Weyl tensor} and $\wt{\Rho}_{a b}$ the \emph{Schouten tensor}, which is related to the \emph{Ricci tensor} $\wt{\Ric}_{a b} = \wt{\Riem}_{c a}{}^{c}{}_{b}$ and the \emph{Ricci scalar} $\wt{\Sc} = \wt{\Ric}_{a}{}^{a}$ by
	\begin{align*}
		\wt{\Rho}_{a b} & = \tfrac{1}{n} \left( \wt{\Ric}_{a b} - \tfrac{1}{2(n+1)} \wt{\Sc} \wt{\bm{g}}_{a b} \right) \, .
	\end{align*}
	The \emph{Schouten scalar} is defined to be $\wt{\Rho} := \wt{\Rho}_{a b} \wt{\bm{g}}^{a b} = \tfrac{1}{2(n+1)} \wt{\Sc} $. The \emph{Cotton tensor} is given by
	\begin{align*}
		\wt{\Cot}_{c a b} & = 2 \wt{\nabla}_{[a} \wt{\Rho}_{b] c} \, ,
	\end{align*}
	and, by the Bianchi identities, satisfies $(n-1)  \wt{\Cot}_{c a b} = \wt{\nabla}^{d} \wt{\Weyl}_{d c a b}$. 
	
	While the Weyl tensor is conformally invariant, the Schouten tensor, Schouten scalar and Cotton tensor transform as
	\begin{align}
		\widehat{\wt{\Rho}}_{a b} &  = \wt{\Rho}_{a b} - \wt{\nabla}_{a} \wt{\Upsilon}_{b} + \wt{\Upsilon}_{a} \wt{\Upsilon}_{b} - \tfrac{1}{2} \wt{\Upsilon}^{c} \wt{\Upsilon}_{c} \wt{\bm{g}}_{a b} \, , &
		\widehat{\wt{\Rho}} &  = \wt{\Rho} - \wt{\nabla}^{a} \wt{\Upsilon}_{a} - \tfrac{n}{2} \wt{\Upsilon}^{c} \wt{\Upsilon}_{c} \, , \label{eq:Rho_transf} \\
				\widehat{\wt{\Cot}}_{c a b} & = \wt{\Cot}_{c a b} + \wt{\Upsilon}^d \wt{\Weyl}_{d c a b} \, , \label{eq:Cot_transf}
	\end{align}
	respectively.

	\subsection{Optical geometry}\label{sec:optical}
	We recall a number of notions introduced in \cite{Trautman1984,Trautman1985,Robinson1985,Penrose1986,Robinson1986,Robinson1989,Trautman1999,Fino2023a}. An \emph{optical geometry} consists of a triple $(\wt{\mc{M}}, \wt{\mbf{c}}, \wt{K})$, where $(\wt{\mc{M}}, \wt{\mbf{c}})$ is an oriented and time-oriented Lorentzian conformal manifold of dimension $n+2$ and $\wt{K}$ is an oriented null line distribution, which we shall also refer to as \emph{an optical structure}. The rank-$n$ \emph{screen bundle} $\wt{H} := \wt{K}^\perp /\wt{K}$ inherits a conformal structure $\wt{\mbf{c}}_{\wt{H}}$ of Riemannian signature.

Of importance in the present article is when $\wt{K}$ is tangent to a \emph{non-shearing congruence of null geodesics} $\wt{\mc{K}}$, i.e.\ for any nowhere vanishing section $\wt{k}$ of $\wt{K}$, i.e.\
		\begin{align}
		\mathsterling_{\wt{k}} \wt{g} (\wt{v} , \wt{w} ) & = \wt{\epsilon} \wt{g} (\wt{v} , \wt{w} )  \, , & \wt{v}, \wt{w} \in \Gamma(\wt{K}^\perp) \, , \label{eq-nonshear} 
	\end{align}
	for some smooth function $\wt{\epsilon}$. This means that the integral curves of $\wt{k}$ are null geodesics and the conformal structure on $\wt{H}$ is preserved along these. By convention, we will always assume that $\wt{K}$, or equivalently $\wt{k}$, is positively oriented. The local leaf space $\mc{M}$ of $\wt{\mc{K}}$ thus inherits a rank-$n$ distribution $H$ from $\wt{H}$, while $\wt{\mbf{c}}_{\wt{H}}$ descends to a bundle conformal structure of Riemannian signature on $H$.
	
	In addition, we shall assume that $\wt{\mc{K}}$ is \emph{twisting}, i.e.\
for any nowhere vanishing one-form $\wt{\kappa} \in \Gamma( \Ann(\wt{K}))$,
\begin{align}
	\d \wt{\kappa} ( \wt{v} , \wt{w}) & \neq 0 \, , & \wt{v}, \wt{w} & \in \Gamma(\wt{K}^\perp) \, ,
\end{align}
which means that $\wt{K}^\perp$, and thus $H$, are not integrable.

\subsection{Almost Robinson geometry}\label{sec:Robinson}
When $(\wt{\mc{M}}, \wt{\mbf{c}})$ is of dimension $2m+2$, a particular case of an optical structure is provided by the notion of \emph{almost Robinson structure} \cite{Nurowski2002,Fino2023}, that is, a pair $(\wt{N},\wt{K})$ where $\wt{N}$ is a totally null complex $(m+1)$-plane distribution, i.e.\
\begin{align*}
	 \wt{\bm{g}} (\wt{v},\wt{w}) & = 0 \, , & \mbox{for all $\wt{v}, \wt{w} \in \Gamma(\wt{N})$,}
\end{align*}
and $\wt{K}$ a real null line distribution such that ${}^\C \wt{K} = \wt{N} \cap \ol{\wt{N}}$. One can show that $(\wt{N},\wt{K})$ is equivalent to an optical structure $\wt{K}$ whose screen bundle is equipped with a bundle complex structure $\wt{J}$ compatible with the induced conformal structure $\wt{\mbf{c}}_{\wt{H}}$. Again, let us consider the leaf space $\mc{M}$ of the congruence $\wt{\mc{K}}$ of null curves tangent to $\wt{K}$. When $\wt{N}$ is preserved along the flow of any generator of $\wt{\mc{K}}$, i.e.\ $[\wt{K},\wt{N}] \subset \wt{N}$, we refer to $(\wt{N},\wt{K})$ as a \emph{nearly Robinson structure}. In this case, the curves of $\wt{K}$ are geodesics, and the leaf space $\mc{M}$ inherits an almost CR structure $(H,J)$. More precisely, local sections of $\Ann(H^{(0,1)})$ pull back to local sections of $\Ann(\ol{\wt{N}})$. If in addition $\wt{N}$ is involutive, i.e.\ $[\wt{N},\wt{N}] \subset \wt{N}$, we refer to $(\wt{N},\wt{K})$ simply as a \emph{Robinson structure}, which also implies the involutivity of the almost CR structure $(H,J)$ on $\mc{M}$.

In dimension four, an optical structure is equivalent to an almost Robinson structure. The former is tangent to a non-shearing congruence of null geodesics if and only if the latter is involutive. This being case, the Weyl tensor satisfies the integrability condition
	\begin{align}\label{eq:int-non-sh}
		\wt{\Weyl}( \wt{k}, \wt{v}, \wt{k}, \wt{v}) & = 0 \, , & \mbox{for any $\wt{k} \in \Gamma(\wt{K})$, $\wt{v} \in \Gamma(\wt{K}^\perp)$.}
	\end{align}
In the jargon of mathematical relativity, $\wt{K}$ is said to be a \emph{principal null direction} of the Weyl tensor.

In contrast, in even dimensions greater than four, an optical geometry does not in general admit a distinguished almost Robinson structure. We note however the following special case: we say that an almost Robinson structure $(\wt{N},\wt{K})$ is \emph{twist-induced} if the congruence $\wt{\mc{K}}$ tangent to $\wt{K}$ is geodesic and twisting, and the associated bundle complex structure $\wt{J}$ on the screen bundle $\wt{H}$ is compatible with the twist in the sense that if $\wt{\kappa}$ is any nowhere vanishing section of $\Ann(\wt{K}^\perp)$, then
\begin{align*}
	\d \wt{\kappa} (\wt{v},\wt{w}) & \propto \wt{h} \left( \wt{J} (\wt{v} + \wt{K}) , \wt{w} + \wt{K} \right) \, , & \mbox{for all $\wt{v}, \wt{w} \in \Gamma(\wt{K}^\perp)$,}
\end{align*}
for some bundle metric $\wt{h}$ in $\wt{\mbf{c}}_{\wt{H}}$. If in addition $\wt{\mc{K}}$ is non-shearing, then $(\wt{N},\wt{K})$ descends to a partially integrable contact almost CR structure of positive definite signature on the leaf space $\mc{M}$ of $\wt{\mc{K}}$.

Crucially, we have the next result:
\begin{thm}[\cite{TaghaviChabert2022}]\label{thm:non-shearing-hidim}
	Let $(\wt{\mc{M}},\wt{\mbf{c}},\wt{K})$ be an optical geometry of dimension $2m+2$ equipped with a twisting non-shearing congruence of null geodesics $\wt{\mc{K}}$. Then the twist of $\wt{\mc{K}}$ induces a nearly Robinson structure $(\wt{N},\wt{K})$ if and only if the Weyl tensor satisfies \eqref{eq:int-non-sh}. This being the case, the local leaf space $\mc{M}$ of the congruence inherits a partially integrable contact almost CR structure $(H,J)$ from $(\wt{N},\wt{K})$.
\end{thm}

\begin{rem}
	For $m=1$, the above theorem is perfectly consistent with the extent literature since the integrability condition then becomes vacuous. On the other hand, we may rightly see condition \eqref{eq:int-non-sh} for $m>1$ as the generalisation of the defining equation for $\mc{K}$ being a principal null direction of the Weyl tensor.
\end{rem}

\subsection{Twist-induced nearly Robinson geometry with non-shearing congruence}\label{sec:prefered}
Let us focus on a twist-induced nearly Robinson geometry $(\wt{\mc{M}},\wt{\mbf{c}},\wt{N},\wt{K}) \accentset{\varpi}{\longrightarrow} (\mc{M},H,J)$ with non-shearing congruence $\wt{\mc{K}}$. For a congruence of null curves, the geodesic property of $\wt{K}$, its shear and twist are all conformally invariant and do not depend on the choice of generator of $\wt{\mc{K}}$. The question that we need to address is how metrics in $\wt{\mbf{c}}$ are related to contact forms for $(\mc{M},H,J)$. To this end, we shall use the following canonical structures arising from our geometric setup \cite{TaghaviChabert2022,Fino2023a,Fino2023}:
\begin{itemize}
	\item 		There is a conformal subclass $\accentset{n.e.}{\wt{\mbf{c}}}$ of metrics in $\wt{\mbf{c}}$ with the property that whenever $\wt{g}$ is in $\accentset{n.e.}{\wt{\mbf{c}}}$, the congruence $\wt{\mc{K}}$ is \emph{non-expanding}, i.e.\
	\begin{align}
		\mathsterling_{\wt{k}} \wt{g} (\wt{v} , \wt{w} ) & = 0  \, , & \wt{v}, \wt{w} & \in \Gamma(\wt{K}^\perp) \, . \label{eq-nonshear_nonexp} 
	\end{align}
	Any two metrics in $\accentset{n.e.}{\wt{\mbf{c}}}$ differ by a factor constant along $\wt{\mc{K}}$.
	\item 	From the twisting property of $\wt{\mc{K}}$, there is a \emph{distinguished} generator $\wt{k}$ of $\wt{\mc{K}}$ with the property that for any metric $\wt{g}$ in $\accentset{n.e.}{\wt{\mbf{c}}}$, we have that
	\begin{align*}
		\wt{\kappa} & := \wt{g}(\wt{k},\cdot) = 2 \varpi^* \theta \, ,
	\end{align*}
	where $\theta$ is a contact form for $(H,J)$. In addition, for each such choice of metric, there is a unique associated null vector field $\wt{\ell}$ with the property that $\wt{\kappa}(\wt{\ell})=1$ and $\d \wt{\kappa}(\wt{\ell},\cdot)=0$.
\end{itemize}
Thanks to $\wt{k}$ and $\accentset{n.e.}{\wt{\mbf{c}}}$, we obtain a tight relation between $(\wt{\mc{M}},\wt{\mbf{c}},\wt{K})$ and the leaf space $(\mc{M},H,J)$ of $\wt{\mc{K}}$. To be precise, there is a one-to-one correspondence between metrics in $\accentset{n.e.}{\wt{\mbf{c}}}$ and contact forms for $(H,J)$. Suppressing the pullback symbol $\varpi^*$, each metric $\wt{g}_{\theta}$ in $\accentset{n.e.}{\wt{\mbf{c}}}$ associated to some contact form $\theta$ can be written as 
\begin{subequations}\label{eq:can_met}
	\begin{align}\label{eq:can_met1}
		\wt{g}_{\theta} & = 4 \theta \odot \wt{\lambda} + h \, ,
	\end{align}
	where $h$ is (the metric induced from) the Levi form of $\theta$, and $\wt{\lambda} = \wt{g}_{\theta}(\wt{\ell},\cdot)$ with $\wt{\ell}$ as defined above. Choosing an admissible coframe $(\theta^{\alpha})$ for $\theta$, and an affine parameter $\phi$ along the geodesics of $\wt{k}$ so that $\wt{k}=\parderv{}{\phi}$, we can write
	\begin{align}\label{eq:can_met2}
		h & = 2 h_{\alpha \bar{\beta}} \theta^{\alpha} \odot \ol{\theta}{}^{\bar{\beta}} \, , & 
		\wt{\lambda} & = \d \phi + \wt{\lambda}_{\alpha} \theta^{\alpha} + \wt{\lambda}_{\bar{\alpha}} \ol{\theta}{}^{\bar{\alpha}} + \wt{\lambda}_0 \theta \, ,
	\end{align}
\end{subequations}
for some complex-valued functions $\wt{\lambda}_{\alpha}$ and $\wt{\lambda}_{\bar{\alpha}}=\ol{\wt{\lambda}_{\alpha}}$, and real-valued function $\wt{\lambda}_0$ on $\wt{\mc{M}}$. Differentiation with respect to $\phi$ will be denoted by a dot, i.e.\ $\dot{\wt{f}} := \mathsterling_{\wt{k}} \wt{f}$ for any smooth tensor-valued function $\wt{f}$ on $\wt{\mc{M}}$, and this notation will be extended to tensor components.

Any other choice of metric $\wt{g}{}_{\wh{\theta}} = \e^\varphi \wt{g}{}_{\theta}$ in $\accentset{n.e.}{\wt{\mbf{c}}}$ for some smooth function $\varphi$ on $\mc{M}$ is given by
\begin{align*}
	\wt{g}_{\wh{\theta}} & = 4 \wh{\theta} \odot \wh{\wt{\lambda}} + \wh{h} \, ,
\end{align*}
where
\begin{align*}
	\wh{\theta} & = \tfrac{1}{2} \wt{g}_{\wh{\theta}}(\wt{k}, \cdot) = \e^{\varphi} \theta \, , \\
	\wh{h} & = \e^{\varphi} \left( h - 2 \i \Upsilon_{\alpha} \theta^{\alpha} \odot \theta + 2 \i \Upsilon_{\bar{\alpha}} \ol{\theta}{}^{\bar{\alpha}} \odot \theta + 2 \Upsilon_{\alpha} \Upsilon^{\alpha} \theta \odot \theta \right) \, , \\
	\wh{\wt{\lambda}} & = \wt{\lambda} + \tfrac{1}{2} \i \Upsilon_{\alpha} \theta^{\alpha} - \tfrac{1}{2} \i \Upsilon_{\bar{\alpha}} \ol{\theta}{}^{\bar{\alpha}} - \tfrac{1}{2} \Upsilon_{\alpha} \Upsilon^{\alpha} \theta \, ,
\end{align*}
with $(\theta^{\alpha})$ being admissible for $(H,J,\theta)$, and $\Upsilon_{\alpha} = \nabla_{\alpha} \varphi$. On the other hand, a change of affine parametrisation
\begin{align}\label{eq:reparam}
	\phi \mapsto \phi'= \phi - \mr{\phi} \, ,
\end{align}
for some smooth function $\mr{\phi}$ on $\mc{M}$, accompanied by the redefinitions
\begin{align*}
	\wt{\lambda}'_{\alpha} & = \wt{\lambda}_{\alpha} + \nabla_{\alpha} \mr{\phi} \, , & \wt{\lambda}'_{0} & = \wt{\lambda}_{0} + \nabla_{0} \mr{\phi} \, ,
\end{align*}
preserves the form of the metric \eqref{eq:can_met} and clearly the vector field $\wt{k}=\tfrac{\partial}{\partial \phi'}$ too.

In the light of Theorem \ref{thm:non-shearing-hidim}, we shall henceforth denote a twist-induced nearly Robinson geometry with non-shearing congruence by the quadruple $(\wt{\mc{M}},\wt{\mbf{c}},\wt{N},\wt{k}) \longrightarrow (\mc{M},H,J)$, where $\wt{k}$ is the distinguished section of $\wt{K}$ as described above. The perturbed Fefferman spaces to be introduced in Section \ref{sec:Fefferman} are special cases of these.

\subsection{Almost Einstein scales and generalisations}\label{sec:confsc}
A metric $\wt{g}$ in $\wt{\mbf{c}}$ is said to be \emph{Einstein} if its Ricci tensor satisfies
\begin{align*}
	\wt{\Ric} & = \wt{\Lambda} \wt{g} \, , & \mbox{for some constant $\wt{\Lambda}$.}
\end{align*}
 If $(\wt{\mc{M}},\wt{\mbf{c}})$ is equipped with an almost Robinson structure, we may consider subsystems of the Einstein equations:
\begin{defn}\label{defn:redEins}
	Let $(\wt{\mc{M}},\wt{\mbf{c}},\wt{N},\wt{K})$ be an almost Robinson geometry. We say that a metric $\wt{g}$ in $\wt{\mbf{c}}$ is:
	\begin{itemize}
		\item a \emph{weakly half-Einstein} metric if its Ricci tensor satisfies
		\begin{align}\label{eq:wk_hlf_Einstein}
			\wt{\Ric}_{a b} \wt{v}^a \wt{v}^b & = 0 \, , & \mbox{for any $\wt{v} \in \Gamma(\wt{N})$;}
		\end{align}
		\item a \emph{half-Einstein} metric if it is weakly half-Einstein and has constant Ricci scalar curvature.
		\end{itemize}
\end{defn}
For optical geometries, we have:
	\begin{defn}\label{defn:purad}
		Let $(\wt{\mc{M}},\wt{\mbf{c}},\wt{K})$ be an optical geometry. A metric $\wt{g}$ in $\wt{\mbf{c}}$ is said to be a \emph{pure radiation} metric if its Ricci scalar is constant and the trace-free part of the Ricci tensor satisfies
		\begin{align}\label{eq:puradvac}
			\wt{\kappa}_{[a} \left( \wt{\Ric}_{b] c} \right)_\circ & = 0 \, ,  & \mbox{where $\wt{\kappa}_{a} = \wt{g}_{a b} \wt{k}^b$ for some $\wt{k} \in \Gamma(\wt{K})$.}
		\end{align}
\end{defn}

\begin{rem}
Clearly, if $(\wt{\mc{M}},\wt{\mbf{c}},\wt{N},\wt{K})$ is an almost Robinson geometry that admits a pure radiation metric $\wt{g}$ in $\wt{\mbf{c}}$, then $\wt{g}$ is in particular (weakly) half-Einstein.
\end{rem}

We may, however, be interested in metrics that are defined only off certain hypersurfaces. Thus, motivated by the notion of \emph{almost pseudo-Riemannian structure} \cite{Curry2018}, we make the following definition:
\begin{defn}
	Let $\wt{\sigma} \in \Gamma(\wt{\mc{E}}[1])$  with zero set $\wt{\mc{Z}}=\{ \wt{p} \in \wt{\mc{M}} | \wt{\sigma}(\wt{p}) = 0\}$. We say that $\wt{\sigma}$ is an \emph{almost Lorentzian scale} if $\wt{\nabla} \wt{\sigma} \neq 0$ on $\wt{\mc{Z}}$ where $\wt{\nabla}$ is the Levi-Civita connection of some (and thus any) metric in $\wt{\mbf{c}}$. 
\end{defn}
It is easy to check that this definition does not depend on the choice of metric. The density $\wt{\sigma}$ defines a metric $\wt{g} = \wt{\sigma}^{-2} \wt{\bm{g}}$ in $\wt{\mbf{c}}$, but regular only off $\wt{\mc{Z}}$ --- there is clearly a sign ambiguity since $-\wt{\sigma}$ is also an almost Lorentzian scale that defines the same metric, but this will have no bearing on the subsequent discussion.

An important class of almost Lorentzian scales was introduced in \cite{Gover2005a} --- see also \cite{LeBrun1985}: an \emph{almost Einstein scale} is an almost Lorentzian scale $\wt{\sigma}$ that satisfies the conformally invariant equation
\begin{align}
	\left( \wt{\nabla}_{a} \wt{\nabla}_{b} \wt{\sigma} + \wt{\Rho}_{a b} \wt{\sigma} \right)_\circ = 0 \, . \label{eq:alEinstein}		
\end{align}
If $\wt{\sigma}$ has empty zero set $\wt{\mc{Z}}$, it is then referred to as an \emph{Einstein scale}, and it defines a (global) Einstein metric. Otherwise, if $\wt{\mc{Z}}$ is non-empty, then off $\wt{\mc{Z}}$, the metric $\wt{g} = \wt{\sigma}^{-2} \wt{\bm{g}}$ is Einstein. This is convenient since we may then regard $\wt{\mc{Z}}$ as the \emph{conformal infinity} of some Lorentzian Einstein manifold --- see \cite{Curry2018}.

This can also be applied to the metrics introduced in Definition \ref{defn:redEins}:
\begin{defn}\label{defn:aESc}
	Let $\wt{\sigma}$ be an almost Lorentzian scale on an almost Robinson geometry $(\wt{\mc{M}},\wt{\mbf{c}},\wt{N},\wt{K})$. We say that $\wt{\sigma}$ is 
	\begin{itemize}
		\item an \emph{almost weakly half-Einstein scale} if it satisfies
		\begin{subequations}\label{eq:alwkEins}
			\begin{align}\label{eq:alwkEins1}
			&	\left( \wt{\nabla}_{a} \wt{\nabla}_{b} \wt{\sigma} + \wt{\Rho}_{a b} \wt{\sigma} \right)_\circ = \tfrac{1}{n} \wt{\Phi}_{a b} \wt{\sigma} \, , 
		\end{align}
		for some trace-free symmetric tensor $\wt{\Phi}_{a b}$ satisfying
		\begin{align}\label{eq:RiccwkEins}
			\wt{\Phi}_{a b} \wt{v}^a \wt{v}^b & = 0 \, , & \mbox{for any $\wt{v}^a \in \Gamma (\wt{N})$;}
		\end{align}
		\end{subequations}
		\item an \emph{almost half-Einstein scale} if it is an almost weakly half-Einstein scale, and
		\begin{align}\label{eq:alcstRicSc}
			& \wt{\Phi}_{a}{}^{b} \wt{\nabla}_{b} \wt{\sigma} - \tfrac{1}{n} \wt{\sigma} \wt{\nabla}_{b} \wt{\Phi}_{a}{}^{b} = 0 \, .
		\end{align}
	\end{itemize} 
\end{defn}

Similarly, we introduce the following:
\begin{defn}\label{defn:puradSc}
	Let $\wt{\sigma}$ be an almost Lorentzian scale on an optical geometry $(\wt{\mc{M}},\wt{\mbf{c}},\wt{K})$. We say that $\wt{\sigma}$ is an \emph{almost pure radiation scale} if it satisfies
	\begin{align}
		&	\left( \wt{\nabla}_{a} \wt{\nabla}_{b} \wt{\sigma} + \wt{\Rho}_{a b} \wt{\sigma} \right)_\circ = \tfrac{1}{n} \wt{\Phi}_{a b} \wt{\sigma} \, , &
		& \wt{\Phi}_{a}{}^{b} \wt{\nabla}_{b} \wt{\sigma} - \tfrac{1}{n} \wt{\sigma} \wt{\nabla}_{b} \wt{\Phi}_{a}{}^{b} = 0 \, ,
	\end{align}
where
\begin{align}\label{eq:purad}
	& \wt{\Phi}_{a b} \wt{v}^{b}  = 0 \, , & \mbox{for any $\wt{v} \in \Gamma(\wt{K}^\perp)$.}
\end{align}
\end{defn}

All the conditions given in the above definition are conformally invariant.
\begin{rem}
	Note that condition \eqref{eq:purad} is equivalent to
	\begin{align}\label{eq:purad_alt}
		\wt{\Phi}_{a b} & = \wt{\bm{\Phi}} \wt{\bm{\kappa}}_{a} \wt{\bm{\kappa}}_{b} \, , & \mbox{for some $\wt{\bm{\Phi}} \in \Gamma(\wt{\mc{E}}[-4])$,}
	\end{align}
	where $\wt{\bm{\kappa}}_{a} \in \Gamma(\Ann(\wt{K}^\perp \otimes \wt{\mc{E}}[2]))$.
\end{rem}
The relation between Definition \ref{defn:redEins} and Definition \ref{defn:aESc} is made precise below:
\begin{prop}\label{prop:scale2metric}
	Let $(\wt{\mc{M}},\wt{\mbf{c}},\wt{N},\wt{K})$ be an almost Robinson geometry, and let $\wt{\sigma} \in \Gamma(\mc{E}[1])$ so that $\wt{g} = \wt{\sigma}^{-2} \wt{\bm{g}} \in \wt{\mbf{c}}$ is a smooth metric off the zero set of $\wt{\sigma}$. Then
	\begin{enumerate}
		\item $\wt{\sigma}$ is an almost weakly half-Einstein scale if and only if $\wt{g}$ is a weakly half-Einstein metric;
		\item $\wt{\sigma}$ is an almost half-Einstein scale if and only if $\wt{g}$ is a half-Einstein metric.
	\end{enumerate}
\end{prop}

\begin{proof}
	Throughout $\wt{\nabla}{}^{(\wt{g})}$ will denote the Levi-Civita connection preserving $\wt{g}$, so that for any other connection $\wt{\nabla}$, $\wt{\nabla} = \wt{\nabla}{}^{(\wt{g})} + \Upsilon$ where $\wt{\Upsilon} = \wt{\sigma}^{-1} \wt{\nabla} \wt{\sigma}$. We first prove the implication $\Rightarrow$ for each of the cases, working off the zero set of $\wt{\sigma}$:
	\begin{enumerate}
		\item Suppose $\wt{\sigma}$ satisfies \eqref{eq:alwkEins}. Multiplying \eqref{eq:alwkEins} through by $\wt{\sigma}^{-1}$ and using the transformation rule for the Schouten tensor \eqref{eq:Rho_transf}, we find $\left( \wt{\Ric}{}^{(\wt{g})}_{a b} \right)_\circ = \wt{\Phi}_{a b}$, and comparing \eqref{eq:wk_hlf_Einstein} with \eqref{eq:RiccwkEins} tells us that $\wt{g}$ is a weakly half-Einstein metric.
		\item Suppose that $\wt{\sigma}$ satisfies \eqref{eq:alcstRicSc} in addition to \eqref{eq:alwkEins}. By case (1), we already know $\wt{g}$ is weakly half-Einstein metric. We proceed to check that its scalar curvature is constant. Multiplying the left-hand side of \eqref{eq:alcstRicSc} through by $\wt{\sigma}^{-1}$ gives the transformation rule
		\begin{align*}
			\wt{\nabla}{}^{(\wt{g})}_{b} \wt{\Phi}_{a}{}^{b} & = \wt{\nabla}_{b} \wt{\Phi}_{a}{}^{b} + n \wt{\Upsilon}{}_{b} \wt{\Phi}_{a}{}^{b} \, .
		\end{align*}
		The right-hand side is zero by virtue of \eqref{eq:alcstRicSc}. On the other hand, we have $(\wt{\Ric}{}^{(\wt{g})}_{a b})_\circ = \wt{\Phi}_{a b}$, and the contracted Bianchi identity tells us that
		\begin{align*}
			\wt{\nabla}_{a}\wt{\Sc}{}^{(\wt{g})} & = \tfrac{2(n+2)}{n} \wt{\nabla}{}^{(\wt{g})}_{b} \wt{\Phi}_{a}{}^{b} = 0 \, ,
		\end{align*}
		i.e.\ $\wt{\Sc}{}^{(\wt{g})}$ is constant, which makes $\wt{g}$ a half-Einstein metric.
	\end{enumerate}
	
	For the converse direction, if we assume that the metric $\wt{g}$ defined by $\wt{\sigma}$ is weakly half-Einstein off the zero set $\wt{\mc{Z}}$ of $\wt{\sigma}$, then choosing $\wt{\nabla}$ to preserve $\wt{\sigma}$ tells us that \eqref{eq:alwkEins} is satisfied off $\wt{\mc{Z}}$. But \eqref{eq:alwkEins} clearly holds on $\wt{\mc{Z}}$ too. Thus $\wt{\sigma}$ is an almost weakly half-Einstein scale. The almost half-Einstein case is similar.
\end{proof}

In a completely analogous way, we prove:
\begin{prop}\label{prop:scale2metricB}
	Let $(\wt{\mc{M}},\wt{\mbf{c}},\wt{K})$ be an optical geometry, and let $\wt{\sigma} \in \Gamma(\mc{E}[1])$ so that $\wt{g} = \wt{\sigma}^{-2} \wt{\bm{g}} \in \wt{\mbf{c}}$ is a smooth metric off the zero set of $\wt{\sigma}$. Then $\wt{\sigma}$ is an almost pure radiation scale if and only if $\wt{g}$ is a pure radiation metric.
\end{prop}

\begin{proof}
	Following from the proof of Proposition \ref{prop:scale2metric}, suppose that $\wt{\Phi}_{a b}$ satisfies \eqref{eq:purad}, in addition to $\wt{\sigma}$ being a solution to \eqref{eq:alcstRicSc}. Then clearly \eqref{eq:puradvac} follows from \eqref{eq:purad}. The converse works in the same way as in the proof of Proposition \ref{prop:scale2metric}.
\end{proof}

\section{The Fefferman construction}\label{sec:Fefferman}
Among the twist-induced nearly Robinson geometries with non-shearing congruences introduced in Section \ref{sec:prefered} can be found, as a special case, Fefferman's conformal structure, constructed in a \emph{canonical} manner from a given almost CR structure. In the present section we show how the non-involutivity of the CR structure in fact gives rise to a one-parameter family of conformal structures. We also highlight the existence of distinguished spinor fields on these geometries, before providing a characterisation of our generalisations. 
\subsection{Fefferman spaces}
Let $(\mc{M},H,J)$ be almost CR manifold of dimension $2m+1$. Henceforth, we assume that the Levi form of $(H,J)$ is positive definite, although this requirement may be relaxed with caution. We adapt the approach of \cite{Cap2008} and \cite{Leitner2007}. The bundle $\mc{E}(-1,0)$ with its zero section removed is a principal bundle with structure group $\C^*$, and taking its quotient by the natural $\R_{>0}$-action yields a circle bundle $\wt{\mc{M}}$ over $\mc{M}$. The projection from $\mc{E}(-1,0)$ with its zero section removed to $\wt{\mc{M}}$ sends a nowhere vanishing density $\tau$ of weight $(-1,0)$ to an equivalence class $[\tau]$, where $\tau, \tau' \in [\tau]$ if and only if $\tau' = \varrho \tau$ for some positive real-valued function $\varrho$ on $\mc{M}$. Let us fix a pseudo-Hermitian structure $\theta$. We then have a natural identification of sections of $\wt{\mc{M}}$ with nowhere vanishing densities $\tau$ of weight $(-1,0)$ satisfying $\theta = \tau \ol{\tau} \bm{\theta}$, and we may define a fibre coordinate $\phi \in [ - \pi, \pi)$ such that $\e^{\i \phi} \tau$ is a nowhere vanishing section of $\wt{\mc{M}} \rightarrow \mc{M}$. For each $\alpha \in \R$, we can then introduce the Lorentzian metric on $\wt{\mc{M}}$
	\begin{align}\label{eq:Feff_metric}
		\wt{g}_{\theta}^{(\alpha)} & = 4 \theta \odot \left( \d \phi + \tfrac{\i}{2} \left( \sigma^{-1} \nabla \sigma - \ol{\sigma}^{-1} \nabla \ol{\sigma} \right) - \left( \tfrac{1}{m+2} \Rho + \tfrac{\alpha}{2m(m+1)}  \| \Nh \|^2 \right) \theta \right) + h \, ,
	\end{align}
where $\sigma := \tau^{-1}$, $h$ is (the degenerate metric induced by) the Levi form of $\theta$, $\Rho$ is the Webster--Schouten scalar of $\theta$ and $\| \Nh \|^2$ is the squared norm of the Nijenhuis tensor of $\theta$ with respect to $h$. Strictly, the tensors involved in the definition of $\wt{g}_{\theta}^{(\alpha)}$ should be viewed as pullbacks from $\mc{M}$ to $\wt{\mc{M}}$, but for readability, we have omitted the pullback symbols.

\begin{defn}\label{defn-alpha-Fefferman}
We shall refer to the metric defined in \eqref{eq:Feff_metric} as the \emph{$\alpha$-Fefferman metric} associated to $\theta$.  A $1$-Fefferman metric shall be referred to simply as a \emph{Fefferman metric}.
\end{defn}
One can conveniently eliminate $\ol{\sigma}$ by noting that since $\sigma$ determines $\nabla$, we have that $\nabla(\sigma \ol{\sigma}) =0$, i.e.\ $\sigma^{-1} \nabla \sigma = - \ol{\sigma}^{-1} \nabla \ol{\sigma}$. Using \eqref{eq:conn1_density}, we can also easily recover the more familiar form of the Fefferman metric in terms of the Webster connection one-form:
\begin{align*}
	\wt{g}_{\theta}^{(\alpha)} & = 4 \theta \odot \left( \d \phi + \tfrac{1}{m+2} \left( \i \Gamma_{\alpha}{}^{\alpha} - \tfrac{\i}{2} h^{\alpha \bar{\beta}} \d h_{\alpha \bar{\beta}} - \left( \Rho + \tfrac{\alpha(m+2)}{2m(m+1)} \| \Nh \|^2 \right) \theta \right) \right) + h \, .
\end{align*}
Note the following:
\begin{itemize}
	\item Since, for any other section $\tau' = \e^{\i \mr{\phi}} \tau$, the corresponding change of coordinate is $\phi' = \phi - \mr{\phi}$, the expression \eqref{eq:Feff_metric} does not depend on the choice of trivialisation $\sigma$.
	\item Under a change of contact form $\wh{\theta} = \e^{\varphi} \theta$ for smooth function $\varphi$ on $\mc{M}$, the metric transforms conformally as $\wt{g}_{\wh{\theta}}^{(\alpha)} = \e^{\varphi} \wt{g}_{\theta}^{(\alpha)}$.
\end{itemize}
We have therefore constructed a conformal class $\wt{\mbf{c}}^{(\alpha)}$ of Lorentzian metrics in $\wt{\mc{M}}$ that contains $\alpha$-Fefferman metrics. In addition, $\wt{\mbf{c}}^{(\alpha)}$ admits a canonical \emph{conformal Killing field} $\wt{k}$, i.e.\ $\mathsterling_{\wt{k}} \wt{\bm{g}}^{(\alpha)} = 0$, namely, the generator $\wt{k} = \tfrac{\partial}{\partial \phi}$ of the fibres of $\wt{\mc{M}} \rightarrow \mc{M}$. Here, $\wt{\bm{g}}^{(\alpha)}$ is the conformal metric of $\wt{\mbf{c}}^{(\alpha)}$. In fact, $\wt{k}$ is Killing for each of the Fefferman metrics in $\wt{\mbf{c}}^{(\alpha)}$, i.e.\ $\mathsterling_{\wt{k}} \wt{g}_{\theta}^{(\alpha)} = 0$ for any contact form $\theta$ for $(H,J)$.

\begin{defn}\label{defn-alpha-Fefferman-space}
	We shall refer to $(\wt{\mc{M}},\wt{\mbf{c}}^{(\alpha)},\wt{k}) \longrightarrow (\mc{M},H,J)$ as the \emph{$\alpha$-Fefferman space} of $(\mc{M},H,J)$. A $1$-Fefferman space shall be referred to simply as a \emph{Fefferman space}.
\end{defn}

\begin{rem}
	When $\alpha = 0$, the conformal structure $\wt{\mbf{c}}^{(0)}$ is identical to the one introduced in \cite{Leitner2007}. However, as will become apparent in Section \ref{sec:al_Lor_sc}, the parameter $\alpha=1$ is more useful --- see Corollary \ref{cor:Einstein_Fefferman}.
	
	When $\| \Nh \|^2 = 0$, in which case $(H,J)$ is involutive, the parameter $\alpha$ becomes irrelevant, and we will simply talk of Fefferman metrics or conformal structure.
\end{rem}

\begin{rem}
	These definitions can also be generalised to conformal structure of  signature $(p+1,q+1)$ with $p,q$ even, but if $pq \neq 0$, $\| \Nh \|^2 = 0$ no longer implies the involutivity of $(H,J)$.
\end{rem}

\subsection{Distinguished forms and spinor fields}
	Recall that the Fefferman space admits a \emph{conformal spin structure} regardless of the integrability of the underlying almost CR manifold --- see e.g.\ \cite[Theorem 2.4]{Cap2008} --- and clearly, this result still applies to $\alpha$-Fefferman spaces as defined above.
	
	It was already demonstrated in the involutive case that on a Fefferman space there exists a spinor field $\wt{\eta}$, defined up to a constant factor, that satisfies the so-called \emph{twistor equation} \cite{Lewandowski1991,Baum1999,Cap2008}. This spinor field is \emph{pure}, i.e.\ it annihilates a totally null complex $(m+1)$-plane distribution, which, in the present context, is none other than the Robinson structure arising from the CR structure. These features subsist in the more general case, except that for the spinor field $\wt{\eta}$ to be a solution to the twistor equation, the almost CR structure must be involutive as we shall soon see.
	
	Before dealing with spinors, we must highlight the existence of a tautological weighted $(m+1)$-form associated to an $\alpha$-Fefferman space:
\begin{lem}\label{lem:pure_sp}
	Let $(\wt{\mc{M}},\wt{\mbf{c}}^{(\alpha)},\wt{k}) \accentset{\varpi}{\longrightarrow} (\mc{M},H,J)$ be an $\alpha$-Fefferman space over an almost CR manifold. Then there exists a tautological complex-valued (self-dual) $(m+1)$-form $\wt{\bm{\eta}}$ of weight $m+2$ on $\wt{\mc{M}}$ with the following properties: for any nowhere vanishing density $\sigma \in \Gamma(\mc{E}(1,0))$ with associated contact form $\theta$, $\alpha$-Fefferman conformal scale $\wt{\sigma}_\theta$ and fibre coordinate $\phi$, the value of $\wt{\bm{\eta}}$ at any point $\wt{p} = (\varpi(\wt{p}),\e^{\i \phi(\wt{p})} \sigma^{-1} (\varpi(\wt{p})))$ of $\wt{\mc{M}}$ is given by
	\begin{align}\label{eq:tautoform}
		\wt{\bm{\eta}}_{\wt{p}} = \wt{\sigma}_{\theta}^{m+2} (\wt{p}) \e^{(m+2)\i \phi(\wt{p})} \varpi^*(\zeta_{\varpi(\wt{p})}) \, ,
	\end{align}
	where $\zeta$ is the volume-normalised section of $\bigwedge^{m+1} \Ann(H^{0,1})$ defined by $\sigma^{-m-2}$.
\end{lem}

\begin{proof}
	It suffices to check that the form \eqref{eq:tautoform} is well-defined under changes of trivialisations and contact forms, which we leave to the reader. That $\wt{\bm{\eta}}$ is self-dual follows from the fact that it is a totally null indecomposable $(m+1)$-form and our choice of orientation.
\end{proof}

In the following argument, $\wt{\mc{S}}^+$ will denote the bundle of positive spinors and $\bigwedge^{m+1}_+ {}^\C T^* \wt{\mc{M}}$ the bundle of (complex-valued) self-dual $(m+1)$-forms.  In addition, $\wt{\gamma}_{a}$ will denote the weighted generators of the Clifford algebra of an $\alpha$-Fefferman space $(\wt{\mc{M}},\wt{\mbf{c}}^{(\alpha)},\wt{k})$ with the convention that $\wt{\gamma}_{(a} \wt{\gamma}_{b)} = - \wt{\bm{g}}_{a b}$. We shall not make any notational difference between the Levi-Civita connection and the induced connection on the spinor bundles.
\begin{prop}\label{prop:non_tw-sp}
	Let $(\wt{\mc{M}},\wt{\mbf{c}}^{(\alpha)},\wt{k}) \longrightarrow (\mc{M},H,J)$ be an $\alpha$-Fefferman space over an almost CR manifold. Then there exists a pure spinor field $\wt{\eta}$ that satisfies the conformally invariant differential equation
	\begin{align}\label{eq:non_tw-sp}
		\wt{\nabla}_{a} \wt{\eta} + \tfrac{1}{2(m+1)} \wt{\gamma}_{a} \wt{\slashed{\nabla}} \wt{\eta} & =  \tfrac{\alpha \i}{32m} \| \wt{\Weyl}(\wt{k}) \|^2 \wt{\bm{\kappa}}_{a}  \wt{\eta} + \tfrac{\i}{8} \wt{k}^{d} \wt{\Weyl}_{d a b c} \wt{\gamma}^{b} \wt{\gamma}^{c} \wt{\eta} \, ,
	\end{align}
	where $\wt{\Weyl}_{a b c d}$ is the Weyl tensor of $\wt{\mbf{c}}^{(\alpha)}$, $\| \wt{\Weyl}(\wt{k}) \|^2 := \wt{k}^{a} \wt{\Weyl}_{a b c d} \wt{k}^e \wt{\Weyl}_{e}{}^{b c d}$, $\wt{\bm{\kappa}}_a := \wt{\bm{g}}(\wt{k},\cdot)$ and $\wt{\slashed{\nabla}} \wt{\eta} := \wt{\gamma}^{a} \wt{\nabla}_{a} \wt{\eta}$.
	
	In particular, $(\mc{M},H,J)$ is involutive if and only if $\wt{\eta}$ is a twistor spinor, i.e.\ $\wt{\eta}$ satisfies
	\begin{align*}
		\wt{\nabla}_{a} \wt{\eta} + \tfrac{1}{2(m+1)} \wt{\gamma}_{a} \wt{\slashed{\nabla}} \wt{\eta} & =  0 .
	\end{align*}
\end{prop}

\begin{proof}
	 We adapt Cartan's classical results \cite{Cartan1967} --- see also \cite[Appendix B]{Penrose1986} and \cite[Section 4.2]{Fino2023} --- to the conformal setting by introducing a bundle map $\bm{\beta} : \wt{\mc{S}}^+ \times \wt{\mc{S}}^+ \rightarrow \bigwedge^{m+1}_+ {}^\C T^* \wt{\mc{M}} \otimes \wt{\mc{E}}[m+2]$. This map is in fact symmetric and provides a bundle isomorphism between the Cartan product of $\wt{\mc{S}}^+$ and $\bigwedge^{m+1}_+ T^* \wt{\mc{M}} \otimes \wt{\mc{E}}[m+2]$. It also restricts to a bundle isomorphism between positive pure spinors up to sign and self-dual totally null $(m+1)$-forms of conformal weight $m+2$ --- see also \cite[Proposition 5.3]{Fino2023}. Applying this to the weighted form $\wt{\bm{\eta}}$ defined in Lemma \ref{lem:pure_sp}, we obtain a pure spinor field $\wt{\zeta}$, unique up to sign, such that $\bm{\beta} (\wt{\eta},\wt{\eta}) = \wt{\bm{\eta}}$ --- we can essentially think of $\wt{\eta}$ as a choice of square root of $\wt{\bm{\eta}}$.
	
	To derive the differential condition satisfied by $\wt{\eta}$, it suffices to work with a choice of trivialisation of $\wt{\mc{M}} \accentset{\varpi}{\longrightarrow} \mc{M}$, which singles out a contact form $\theta$, a fibre coordinate $\phi$ and an $\alpha$-Fefferman metric $\wt{g}^{(\alpha)}_{\theta}$ in $\wt{\mbf{c}}^{(\alpha)}$. Using a unitary adapted coframe $(\theta,\theta^\alpha, \ol{\theta}{}^{\bar{\alpha}})$, we can express the tautological $(m+1)$-form \eqref{eq:tautoform} as
		\begin{align*}
			\wt{\sigma}_{\theta}^{-m-2} \wt{\bm{\eta}} = \e^{(m+2)\i \phi} \theta \wedge \theta^1 \wedge \ldots \wedge \theta^m \, ,
	\end{align*}
	where $\wt{\sigma}_{\theta}$ is the conformal scale associated to $\wt{g}^{(\alpha)}_{\theta}$, and this, in turn, yields an explicit expression for our spinor $\wt{\eta}$. We can then compute \eqref{eq:non_tw-sp} using the Levi-Civita connection preserving $\wt{\sigma}_{\theta}$ and the spinor frame induced by the pullback of $(\theta,\theta^\alpha, \ol{\theta}{}^{\bar{\alpha}})$. Omitting the details of this rather lengthy computation, we note that the curvature terms appearing in that equation encode the Nijenhuis tensor of $(\mc{M},H,J)$: more specifically, we have $\| \wt{\Weyl}(\wt{k}) \|^2 = 8 \| \Nh \|^2$, while $\wt{k}^{d} \wt{\Weyl}_{d a b c}$ depends on both $\Nh_{\alpha \beta \gamma}$ and $\nabla^{\gamma} \Nh_{\alpha \beta \gamma}$ --- see especially \cite[equations (A.4) and (A.6)]{TaghaviChabert2022}. Conformal invariance of \eqref{eq:non_tw-sp} follows from the fact that the left-hand side is simply the conformally invariant twistor operator on $\wt{\eta}$ while the right-hand side consists of tensor products of conformally invariant quantities. This completes the proof.
\end{proof}

\begin{rem}
	Note that $\slashed{\nabla} \wt{\eta}$ depends on the choice of metric in $\wt{\mbf{c}}^{(\alpha)}$. In fact, for each choice of $\alpha$-Fefferman metric, it can be shown that $\slashed{\nabla} \wt{\eta} = (m+1) \i\wt{\ell}^{a} \wt{\gamma}_{a} \wt{\eta}$, where $\wt{\ell}^{a}$ is the vector field dual to $\wt{k}^{a}$ introduced in Section \ref{sec:prefered}. This means that $\slashed{\nabla} \wt{\eta}$ is pure and its associated totally null $(m+1)$-plane distribution intersects that of $\wt{\eta}$ in an $m$-plane distribution. If $\wt{\eta}$ is a twistor-spinor, then, by \cite[Proposition 5.20]{Taghavi-Chabert2016} its associated distribution is involutive, and so is the almost CR structure.
\end{rem}

\begin{rem}
	The spinor field $\wt{\eta}$, which we first introduced as the square root of the tautological weighted $(m+1)$-form, is defined up to sign. But other normalisations are possible: For instance, we can rescale $\wt{\eta}$ by a non-zero constant factor in such a way that its pairing with its \emph{charge conjugate} is precisely $\wt{k}$. In this case, the remaining freedom consists in rescalings by constant phases.
\end{rem}

\subsection{Characterisations}
We first remind the reader of the well-known characterisation of Fefferman spaces for CR manifolds:
\begin{thm}[\cite{Graham1987,Cap2008}]\label{thm:Fefferman-CR}
	Let $(\wt{\mc{M}},\wt{\mbf{c}})$ be an oriented and time-oriented Lorentzian conformal manifold of dimension $n+2=2m+2$. Then $(\wt{\mc{M}},\wt{\mbf{c}})$ is locally conformally isometric to the Fefferman space of a CR manifold if and only if it admits a null conformal Killing field $\wt{k}$ and the following integrability conditions are satisfied:
	\begin{align*}
		& \tfrac{1}{n^2} (\wt{\nabla}_{a} \wt{k}^{a} )^2 - \wt{\Rho}_{a b} \wt{k}^a \wt{k}^b - \tfrac{1}{n} \wt{k}^a \wt{\nabla}_a \wt{\nabla}_b \wt{k}^b <  0\, , 
		&& \wt{k}^a \wt{\Weyl}_{a b c d}  = 0 \, , 
		&& \wt{k}^a \wt{\Cot}_{a b c}  = 0 \, , 
	\end{align*}
	where $\wt{\nabla}$ is the Levi-Civita connection of any metric in $\wt{\mbf{c}}$ with Schouten tensor $\wt{\Rho}_{a b}$, Cotton tensor $\wt{\Cot}_{a b c}$ and Weyl tensor $\wt{\Weyl}_{a b c d}$.
\end{thm}
We next state characterisations of $\alpha$-Fefferman spaces for almost CR manifolds. The proof has been relegated to Appendix \ref{app:proofchi}.
\begin{thm}\label{thm:Fefferman-CRA}
	Let $(\wt{\mc{M}},\wt{\mbf{c}})$ be an oriented and time-oriented Lorentzian conformal manifold of dimension $n+2=2m+2$. Then, for any $\alpha \in \R$, $(\wt{\mc{M}},\wt{\mbf{c}})$ is locally conformally isometric to an $\alpha$-Fefferman space of an almost CR structure if and only if it admits a null conformal Killing field $\wt{k}$ and the following integrability conditions are satisfied:
	\begin{subequations}\label{eq-ICaFeff}
	\begin{align}
		& \tfrac{1}{n^2} (\wt{\nabla}_{a} \wt{k}^{a} )^2 - \wt{\Rho}_{a b} \wt{k}^a \wt{k}^b - \tfrac{1}{n} \wt{k}^a \wt{\nabla}_a \wt{\nabla}_b \wt{k}^b <  0\, , \label{eq-Rho_sc} \\
		& \wt{k}^a \wt{\Weyl}_{a b c d} \wt{k}^d  = \tfrac{\alpha - 1}{8(2m+1)}\wt{\bm{\kappa}}_{b} \wt{\bm{\kappa}}_{c} \| \wt{\Weyl}(\wt{k}) \|^2 \, , \label{eq-Wkk} \\
		& \wt{k}^a \wt{\Cot}_{a b c} \wt{k}^c  = \tfrac{\alpha - 1}{16(2m+1)} \wt{\bm{\kappa}}_{b} \wt{k}^{c} \wt{\nabla}_{c} \| \wt{\Weyl}(\wt{k}) \|^2  \, ,  \label{eq-Ykk} \\
		&
		\wt{\Weyl}_{a b}{}^{c d} \wt{\bm{\tau}}_{c d} - 2 \wt{k}^c \wt{\Cot}_{c a b} - \tfrac{1}{2} \left( \wt{\bm{\tau}}_{c[a} \wt{k}^{d} \wt{\Weyl}_{b] d}{}^{e f} \wt{\Weyl}_{e f g}{}^{c} \wt{k}^g  + \wt{\bm{\kappa}}_{[a} \wt{k}^{c} \wt{\Weyl}_{b] c}{}^{d e} \wt{\Cot}_{f d e} \wt{k}^f \right) \nonumber  \\
		& \qquad \qquad \qquad = \tfrac{1}{4m} \left( 1 + \tfrac{2m(m+1)}{2m+1} (\alpha - 1) \right) \wt{\nabla}_{[a}  \left( \wt{\bm{\kappa}}_{b]} \| \wt{\Weyl}(\wt{k}) \|^2 \right) \, ,  \label{eq-int_cond} 
	\end{align}
	\end{subequations}
	where $\wt{\nabla}$ is the Levi--Civita connection of any choice of metric in $\wt{\mbf{c}}$ with Schouten tensor $\wt{\Rho}_{a b}$, Cotton tensor $\wt{\Cot}_{a b c}$ and Weyl tensor $\wt{\Weyl}_{a b c d}$,
	\begin{align*}
		\| \wt{\Weyl}(\wt{k}) \|^2 & := \wt{k}^{a} \wt{\Weyl}_{a b c d} \wt{k}^e \wt{\Weyl}_{e}{}^{b c d}  \, ,
	\end{align*}
	and $\wt{\bm{\kappa}}_{a} = \wt{\bm{g}}_{a b} \wt{k}^b$, $\wt{\bm{\tau}}_{a b} = \wt{\nabla}_{[a} \wt{\bm{\kappa}}_{b]}$.
\end{thm}

\begin{rem}
	Conditions \eqref{eq-Rho_sc}, \eqref{eq-Wkk}, \eqref{eq-Ykk} and \eqref{eq-int_cond} are all conformally invariant, as can be seen from the transformations \eqref{eq:Cot_transf} and $\widehat{\wt{\bm{\tau}}}_{a b} = \wt{\bm{\tau}}_{a b} + 2 \, \wt{\Upsilon}_{[a} \wt{\bm{\kappa}}_{b]}$.
\end{rem}

\section{Perturbed Fefferman spaces}\label{sec:perturbed_Fefferman}
Henceforth, and in line with the definitions introduced earlier, a Fefferman space will refer to a $1$-Fefferman space, so that a Fefferman metric $\wt{g}_{\theta}^{(1)}$ will be denoted by $\wt{g}_{\theta}$, and a Fefferman conformal structure $\wt{\mbf{c}}^{(1)}$ by $\wt{\mbf{c}}$.

We introduce the formal definition of a perturbation of a Fefferman conformal structure, before providing characterisations in Sections \ref{sec:algsp} and \ref{sec:char_pert}. 

\subsection{General definitions and properties}
\begin{defn}\label{defn:perturbed_Fefferman}
	Let $(\wt{\mc{M}},\wt{\mbf{c}},\wt{k}) \longrightarrow (\mc{M},H,J)$ be a Fefferman space. Let $\wt{\xi}$ be a semi-basic one-form on $\wt{\mc{M}}$, that is, $\wt{k} \hook \wt{\xi} = 0$. Given a Fefferman metric $\wt{g}_{\theta}$ in $\wt{\mbf{c}}$, we define the \emph{Fefferman metric perturbed by $\wt{\xi}$} as the metric
	\begin{align*}
		\wt{g}_{\theta,\wt{\xi}} & = \wt{g}_{\theta} + 4 \theta \odot \wt{\xi} \, .
	\end{align*}
	This naturally extends to a conformal structure $\wt{\mbf{c}}_{\wt{\xi}}$, which we refer to as the \emph{Fefferman conformal structure perturbed by $\wt{\xi}$}, and we call $\wt{\xi}$ the \emph{perturbation one-form} of $(\wt{\mc{M}},\wt{\mbf{c}},\wt{k})$, and the triple $(\wt{\mc{M}},\wt{\mbf{c}}_{\wt{\xi}},\wt{k})$ as a \emph{perturbed Fefferman space}.
\end{defn}

Since $\wt{\xi}$ lives on a circle bundle, its components can be Fourier expanded, and in fact, the Fourier coefficients should be understood as trivialised densities on $\mc{M}$ as the lemma below makes clear. Example \ref{exa:CR_data} should clarify any ambiguity in the notation adopted.
\begin{lem}\label{lem:CR_data}
	Let $(\wt{\mc{M}},\wt{\mbf{c}}_{\wt{\xi}},\wt{k}) \longrightarrow (\mc{M},H,J)$ be a perturbed Fefferman space. For any subsets $\mc{I} \subset \Z$, $\mc{J} \subset \Z_{\geq 0}$, consider the tuple $\left( \bm{\xi}^{(i)}_{\alpha} , [\nabla, \bm{\xi}^{(j)}_{0}] \right)_{i \in \mc{I}, j \in \mc{J}}$ where:
	\begin{enumerate}
		\item  for $k \in \mc{I}$, $\bm{\xi}_{\alpha}^{(k)} \in \Gamma(\mc{E}_{\alpha}\left(\tfrac{k}{2},-\tfrac{k}{2}\right))$,
		\item  for $k \in \mc{J}$, $[\nabla, \bm{\xi}^{(k)}_0]$ denotes the equivalence class $(\nabla, \bm{\xi}_{0}^{(k)}) \sim (\wh{\nabla}, \wh{\bm{\xi}}_{0}^{(k)})$ where $\nabla, \wh{\nabla}$ are Webster connections related by $\wh{\nabla} = \nabla + \Upsilon$ for some exact one-form $\Upsilon$, and for  $k \in \mc{I}$, $\bm{\xi}_{0}^{(k)}$ and $\wh{\bm{\xi}}_{0}^{(k)}$ are sections of $\mc{E} \left(\tfrac{k}{2}-1,-\tfrac{k}{2}-1\right)$ related by
		\begin{align}\label{eq:wtd_Fourier}
			\wh{\bm{\xi}}{}^{(k)}_0 & = \bm{\xi}^{(k)}_0 - \i \bm{\xi}^{(k)}_{\alpha} \Upsilon^\alpha + \i \bm{\xi}^{(k)}_{\bar{\alpha}} \Upsilon^{\bar{\alpha}} \, ,
		\end{align}
		with the understanding that, for $i \in \mc{I}, j \in \mc{J}$, $\bm{\xi}_{\bar{\alpha}}^{(i)} = \ol{\bm{\xi}_{\alpha}^{(-i)}}$, $\bm{\xi}_{0}^{(j)} = \ol{\bm{\xi}_{0}^{(-j)}}$. In particular $\bm{\xi}_{0}^{(0)}$ is real-valued.
	\end{enumerate}
	Choose a nowhere vanishing density $\sigma \in \Gamma(\mc{E}(1,0))$ to trivialise $\wt{\mc{M}} \rightarrow \mc{M}$ with fibre coordinate $\phi$, and set
	\begin{align}\label{eq:Fourier_coef}
		\xi_{\alpha}^{(k)} & = \bm{\xi}_{\alpha}^{(k)} \sigma^{-\tfrac{k}{2}} \ol{\sigma}^{\tfrac{k}{2}} \, , & \xi_{\bar{\alpha}}^{(k)} & = \bm{\xi}_{\bar{\alpha}}^{(k)} \ol{\sigma}^{-\tfrac{k}{2}} \sigma^{\tfrac{k}{2}} \, , & \xi_{0}^{(k)} & =  \bm{\xi}_{0}^{(k)}  \sigma^{1-\tfrac{k}{2}} \ol{\sigma}^{1+\tfrac{k}{2}} \, .
	\end{align}
	Then, viewing
	\begin{align}\label{eq:xi_Fourier}
		\wt{\xi}_{\alpha} & = \sum_{k \in \mc{I}} \xi_{\alpha}^{(k)} \e^{k \i \phi} \, , &
		\wt{\xi}_{\bar{\alpha}} & = \sum_{k \in \mc{I}} \xi_{\bar{\alpha}}^{(-k)} \e^{-k \i \phi} \, , &
		\wt{\xi}_{0} & = \sum_{k \in - \mc{J} \cup \mc{J}} \xi_{0}^{(k)} \e^{k \i \phi} \, ,
	\end{align}
	as components with respect to some adapted coframe $(\theta, \theta^{\alpha}, \ol{\theta}{}^{\bar{\alpha}})$ with $\theta = (\sigma \ol{\sigma})^{-1} \bm{\theta}$, the semi-basic one-form
	\begin{align}\label{eq:xi_cmpnt}
		\wt{\xi} & = \wt{\xi}_{\alpha} \theta^{\alpha} + \wt{\xi}_{\bar{\alpha}} \ol{\theta}{}^{\bar{\alpha}} + \wt{\xi}_{0} \theta \, ,
	\end{align}
	is well-defined on $\wt{\mc{M}}$.
	
	Conversely, any semi-basic one-form on $\wt{\mc{M}}$ arises in this way.
\end{lem}

\begin{proof}
	Starting from the tuple $\left( \bm{\xi}^{(i)}_{\alpha} , [\nabla, \bm{\xi}^{(j)}_{0}] \right)_{i \in \mc{I}, j \in \mc{J}}$, we only need to check that \eqref{eq:xi_cmpnt} and \eqref{eq:xi_Fourier} are well-defined. First, $\wt{\xi}$ is real-valued by virtue of the reality conditions on $\left( \bm{\xi}^{(i)}_{\alpha} , [\nabla, \bm{\xi}^{(j)}_{0}] \right)_{i \in \mc{I}, j \in \mc{J}}$. Next, $\wt{\xi}$ does not depend on the choice of trivialisation, since under the transformation $\sigma' = \e^{-\i \mr{\phi}} \sigma$, $\phi'= \phi - \mr{\phi}$, the relations \eqref{eq:wtd_Fourier} tell us that the Fourier coefficients \eqref{eq:Fourier_coef} transform as
	\begin{align*}
		{\xi'}_{\alpha}^{(k)} & = \e^{k \i \mr{\phi}} \xi_{\alpha}^{(k)} \, , &
		{\xi'}_{0}^{(k)} & = \e^{k \i \mr{\phi}} \xi_{0}^{(k)} \, .
	\end{align*}
	Finally, $\wt{\xi}$ does not depend on the choice of adapted coframe as follows from \eqref{eq:wtd_Fourier}. The converse works analogously.
\end{proof}

\begin{defn}\label{defn:CR_data}
	We shall refer to the tuple $\left( \bm{\xi}^{(i)}_{\alpha} , [\nabla, \bm{\xi}^{(j)}_{0}] \right)_{i \in \mc{I}, j \in \mc{J}}$ given in Lemma \ref{lem:CR_data}, where $\mc{I} \subset \Z$, $\mc{J} \subset \Z_{\geq 0}$, as the \emph{CR data associated to the perturbation one-form $\wt{\xi}$}.
\end{defn}

\begin{exa}\label{exa:CR_data}
	A perturbation one-form $\wt{\xi}$ with CR data $\left( \bm{\xi}_{\alpha}^{(0)}, [\nabla, \bm{\xi}_{0}^{(0)}], \bm{\xi}_{0}^{(2k)}\right)_{k=1,\ldots,m+1}$ means that with a choice of trivialisation $\sigma$ of $\wt{\mc{M}}$ with fibre coordinate $\phi$, and of adapted coframe $(\theta,\theta^{\alpha}, \ol{\theta}{}^{\bar{\alpha}})$ with $\theta = (\sigma \ol{\sigma})^{-1} \bm{\theta}$, the one-form $\wt{\xi}$ is given by
	\begin{align*}
		\wt{\xi} = \xi_{\alpha}^{(0)}  \theta^{\alpha} + \xi_{\bar{\alpha}}^{(0)} \ol{\theta}{}^{\bar{\alpha}} + \sum_{k=-m-1}^{m+1} \xi_{0}^{(2k)} \e^{2k\i \phi}  \theta \, , 
	\end{align*}
	where the coefficients are related to the CR data via \eqref{eq:Fourier_coef}. Note that, for $k \neq 0$, $\wh{\bm{\xi}}_{0}^{(2k)} = {\bm{\xi}}_{0}^{(2k)}$ under a change of contact form, which justifies our notation for the tuple.
\end{exa}

\begin{rem}
	Let $(\wt{\mc{M}},\wt{\mbf{c}},\wt{k}) \rightarrow (\mc{M},H,J)$ be a Fefferman space. It is easy to check that two perturbations $\wt{\mbf{c}}_{\wt{\xi}}$ and $\wt{\mbf{c}}_{\wt{\xi}'}$ are locally conformally isometric if and only if their corresponding CR data $\left( \bm{\xi}^{(i)}_{\alpha} , [\nabla, \bm{\xi}^{(j)}_{0}] \right)_{i \in \mc{I}, j \in \mc{J}}$ and $\left( \bm{\xi}'{}^{(i)}_{\alpha} , [\nabla, \bm{\xi}'{}^{(j)}_{0}] \right)_{i \in \mc{I}', j \in \mc{J}'}$ are such that $\mc{I}' = \mc{I}$ and $\mc{J}' = \mc{J}$, and
	\begin{align*}
		\bm{\xi}'{}^{(0)}_{\alpha} & = 	\bm{\xi}^{(0)}_{\alpha} + \nabla_{\alpha} \mr{\phi} \, , & [\nabla, \bm{\xi}'{}^{(0)}_{0} ]& = 	[\nabla, \bm{\xi}^{(0)}_{0} + \nabla_{0} \mr{\phi} ] \, , \\
		\bm{\xi}'{}^{(k)}_{\alpha}  & = 	\e^{k \i \mr{\phi}} \bm{\xi}^{(k)}_{\alpha}  \, , & [\nabla, \bm{\xi}'{}^{(k)}_{0}] & = [\nabla, \e^{k \i \mr{\phi}} \bm{\xi}^{(k)}_{0} ] \, , & k \neq 0 \, ,
	\end{align*}
for some smooth function $\mr{\phi}$ on $(\mc{M},H,J)$. In the special case where $\wt{\xi}'=0$, we conclude that $\wt{\mbf{c}}_{\wt{\xi}}$ and $\wt{\mbf{c}}$ are locally conformally isometric if and only if $\wt{\xi}$ is closed (i.e. locally exact), and one can use the criterion of Lemma \ref{lem:closedg} on $\xi_{\alpha} = \xi_{\alpha}^{(0)}$ and $\xi_{0} = \xi_{0}^{(0)}$ to verify this property. See also \cite{Leitner2010}. One can extend these arguments as done in \cite{Graham1987} to verify that an obstruction to a global conformal isometry between two perturbed Fefferman spaces lies in the cohomology group $H^1(\mc{M},S^1)$.
\end{rem}

Finally, we state the next self-evident result without proof.
\begin{lem}
	Any perturbed Fefferman space is a twist-induced nearly Robinson geometry with non-shearing congruence.
\end{lem}
In the next section, we investigate under which conditions the converse is true.

\subsection{Relation to twist-induced nearly Robinson geometries}\label{sec:algsp}
Consider a $(2m+2)$-dimen-sional twist-induced nearly Robinson geometry $(\wt{\mc{M}},\wt{\mbf{c}},\wt{N},\wt{k}) \longrightarrow (\mc{M},H,J)$ with non-shearing congruence $\wt{\mc{K}}$. We refer the reader to Sections \ref{sec:optical} and \ref{sec:Robinson} for the general setup and notation concerning these geometries, and more particularly Section \ref{sec:prefered}. Since a perturbed Fefferman metric is a particular case of such an optical geometry, it will be convenient to write
\begin{align}\label{eq:metNSFeff}
	\wt{g}_{\theta,\wt{\xi}} & = 4 \theta \odot \wt{\lambda} + h \, , &
	\wt{\lambda} & := \wt{\omega}_{\theta} - \left( \tfrac{1}{m+2}\Rho + \tfrac{1}{2m(m+1)} \| \Nh \|^2 \right) \theta + \wt{\xi} \, ,
\end{align}
where $\wt{\omega}_{\theta}$ is the induced Webster connection one-form on $\wt{\mc{M}}$ with Webster--Schouten tensor $\Rho$ and Nijenhuis tensor $\Nh_{\alpha \beta \gamma}$. Just as we did in Lemma \ref{lem:CR_data}, we can Fourier expand the components of $\wt{\lambda}$ with respect to an adapted coframe.

We start with a technical lemma.
\begin{lem}\label{lem:NSh2Feff}
	Let $(\wt{\mc{M}},\wt{\mbf{c}},\wt{N},\wt{k}) \accentset{\varpi}{\longrightarrow} (\mc{M},H,J)$ be a twist-induced nearly Robinson geometry with non-shearing congruence of dimension $2m+2 \geq 4$. Let $\wt{g}_{\theta} \in \accentset{n.e.}{\wt{\mbf{c}}}$ for some pseudo-Hermitian structure $\theta$ so that $\wt{g}_{\theta}$ is given by \eqref{eq:can_met}. Suppose that
	\begin{align}\label{eq:lambda_Fourier}
		\wt{\lambda}_{\alpha} & = \sum_{k \in \mc{I}} \lambda_{\alpha}^{(k)} \e^{k \i \phi} \, , &
		\wt{\lambda}_{\bar{\alpha}} & = \sum_{k \in \mc{I}} \lambda_{\bar{\alpha}}^{(-k)} \e^{-k \i \phi} \, , &
		\wt{\lambda}_{0} & = \sum_{k \in - \mc{J} \cup \mc{J}} \lambda_{0}^{(k)} \e^{k \i \phi} \, ,
	\end{align}
	for some $\mc{I} \subset \Z$, $\mc{J} \subset \Z_{\geq0}$. Then $(\wt{\mc{M}},\wt{\mbf{c}},\wt{N},\wt{k})$ is locally conformally isometric to a perturbed Fefferman space $(\wt{\mc{M}}',\wt{\mbf{c}}'_{\wt{\xi}},\wt{k}') \longrightarrow (\mc{M},H,J)$ with perturbation one-form $\wt{\xi}$ determined by the CR data $\left(\bm{\xi}_{\alpha}^{(i)}, [\nabla,\bm{\xi}_{0}^{(j)}] \right)_{i \in \mc{I}, j \in \mc{J}}$, where for each choice $\sigma \in \Gamma(\mc{E}(1,0))$ trivialising $\wt{\mc{M}}$, we have
	\begin{subequations}\label{eq:lmb2xi}
		\begin{align}
			\lambda_{\alpha}^{(0)} & = \i \sigma^{-1} \nabla_{\alpha} \sigma + \xi_{\alpha}^{(0)}   \, , &
			\lambda_{0}^{(0)} & = \i \sigma^{-1} \nabla_{0} \sigma + \xi_{0}^{(0)} - \left( \tfrac{1}{m+2} \Rho + \tfrac{1}{2m(m+1)} \| \Nh \|^2 \right)  \, , &
			\\
			\lambda_{\alpha}^{(k)} & =  \xi_{\alpha}^{(k)} \, , &
			\lambda_{0}^{(k)} & =  \xi_{0}^{(k)} \, , & k \neq 0 \, ,
		\end{align}
	\end{subequations}
	where $\xi_{\alpha}^{(i)}$ and $\xi_{0}^{(j)}$ are given by \eqref{eq:Fourier_coef}.
\end{lem}

\begin{proof}
	Let us first remark that \eqref{eq:can_met} with coefficients given in \eqref{eq:lmb2xi} and $\wt{k} = \parderv{}{\phi}$ are invariant under translations $\phi \mapsto \phi + 2 k \pi$ for any $k \in \Z$. So, locally, we may consider the quotient $\wt{\mc{M}}$ by these, and restrict $\phi$ to $[-\pi,\pi)$.
	
	Since $(\mc{M},H,J)$ is an almost CR manifold, we associate to it its conformal Fefferman space $(\wt{\mc{M}}',\wt{\mbf{c}}',\wt{k}') \accentset{\varpi'}{\longrightarrow} (\mc{M},H,J)$. To the pseudo-Hermitian structure $\theta$, we have an associated Webster connection $\nabla$ with induced connection one-form $\wt{\omega}_{\theta}$ on $\wt{\mc{M}}'$, and a corresponding Fefferman metric $\wt{g}'_{\theta}$ in $\wt{\mbf{c}}'$.
	
	Next, choose a nowhere vanishing density $\sigma$ of weight $(1,0)$ such that $\theta = (\sigma \ol{\sigma})^{-1} \bm{\theta}$ and denote by $\phi'$ the corresponding fibre coordinate on $\wt{\mc{M}}' \rightarrow \mc{M}$ so that $\wt{k}' = \parderv{}{\phi'}$ and a point $\wt{p}'$ of $\wt{\mc{M}}'$ can be expressed locally as $(\varpi'(\wt{p}'),\e^{\i \phi'(\wt{p}')} \sigma^{-1}(\varpi'(\wt{p}')))$. The map that sends a point $\wt{p}$ in $\wt{\mc{M}}$ to $(\varpi(\wt{p}) , \e^{\i \phi (\wt{p}) } \sigma^{-1} (\varpi(\wt{p})))$ in $\wt{\mc{M}}'$ is a bundle map $\iota$ from $(\wt{\mc{M}},\wt{\mbf{c}},\wt{k})$ to $(\wt{\mc{M}}',\wt{\mbf{c}}',\wt{k}')$: clearly $\varpi' \circ \iota = \varpi$, and any change of affine parametrisation of $\wt{k}$ of the form $\phi' = \phi - \mr{\phi}$ induces a change of trivialisation as $\sigma' = \e^{-\i \mr{\phi}}\sigma$, so $\iota$ does not depend on the choice of trivialisation.
	
Now, since all the coefficients $\lambda_{\alpha}^{(k)}$, $\lambda_{0}^{(k)}$ are independent of the coordinate $\phi$, we view them as tensors on $(\mc{M},H,J)$, and define
	\begin{align*}
		\xi_{\alpha}^{(0)} & := \lambda_{\alpha}^{(0)} - \i \sigma^{-1} \nabla_{\alpha} \sigma  \, , &
		\xi_{0}^{(0)} & := \lambda_{0}^{(0)} - \i \sigma^{-1} \nabla_{0} \sigma + \left( \tfrac{1}{m+2} \Rho + \tfrac{1}{2m(m+1)} \| \Nh \|^2 \right)  \, , &
		\\
		\xi_{\alpha}^{(k)} & := \lambda_{\alpha}^{(k)}  \, , &
		\xi_{0}^{(k)} & := \lambda_{0}^{(k)} \, , & k \neq 0 \, .
	\end{align*}
	It is then straightforward to check that under an affine reparametrisation of the geodesics of $\wt{k}$ and a change of contact form, the induced transformations of $\xi_{0}^{(k)}$ and $\xi_{0}^{(k)}$ allow us to view them as the trivialisations of some CR data $\left(\bm{\xi}_{\alpha}^{(i)}, [\nabla,\bm{\xi}_{0}^{(j)}] \right)_{i \in \mc{I}, j \in \mc{J}}$. We can therefore define a perturbation one-form $\wt{\xi}$ from this CR data, and an associated perturbed Fefferman metric
	\begin{align*}
		\wt{g}'_{\theta,\wt{\xi}} & := \wt{g}'_{\theta} + 4 \theta \odot \wt{\xi} \, .
	\end{align*}
	One can also verify that under a change of contact form, the metrics thus constructed rescale by the same conformal factor. Checking that $\iota$ is a conformal isometry that sends each metric $\wt{g}_{\theta}$ in $\accentset{n.e.}{\wt{\mbf{c}}}$ to a perturbed Fefferman metric $\wt{g}'_{\theta,\wt{\xi}}$, is routine. Finally, $\iota_* \wt{k}|_{\wt{p}} = \wt{k}'|_{\iota(\wt{p})}$ at every point $\wt{p}$ of $\wt{\mc{M}}$.
\end{proof}

\subsection{Curvature conditions}\label{sec:char_pert}
	Let $(\wt{\mc{M}},\wt{\mbf{c}},\wt{K})$ be an optical geometry of dimension $2m+2\geq4$ with twisting non-shearing congruence of null geodesics. Then, we already know, by Theorem \ref{thm:non-shearing-hidim}, that the twist induces a nearly Robinson structure $(\wt{N},\wt{k})$, where $\wt{k}$ is the distinguished section of $\wt{K}$ of Section \ref{sec:prefered}, if and only if the Weyl tensor satisfies
	\begin{align}\label{eq:Wkvkv}
		\wt{\Weyl}_{a b c d} \wt{k}^{a} \wt{v}^{b} \wt{k}^{c} \wt{v}^{d} & = 0 \, ,  & \mbox{for any $\wt{k} \in \Gamma(\wt{K}), \wt{v} \in \Gamma (\wt{K}^\perp)$.}
	\end{align}
	This condition is vacuous in dimension four, and we will henceforth focus on dimensions greater than four.

We presently examine further conformally invariant degeneracy curvature conditions.

\begin{prop}\label{prop:PetrovII}
	Let $(\wt{\mc{M}},\wt{\mbf{c}},\wt{K})$ be an optical geometry of dimension $2m+2>4$ with twisting non-shearing congruence of null geodesics $\wt{\mc{K}}$. Then the Weyl tensor satisfies
	\begin{align}
		\wt{\Weyl}_{a b c d} \wt{k}^{a} \wt{v}^{b} \wt{k}^{c} & = 0 \, , & \mbox{for any $\wt{k} \in \Gamma(\wt{K}), \wt{v} \in \Gamma (\wt{K}^\perp)$,} \label{eq:PetrovIIa} \\
		\wt{\Weyl}_{a b c d} \wt{k}^{a} \wt{v}^{b} \wt{\nabla}^{c} \wt{k}^{d} & = 0 \, , & \mbox{for any $\wt{k} \in \Gamma(\wt{K}), \wt{v} \in \Gamma (\wt{K}^\perp)$,} \label{eq:PetrovIIb}
	\end{align}
if and only if the twist of $\wt{\mc{K}}$ induces a nearly Robinson structure, and, for any $\wt{g} \in \wt{\mbf{c}}$,
	\begin{align}\label{eq:strongNS}
		\mathsterling_{\wt{k}} \wt{g} ( \wt{v},\cdot ) & \propto \wt{g} ( \wt{v},\cdot ) \, , & \mbox{for any $\wt{k} \in \Gamma(\wt{K}), \wt{v} \in \Gamma (\wt{K}^\perp)$.}
	\end{align}
\end{prop}

\begin{proof}
Condition \eqref{eq:PetrovIIa} implies \eqref{eq:Wkvkv}, which tells us, by Theorem \ref{thm:non-shearing-hidim}, that we can view our conformal manifold as a twist-induced nearly Robinson geometry $(\wt{\mc{M}},\wt{\mbf{c}},\wt{N},\wt{k})$ as in Section \ref{sec:prefered}. We work with a metric $\wt{g}_{\theta}$ in $\accentset{n.e.}{\wt{\mbf{c}}}$ for some pseudo-Hermitian structure $\theta$ as given by \eqref{eq:can_met}. It is shown in \cite{TaghaviChabert2022} that the Weyl tensor satisfies \eqref{eq:PetrovIIa} if and only if the component $\wt{\lambda}_{\alpha}$ satisfies
\begin{align}\label{eq:condlambda1}
	\wt{\lambda}_{\alpha} & =
	\begin{cases}
		\lambda_{\alpha} + \mu_{\alpha} \phi \, , & m = 2 \, , \\
		\lambda_{\alpha} + \mu_{\alpha} \e^{\tfrac{2m-4}{2m-1} \i \phi} \, , & m > 2 \, ,
	\end{cases}
\end{align}
for some $\lambda_{\alpha}, \mu_{\alpha} \in \Gamma(\mc{E}_{\alpha})$. On the other hand, using \cite[Appendix~A]{TaghaviChabert2022}, we can compute, in the obvious notation,
\begin{align*}
	\wt{\Weyl}_{\beta}{}^{\beta 0}{}_{\alpha}
	& = - \tfrac{1}{4m} \left( \ddot{\wt{\lambda}}_{\alpha} +  2(2m^2 + m - 2)\i \dot{\wt{\lambda}}_{\alpha} \right) \, ,
\end{align*}
and so, condition \eqref{eq:PetrovIIb} is equivalent to
\begin{align}\label{eq:condlambda2}
	\wt{\lambda}_{\alpha} & = \lambda_{\alpha} + \nu_{\alpha} \e^{- 2(2m^2 + m - 2) \i \phi} \, , 
\end{align}
for some $\lambda_{\alpha}, \nu_{\alpha} \in \Gamma (\mc{E}_{\alpha})$. By \eqref{eq:condlambda1} and \eqref{eq:condlambda2}, and since $m>1$, it immediately follows that $\wt{\lambda}_{\alpha} = \lambda_{\alpha}$. But, for any $\wt{v} \in \Gamma( \langle \wt{k} \rangle^\perp)$, we have that
\begin{align*}
	\mathsterling_{\wt{k}} \wt{g}_{\theta}  (\wt{v},\cdot) & = 4 \theta \odot \left( \wt{k} \hook \d \wt{\lambda} + \d (\wt{k} \hook \wt{\lambda} ) \right) (\wt{v},\cdot) = 2 \dot{\wt{\lambda}}_{\alpha} \theta^{\alpha}(\wt{v}) \theta = 0 \, ,
\end{align*}
which establishes condition \eqref{eq:strongNS}. The converse works in the same way.
\end{proof}

\begin{rem}\label{rem:Cotton-Weyl}
	Note that, for any $\wt{k} \in \Gamma(\wt{K}), \wt{v} \in \Gamma (\wt{K}^\perp)$, we have
	\begin{align*}
		(n-1)  \wt{k}^a \wt{v}^{b} \wt{k}^{c} \wt{\Cot}_{abc} 
		& = - \wt{\nabla}^c (\wt{k}^d \wt{k}^a \wt{v}^{b} \wt{\Weyl}_{c d a b} ) - ( \wt{\nabla}^d \wt{v}^{b} ) \wt{k}^c \wt{\Weyl}_{c d a b} \wt{k}^a + \tfrac{3}{2}  \wt{k}^a \wt{v}^{b} \wt{\Weyl}_{a b c d} \wt{\nabla}^{c} \wt{k}^{d} \, .
	\end{align*}
Assuming \eqref{eq:PetrovIIa}, which can be re-expressed as $\wt{k}^c \wt{\Weyl}_{c d a b} \wt{k}^a = \wt{f} \wt{\kappa}_{d} \wt{\kappa}_{b}$ for some function $\wt{f}$ and where $\wt{\kappa} = \wt{g}(\wt{k},\cdot)$, the above equation reduces to $(n-1)  \wt{k}^a \wt{v}^{b} \wt{k}^{c} \wt{\Cot}_{abc} = \tfrac{3}{2}  \wt{k}^a \wt{v}^{b} \wt{\Weyl}_{a b c d} \wt{\nabla}^{c} \wt{k}^{d}$. Hence, \eqref{eq:PetrovIIb} is equivalent to
	\begin{align*}
	\wt{\Cot}_{a b c} \wt{k}^{a} \wt{v}^{b} \wt{k}^{c} & = 0 \, , & \mbox{for any $\wt{k} \in \Gamma(\wt{K}), \wt{v} \in \Gamma (\wt{K}^\perp)$,}
\end{align*}	
which is conformally invariant provided that \eqref{eq:PetrovIIa} holds.
\end{rem}

\begin{rem}\label{rem:repPND}
	In dimension $2m+2 = 4$, condition \eqref{eq:PetrovIIa} is the defining equation for $\wt{K}$ being a \emph{repeated principal null direction} of the Weyl tensor --- see \cite[Section~3.4]{TaghaviChabert2023}. In the context of twisting non-shearing congruences of null geodesics, it immediately implies that $\wt{\lambda}_{\alpha} = \lambda_{\alpha} + \mu_{\alpha} \e^{-2 \i \phi}$ for some $\lambda_{\alpha}, \mu_{\alpha} \in \Gamma(\mc{E}_{\alpha})$, which can readily be interpreted as a Fourier expansion of $\wt{\lambda}_{\alpha}$ viewed as a periodic function with period $\pi$ --- this is treated in \cite{TaghaviChabert2023}. In addition, conditions \eqref{eq:PetrovIIa} and \eqref{eq:PetrovIIb} turn out to be equivalent.
	
	This should be contrasted with the even higher-dimensional case: merely assuming \eqref{eq:PetrovIIa}, the integration \eqref{eq:condlambda1} tells us that $\wt{\lambda}_{\alpha}$ is clearly not periodic in $\phi$ in the case $m=2$, while in the case $m>2$, the period $\tfrac{2m-4}{2m-1} \pi$ is not an integer multiple of $\pi$, which is problematic if one ultimately desires to interpret $\wt{\mc{M}}$ as Fefferman's bundle over $\mc{M}$. It is then necessary to impose the additional requirement \eqref{eq:PetrovIIb} to force $\wt{\lambda}$ to be interpretable as an appropriate Fourier expansion in $\phi$.
\end{rem}

To fully integrate out the fibre dependence of $\wt{\lambda}$, we need the following stronger condition.
\begin{prop}\label{prop:PetrovIII}
	Let $(\wt{\mc{M}},\wt{\mbf{c}},\wt{K})$ be an optical geometry of dimension $2m+2>4$ with twisting non-shearing congruence of null geodesics $\wt{\mc{K}}$. Suppose the Weyl tensor satisfies
	\begin{align}
	\wt{\Weyl}_{a b c d} \wt{k}^{a} \wt{k}^{c} & = 0 \, , & \mbox{for any $\wt{k} \in \Gamma(\wt{K}), \wt{v} \in \Gamma (\wt{K}^\perp)$,} \label{eq:PetrovIIIa} \\
		\wt{\Weyl}_{a b c d} \wt{k}^{a} \wt{v}^{b} \wt{\nabla}^{c} \wt{k}^{d} & = 0 \, , & \mbox{for any $\wt{k} \in \Gamma(\wt{K}), \wt{v} \in \Gamma (\wt{K}^\perp)$.} \label{eq:PetrovIIIb}
\end{align}
Then $(\wt{\mc{M}},\wt{\mbf{c}},\wt{K})$ is locally conformally isometric to a perturbed Fefferman space\linebreak $(\wt{\mc{M}}',\wt{\mbf{c}}'_{\wt{\xi}},\wt{k}') \longrightarrow (\mc{M},H,J)$, and the perturbation one-form $\wt{\xi}$ is determined by the CR data $\left(\bm{\xi}_{\alpha}^{(0)}, [\nabla, \bm{\xi}_{0}^{(0)}] \right)$ where
	\begin{align}
		\bm{\xi}_{0}^{(0)} & = \tfrac{\i}{m} \left( \nabla_{\alpha} \bm{\xi}^{\alpha}_{(0)} -   \nabla^{\alpha} \bm{\xi}_{\alpha}^{(0)} \right) \, , \label{eq:coeff00} 
	\end{align}
In particular, $\wt{k}$ is conformal Killing.
\end{prop}

\begin{proof}
	Again, condition \eqref{eq:PetrovIIa} implies \eqref{eq:Wkvkv}, so that, by Theorem \ref{thm:non-shearing-hidim}, we view our conformal manifold as a twist-induced nearly Robinson geometry $(\wt{\mc{M}},\wt{\mbf{c}},\wt{N},\wt{k})$ as in Section \ref{sec:prefered}. Let $\wt{g}_{\theta}$ be any metric in $\accentset{n.e.}{\wt{\mbf{c}}}$ for some pseudo-Hermitian structure $\theta$ as given by \eqref{eq:can_met}. Clearly, the pair of conditions \eqref{eq:PetrovIIIa} and \eqref{eq:PetrovIIIb} implies the hypotheses of Proposition \ref{prop:PetrovII}, which allows us to write $\wt{\lambda}_{\alpha} = \lambda_{\alpha}^{(0)}$ for some $\lambda_{\alpha}^{(0)} \in \Gamma(\mc{E}_{\alpha})$. To integrate out the $\phi$-dependence of $\wt{\lambda}_{0}$, we must make full use of \eqref{eq:PetrovIIIa} --- incidentally, this constraint was already used in Theorem \ref{thm:Fefferman-CRA} --- and compute $\wt{W} (\wt{k}, \wt{\ell}, \wt{k}, \wt{\ell})$, where $\wt{\ell}$ is such that $\wt{g}_{\theta}(\wt{k},\wt{\ell}) =1$. We find that the component $\wt{\lambda}_0$ is constant along $\wt{k}$, and
\begin{align*}
\wt{\lambda}_0  & = \lambda_0^{(0)} :=  
	 \tfrac{\i}{m}  \left( \nabla_{\alpha} \lambda^{\alpha}_{(0)} - \nabla^{\alpha} \lambda_{\alpha}^{(0)}  \right) 	+ \tfrac{1}{2m(m+1)} \left( \Sc - \| \Nh \|^2 \right) \, .
\end{align*}
We can now use \eqref{eq:lmb2xi} and apply the commutation relation \eqref{eq:CR_com1} to conclude
	\begin{align*}
		\xi_{\alpha}^{(0)}  & = \lambda_{\alpha}^{(0)} - \i \sigma^{-1} \nabla_{\alpha} \sigma    \, , &
		\xi_{0}^{(0)} & = \tfrac{\i}{m}  \left( \nabla_{\alpha} \xi^{\alpha}_{(0)} - \nabla^{\alpha} \xi_{\alpha}^{(0)}  \right) \, .
	\end{align*}
Invoking Lemma \ref{lem:NSh2Feff} completes the proof.
\end{proof}

\begin{rem}
	We note that if the perturbation one-form is (locally) exact then \eqref{eq:coeff00} holds by Lemma \ref{lem:closedg}. It also means that  $\wt{\mbf{c}}$ is locally conformally isometric to a Fefferman space. In particular, the additional curvature condition \eqref{eq-int_cond} with $\alpha=1$ holds.
\end{rem}

\begin{rem}
In dimension $2m+2=4$, condition \eqref{eq:PetrovIIIa} (which implies \eqref{eq:PetrovIIIb}) is equivalent to $\wt{K}$ being a \emph{triple principal null direction} of the Weyl tensor --- see \cite[Section~3.4]{TaghaviChabert2023}.
\end{rem}
	
	It is proved in \cite{TaghaviChabert2023} that in dimension four, full integration of the components of $\wt{\lambda}_{\alpha}$ and $\wt{\lambda}_{0}$ is possible under weaker curvature conditions, namely that the Weyl tensor and the Bach tensor satisfy
	\begin{align*}
		\wt{\Weyl}_{a b c d} \wt{k}^{a} \wt{v}^{b} \wt{k}^{c} & = 0 \, , & \mbox{for any $\wt{k} \in \Gamma(\wt{K}), \wt{v} \in \Gamma (\wt{K}^\perp)$,} \\
		\wt{\Bach}_{a b} \wt{k}^{a} \wt{k}^{b} & = 0 \, , & \mbox{for any $\wt{k} \in \Gamma(\wt{K})$,}
	\end{align*}
	respectively. In higher dimensions, the Bach tensor is not conformally invariant any more, but there is a conformally invariant analogue, known as the \emph{Fefferman--Graham obstruction tensor}, introduced in \cite{Fefferman1985}, see also \cite{Gover2006}, which we shall denote by $\wt{\FG}$. This tensor is an obstruction to the existence of an Einstein metric in the conformal class. Its explicit form is dimension-dependent and rather complicated. We shall content ourselves to conjecture the following generalisation of \cite[Theorem~5.6]{TaghaviChabert2023} to dimensions $2m+2>4$. The form of the perturbation one-form given below is motivated by Theorems \ref{thm:hlfEinsc} and \ref{thm:pertFeffpurad} of the next section.

\begin{conj}\label{conj}
	Let $(\wt{\mc{M}},\wt{\mbf{c}},\wt{K})$ be an optical geometry of dimension $2m+2>4$ with twisting non-shearing congruence of null geodesics $\wt{\mc{K}}$. Suppose the Weyl tensor satisfies
	\begin{align*}
		\wt{\Weyl}_{a b c d} \wt{k}^{a} \wt{v}^{b} \wt{k}^{c} & = 0 \, , &
		\wt{\Weyl}_{a b c d} \wt{k}^{a} \wt{v}^{b} \wt{\nabla}^{c} \wt{k}^{d} & = 0 \, , & \mbox{for any $\wt{k} \in \Gamma(\wt{K}), \wt{v} \in \Gamma (\wt{K}^\perp)$,}
	\end{align*}
	and the Fefferman--Graham obstruction tensor
	\begin{align*}
		\wt{\FG}_{a b} \wt{k}^a \wt{k}^b & = 0 \, , & \mbox{for any $\wt{k} \in \Gamma(\wt{K})$.}
	\end{align*}
	Then $(\wt{\mc{M}},\wt{\mbf{c}},\wt{K})$ is locally conformally isometric to a perturbed Fefferman space\linebreak $(\wt{\mc{M}}',\wt{\mbf{c}}'_{\wt{\xi}},\wt{k}') \longrightarrow (\mc{M},H,J)$, and the perturbation one-form $\wt{\xi}$ is determined by the CR data $\left(\bm{\xi}_{\alpha}^{(0)}, [\nabla, \bm{\xi}_{0}^{(0)}], \bm{\xi}_{0}^{(2k)} \right)_{k=1,\ldots,m+1}$ where $\bm{\xi}_{0}^{(0)} = \tfrac{\i}{m} \left( \nabla_{\alpha} \bm{\xi}^{\alpha}_{(0)} -   \nabla^{\alpha} \bm{\xi}_{\alpha}^{(0)} \right)$.
\end{conj}

\section{Distinguished almost Lorentzian scales}\label{sec:al_Lor_sc}
We now investigate the consequences of the existence of the distinguished almost Lorentzian scales that we introduced in Section \ref{sec:confsc}. We remind the reader that a Fefferman space will here refer to a $1$-Fefferman space except in Corollary \ref{cor:Einstein_Fefferman}. This comes as no cost since, for any $1 \neq \alpha \in \R$,  any $\alpha$-Fefferman space can then be viewed as a perturbed Fefferman space.

\subsection{Preliminary results}
For the time being, the dimension is assumed to be $2m+2 \geq4$. Let us first describe the zero set of almost Lorentzian scales.
\begin{prop}\label{prop:density_Ric_NS}
	Let $(\wt{\mc{M}},\wt{\mbf{c}},\wt{N},\wt{k}) \accentset{\varpi}{\longrightarrow} (\mc{M},H,J)$ be a twist-induced nearly Robinson geometry with non-shearing congruence. Then in the neighbourhood $\wt{\mc{U}}$ of each point, we can always find an almost Lorentzian scale $\wt{\sigma} \in \Gamma(\wt{\mc{E}}[1])$ that satisfies 
	\begin{align}\label{eq:Rickksc_a}
		\wt{k}{}^{a} \wt{k}{}^{b} \left( \wt{\nabla}_{a} \wt{\nabla}_{b} \wt{\sigma} + \wt{\Rho}_{a b} \wt{\sigma} \right) & = 0 \, ,
	\end{align}
	and whose zero set $\wt{\mc{Z}}$ consists of the union of sections of $\varpi: \wt{\mc{U}} \rightarrow \mc{U} := \varpi (\wt{\mc{U}}) \subset \mc{M}$ parametrised by the integers $\Z$.
	
	More specifically, there exists an affine parameter $\phi$ along the geodesics of $\wt{k}$, unique up to transformations $\phi \mapsto \phi + 2 k \pi$ for any $k \in \Z$, such that
	\begin{align}\label{eq:scRickk_a}
		\wt{\sigma} & = \cos \phi \cdot \wt{\sigma}_{\theta} \, ,
	\end{align}
	where $\wt{\sigma}_{\theta}$ is the conformal scale of the metric $\wt{g}_{\theta} \in \accentset{n.e.}{\wt{\mbf{c}}}$ associated to a unique contact form $\theta$ for $(H,J)$. In particular, the zero set of $\wt{\sigma}$ is
	\begin{align}\label{eq:zeroset}
		\wt{\mc{Z}} = \left\{ \phi = \tfrac{2 k + 1}{2}\pi \, | \, k \in \Z \right\} \, ,
	\end{align}
	and off $\wt{\mc{Z}}$, $\wt{\sigma}$ determines a smooth metric $\wt{g} = \wt{\sigma}{}^{-2} \wt{\bm{g}}$ of the form
	\begin{align}\label{eq:metRickk_a}
		\wt{g} & = \sec^2 \phi \cdot \wt{g}_{\theta} \, , 
	\end{align}
 	whose Ricci tensor satisfies
	\begin{align}\label{eq:lowRicond_a}
		\wt{\Ric}(\wt{k},\wt{k}) & = 0 \, .
	\end{align}
\end{prop}

\begin{proof}
	Equation \eqref{eq:Rickksc_a} represents a second-order ordinary differential equation whose local solutions are guaranteed by the Picard--Lindel\"{o}f theorem. To be precise, we work in the neighbourhood $\wt{\mc{U}}$ of a point and we choose a conformal scale $\wt{\sigma}_{\theta}$ associated to a metric $\wt{g}_{\theta} \in \accentset{n.e.}{\wt{\mbf{c}}}$ for some contact form $\theta$ for $(H,J)$. Let $\phi$ be an affine parameter along the null geodesics of the congruence generated by $\wt{k}$, so that $\wt{k} = \parderv{}{\phi}$. Locally, any almost Lorentzian scale can then be written as $\wt{\sigma} = \wt{f} \wt{\sigma}_{\theta}$ for some smooth function $\wt{f}$ on $\wt{\mc{U}}$. We shall naturally exclude the trivial case where $\wt{\sigma}$, and thus $\wt{f}$, are identically zero. Then, choosing $\wt{\nabla}$ to be the Levi-Civita associated to $\wt{\sigma}_{\theta}$, i.e.\ $\wt{\nabla} \wt{\sigma}_{\theta} = 0$ and the fact \cite{TaghaviChabert2022} that $\wt{k}^a \wt{k}^b \wt{\Rho}_{a b} = 1$, we find, using the Leibniz rule, that $\wt{\sigma}$ satisfies \eqref{eq:Rickksc_a} if and only if $\wt{f}$ satisfies
	\begin{align}\label{eq:ODE}
		\ddot{\wt{f}} + \wt{f} = 0 \, .
	\end{align}
	This has general solution
	\begin{align}\label{eq:confact}
		\wt{f} = \mr{\varrho} \cos(\phi - \mr{\phi}) \, ,
	\end{align}
	for some smooth functions $\mr{\varrho}$ and $\mr{\phi}$ on $\mc{U} := \varpi(\wt{\mc{U}}) \subset \mc{M}$, and each choice of a pair $(\mr{\varrho},\mr{\phi})$ singles out a unique solution to \eqref{eq:ODE}. We may choose $\mr{\varrho}$ to be nowhere vanishing on $\mc{U}$.  We can then rescale $\wt{\sigma}_\theta$ by $\mr{\varrho}$, effectively rescaling $\theta$ by $\mr{\varrho}^{-2}$, and reparametrise the geodesics of $\wt{k}$ as in \eqref{eq:reparam} so that $\wt{\sigma}$ now takes the form \eqref{eq:scRickk_a}. It immediately follows that the zero set of $\wt{\sigma}$ is given by \eqref{eq:zeroset}.
	
	Finally, following from the same line of arguments used in the proof of Proposition \ref{prop:scale2metric}, the existence of a density $\wt{\sigma}$ that satisfies \eqref{eq:Rickksc_a} off its zero set $\wt{\mc{Z}}$ is equivalent to the existence of a metric $\wt{g} = \wt{\sigma}^{-2} \wt{\bm{g}}$ that takes the form \eqref{eq:metRickk_a} and satisfies \eqref{eq:lowRicond_a} off $\wt{\mc{Z}}$.
\end{proof}

\begin{rem}
	The form of the metric \eqref{eq:metRickk_a} was also obtained in \cite{Lewandowski1990,TaghaviChabert2022} without the use of an almost Lorentzian scale. Also, condition \eqref{eq:lowRicond_a} is in fact conformally invariant if we restrict ourselves to conformal changes of metrics that are induced from changes of contact forms on $(\mc{M},H,J)$ --- this can be seen from the transformation law \eqref{eq:Rho_transf} for the Schouten tensor. Therefore, it does not single out any particular almost Lorentzian scale.
\end{rem}

Since the zero set of the density $\wt{\sigma}$ in Proposition \ref{prop:density_Ric_NS} is never empty, we obtain as a corollary:
\begin{cor}
	Let $(\wt{\mc{M}},\wt{\mbf{c}},\wt{N},\wt{k}) \longrightarrow (\mc{M},H,J)$ be a twist-induced nearly Robinson geometry with non-shearing congruence. Then there is no global metric in $\wt{\mbf{c}}$ whose Ricci tensor satisfies $\wt{\Ric}(\wt{k},\wt{k}) = 0$. 
\end{cor}

We can refine our findings by assuming that our optical geometry is in fact a perturbed Fefferman conformal structure $(\wt{\mc{M}},\wt{\mbf{c}}_{\wt{\xi}},\wt{k}) \longrightarrow (\mc{M},H,J)$, which means that we work on a circle, rather than line, bundle. In this case, a nowhere vanishing density $\sigma \in \Gamma(\mc{E}(1,0))$ trivialises $\wt{\mc{M}} \rightarrow \mc{M}$ with a local fibre coordinate $\phi$, and determines a contact form $\theta = (\sigma \ol{\sigma})^{-1} \bm{\theta}$ with associated perturbed Fefferman metric $\wt{g}_{\theta,\wt{\xi}}$ corresponding to a conformal scale $\wt{\sigma}_{\theta,\wt{\xi}}$. We can then define an almost Lorentzian scale $\wt{\sigma}= \cos \phi \cdot \wt{\sigma}_{\theta,\wt{\xi}}$, and a direct computation will show that it satisfies equation \eqref{eq:Rickksc_a}. Finally, since the fibre coordinate $\phi$ takes value in $[-\pi, \pi)$, the zero set of $\wt{\sigma}$ is now simply the hypersurfaces $\phi = \pm \tfrac{\pi}{2}$, which, by definition of the Fefferman bundle, can be identified as the cross-sections $[\pm \i \sigma^{-1}] : \mc{M} \rightarrow \wt{\mc{M}}$.

Conversely, we can reinterpret the proof of Proposition \ref{prop:density_Ric_NS} in the special case of a perturbed Fefferman space. We work with an arbitrary choice of trivialising density $\sigma \in \Gamma(\mc{E}(1,0))$ with contact form $\theta$ and perturbed Fefferman metric $\wt{g}_{\theta,\wt{\xi}}$ and seek a solution $\wt{\sigma}$ to \eqref{eq:Rickksc_a}. We are then led to conclude that it must take the form $\wt{\sigma} = \mr{\varrho} \cos(\phi - \mr{\phi}) \wt{\sigma}_{\theta,\wt{\xi}}$ for some functions of integrations $\mr{\varrho}$ and $\mr{\phi}$ of the ODE \eqref{eq:ODE}. This singles out a unique nowhere vanishing density $\wh{\sigma} := z \sigma$ where $z = \mr{\varrho} \e^{-\i \mr{\phi}}$. Note that there is a corresponding change of contact form $\wh{\theta} = \mr{\varrho}^{-2} \theta$, and thus a change of perturbed Fefferman metric $\wt{g}_{\wh{\theta},\wt{\xi}} = \mr{\varrho}^{-2} \wt{g}_{\theta,\wt{\xi}}$ and a change of fibre coordinate $\wh{\phi} = \phi - \mr{\phi}$. For notational convenience, we drop the hats, i.e.\ $\sigma$ is our new density $\wh{\sigma}$, and $\wt{\sigma} = \cos \phi \cdot \wt{\sigma}_{\theta}$ is an almost Lorentzian scale satisfying equation \eqref{eq:Rickksc_a}. We can now reformulate Proposition \ref{prop:density_Ric_NS} as follows:
\begin{prop}\label{prop:density_Ric}
	Let $(\wt{\mc{M}},\wt{\mbf{c}}_{\wt{\xi}},\wt{k}) \longrightarrow (\mc{M},H,J)$ be a  perturbed Fefferman space. Then any choice of nowhere vanishing density $\sigma$ of weight $(1,0)$ on $\mc{M}$ determines an almost Lorentzian scale $\wt{\sigma} \in \Gamma(\wt{\mc{E}}[1])$ on $\wt{\mc{M}}$ that satisfies 
	\begin{align}\label{eq:RickkscFef}
		\wt{k}{}^{a} \wt{k}{}^{b} \left( \wt{\nabla}_{a} \wt{\nabla}_{b} \wt{\sigma} + \wt{\Rho}_{a b} \wt{\sigma} \right) & = 0 \, ,
	\end{align}
	and whose zero set $\wt{\mc{Z}}$ of $\wt{\sigma}$ consists of the union $\mc{Z}_+ \cup \mc{Z}_-$ of the sections $\mc{Z}_{\pm} := [\pm \i \sigma^{-1}] : \mc{M} \rightarrow \wt{\mc{M}}$. More particularly,
	\begin{align*}
		\wt{\sigma} = \cos \phi \cdot \wt{\sigma}_{\theta,\wt{\xi}} \, ,
	\end{align*}
	where $\wt{\sigma}_{\theta,\wt{\xi}}$ is the conformal scale of the perturbed Fefferman metric $\wt{g}_{\theta,\wt{\xi}}$ associated to the contact form $\theta = (\sigma \ol{\sigma})^{-1} \bm{\theta}$, and $\phi$ is the fibre coordinate determined by $\sigma$.
	
	Off $\wt{\mc{Z}}$, the metric $\wt{g} = \wt{\sigma}{}^{-2} \wt{\bm{g}}$ takes the form
	\begin{align}\label{eq:metRickkFef} 
		\wt{g} & = \sec^2 \phi \cdot \wt{g}_{\theta,\wt{\xi}} \, ,
	\end{align}
	and its Ricci tensor satisfies
	\begin{align}\label{eq:lowRicondFef}
		\wt{\Ric}(\wt{k},\wt{k}) & = 0 \, .
	\end{align}
	
	Conversely, any almost Lorentzian scale $\wt{\sigma}$ on $\wt{\mc{M}}$ satisfying \eqref{eq:RickkscFef} arises from a unique nowhere vanishing density $\sigma \in \Gamma(\mc{E}(1,0))$.
\end{prop}

\begin{rem}
	We note that the transformation $\sigma \mapsto -\sigma$ induces the transformation $\wt{\sigma} \mapsto - \wt{\sigma}$ while interchanging $\wt{\mc{Z}}_+$ and $\wt{\mc{Z}}_-$, but clearly leaves the contact form $\theta$, and thus the metrics $\wt{g}_{\theta,\wt{\xi}}$ and $\wt{g}$, unchanged.
\end{rem}

While $\wt{g}_{\theta}$ in Proposition \ref{prop:density_Ric_NS} is regular on $\wt{\mc{U}}$, there is, in general, no guarantee that it is periodic, and this prevents us to conclude that $(\wt{\mc{M}},\wt{\mbf{c}},\wt{N},\wt{k})$ is a circle bundle as for perturbed Fefferman spaces.

We shall soon consider more restrictive conditions on almost Lorentzian scales on the  geometry $(\wt{\mc{M}},\wt{\mbf{c}},\wt{N},\wt{k})$ of Proposition \ref{prop:density_Ric_NS}. We shall also suppose that the  Weyl tensor of $\wt{\mbf{c}}$ satisfies the stronger condition
\begin{align}
	\wt{\Weyl}_{a b c d} \wt{k}^{a} \wt{v}^{b} \wt{k}^{c} & = 0 \, , & \mbox{for any $\wt{v} \in \Gamma (\wt{K}^\perp)$.} \label{eq:PetrovIIa2} 
\end{align}
This is a reasonable assumption since in dimension four, this would tell us that $\langle \wt{k} \rangle$ is a repeated principal null direction of $\wt{\Weyl}$. Since the four-dimensional case is treated separately in \cite{TaghaviChabert2023}, we shall henceforth assume that the dimension of $\wt{\mc{M}}$ is greater than four, i.e.\ $2m+2>4$. 

Now, let $\wt{\sigma}$ be an almost Lorentzian scale, and suppose that $\wt{\sigma}$ satisfies
		\begin{align}\label{eq:Rickvsc_a}
	\wt{k}{}^{a} \wt{v}{}^{b} \left( \wt{\nabla}_{a} \wt{\nabla}_{b} \wt{\sigma} + \wt{\Rho}_{a b} \wt{\sigma} \right) & = 0 \, , & \mbox{for all $\wt{v}^{a} \in \Gamma (\langle \wt{k} \rangle^\perp)$.}
\end{align}
Then, by Proposition \ref{prop:density_Ric_NS}, $\wt{\sigma}$ has non-empty zero set $\wt{\mc{Z}}$, and off $\wt{\mc{Z}}$, the metric $\wt{g} = \wt{\sigma}^{-2} \wt{\bm{g}}$ takes the form \eqref{eq:metRickkFef} for some metric $\wt{g}_{\theta} \in \accentset{n.e.}{\wt{\mbf{c}}}$, which can be cast as \eqref{eq:can_met}. By virtue of \eqref{eq:Rickvsc_a}, the Ricci tensor of $\wt{g}$ satisfies
\begin{align}\label{eq:Rickv}
\wt{\Ric}(\wt{k},\wt{v}) & = 0 \, , & \mbox{for all $\wt{v} \in \Gamma (\langle \wt{k} \rangle^\perp)$.}
\end{align}
This follows from a variation of the proof of Proposition \ref{prop:scale2metric}. Further, it is shown in \cite{TaghaviChabert2022} that this Ricci curvature condition together with \eqref{eq:PetrovIIa2} is equivalent to the component $\wt{\lambda}_{\alpha}$ being identically zero. In conclusion:
\begin{lem}\label{lem:al_half-Einstein}
	Let $(\wt{\mc{M}},\wt{\mbf{c}},\wt{N},\wt{k}) \longrightarrow (\mc{M},H,J)$ be a twist-induced nearly Robinson geometry of dimension $2m+2>4$ with non-shearing congruence. Suppose that the Weyl tensor satisfies \eqref{eq:PetrovIIa2}. Let $\wt{\sigma}$ be an almost Lorentzian scale that satisfies \eqref{eq:Rickksc_a}, so that off the zero set $\wt{\mc{Z}}$ of $\wt{\sigma}$, $\wt{g} = \wt{\sigma}^{-2} \wt{\bm{g}}$ is a smooth metric. Then $\wt{\sigma}$ satisfies \eqref{eq:Rickvsc_a} if and only if off $\wt{\mc{Z}}$, the metric $\wt{g}$ takes the form
		\begin{align}\label{eq:strongNS_2}
			\wt{g} & = \sec^2 \phi \cdot \left( 4 \theta \odot \left( \d \phi + \wt{\lambda}_{0} \theta \right) + h \right) \, ,
		\end{align}
		for some smooth function $\wt{\lambda}_{0}$ and where $\theta$ is a pseudo-Hermitian structure with Levi form $h$, and $\phi$ an affine parameter along $\wt{\mc{K}}$.
\end{lem}

\begin{rem}
Note that condition \eqref{eq:Rickv} is equivalent to $\wt{k}^a \wt{\Rho}_{a b} = \wt{f} \wt{g}_{a b} \wt{k}^{b}$ for some function $\wt{f}$. It then follows that  for any $\wt{v}^{a} \in \Gamma (\langle \wt{k} \rangle^\perp)$,
	\begin{align*}
		\wt{k}^a \wt{v}^b \wt{k}^c \wt{\Cot}_{a b c} & = \wt{k}^a \wt{v}^b \wt{k}^c \left( \wt{\nabla}_{b} \wt{\Rho}_{c a} - \wt{\nabla}_{c} \wt{\Rho}_{b a} \right)  \\
		& = - 2 ( \wt{v}^b \wt{\nabla}_{b} \wt{k}^a ) \wt{k}^c \wt{\Rho}_{c a} + (\wt{k}^c \wt{\nabla}_{c} \wt{k}^a )\wt{v}^b \wt{\Rho}_{b a} + (\wt{k}^c \wt{\nabla}_{c} \wt{v}^b ) \wt{k}^a \wt{\Rho}_{b a} \\
		& = - 2 ( \wt{v}^b \wt{\nabla}_{b} \wt{k}^a ) \wt{f} \wt{g}_{a c} \wt{k}^{c} + (\wt{k}^c \wt{\nabla}_{c} \wt{k}^a )\wt{v}^b \wt{\Rho}_{b a} + (\wt{k}^c \wt{\nabla}_{c} \wt{v}^b ) \wt{f} \wt{g}_{a b} \wt{k}^{b} \, .
	\end{align*}
	The first term is clearly zero, while the last two terms must vanish since $\wt{k}^a$ is tangent to geodesic curves. So the Cotton tensor satisfies $\wt{\Cot}_{a b c} \wt{k}^{a} \wt{v}^b \wt{k}^{c} = 0$ for any $\wt{v}^{a} \in \Gamma (\langle \wt{k} \rangle^\perp)$. We thus see that Lemma \ref{lem:al_half-Einstein} is consistent with Proposition \ref{prop:PetrovII} in view of Remark \ref{rem:Cotton-Weyl}.
\end{rem}

\subsection{Almost weakly half-Einstein scales}
Continuing on from Lemma \ref{lem:al_half-Einstein}, we now suppose that $\wt{\sigma}$ is an almost weakly half-Einstein scale, i.e.\ it satisfies \eqref{eq:alwkEins}. By Proposition \ref{prop:scale2metric}, off $\wt{\mc{Z}}$, the Ricci tensor satisfies $\wt{\Ric}(\wt{v},\wt{v}) = 0$ for all $\wt{v} \in \Gamma (\wt{N})$, which can be shown \cite{TaghaviChabert2022} to be equivalent to the pseudo-Hermitian tensors of $\theta$ satisfying $\WbA_{\alpha \beta} = \nabla^{\gamma} \Nh_{\gamma (\alpha \beta)} = 0$. Analogously to Theorem \ref{thm:alCRE}, we can show that this is equivalent to the existence of a CR scale $s$ satisfying \eqref{eq:AalCRE} and \eqref{eq:NalCRE}. We can thus conclude:

\begin{prop}\label{prop:alwkhalfE}
	Let $(\wt{\mc{M}},\wt{\mbf{c}},\wt{N},\wt{k}) \longrightarrow (\mc{M},H,J)$ be a twist-induced nearly Robinson geometry of dimension $2m+2>4$ with non-shearing congruence. Suppose that the Weyl tensor satisfies \eqref{eq:PetrovIIa2}. Let $\wt{\sigma}$ be an almost Lorentzian scale that satisfies \eqref{eq:Rickksc_a}, so that off the zero set $\wt{\mc{Z}}$ of $\wt{\sigma}$, $\wt{g} = \wt{\sigma}^{-2} \wt{\bm{g}}$ is a smooth metric. Then $\wt{\sigma}$ is an almost weakly half-Einstein scale if and only if off $\wt{\mc{Z}}$, the metric $\wt{g}$ takes the form \eqref{eq:strongNS_2}, and the CR scale $s$ corresponding to $\theta$ satisfies
		\begin{subequations}\label{eq:CRscalealK}
			\begin{align}
				& \nabla_{(\alpha} \nabla_{\beta)} s + \i \WbA_{\alpha \beta} s + \Nh_{\gamma (\alpha \beta)} \nabla^{\gamma} s = 0 \, , \\
				& \nabla^{\gamma} s \Nh_{\gamma (\alpha \beta)} - \tfrac{1}{m} \nabla^{\gamma} \Nh_{\gamma (\alpha \beta)} s = 0 \, .
			\end{align}
		\end{subequations}
\end{prop}

Specialising the optical geometry to a Fefferman space leads to:
\begin{thm}\label{thm:wkhlfEinsc}
	Let $(\wt{\mc{M}},\wt{\mbf{c}}_{\wt{\xi}},\wt{k}) \longrightarrow (\mc{M},H,J)$ be a perturbed Fefferman space of dimension $2m+2>4$. Suppose that the Weyl tensor satisfies \eqref{eq:PetrovIIa2}. Then $\wt{\mbf{c}}_{\wt{\xi}}$ admits an almost weakly half-Einstein scale $\wt{\sigma}$ if and only if there exists a nowhere vanishing density $\sigma \in \Gamma(\mc{E}(1,0))$ such that the CR scale $s := \sigma \ol{\sigma}$ satisfies
	\begin{subequations}\label{eq:CRscalealK_B0}
		\begin{align}
			& \nabla_{(\alpha} \nabla_{\beta)} s + \i \WbA_{\alpha \beta} s + \Nh_{\gamma (\alpha \beta)} \nabla^{\gamma} s = 0 \, , \\
			& \nabla^{\gamma} s \Nh_{\gamma (\alpha \beta)} - \tfrac{1}{m} \nabla^{\gamma} \Nh_{\gamma (\alpha \beta)} s = 0 \, ,
		\end{align}
	\end{subequations}
	and $\wt{\xi}$ is determined by the CR data $\left(\bm{\xi}_{\alpha}^{(0)}, [\nabla , \bm{\xi}_{0}^{(0)}], \bm{\xi}_{0}^{(k)} \right)_{k\in \mc{I}}$ for some subset $\mc{I} \subset \Z_{\geq0}$.
	
	Moreover, the zero set $\wt{\mc{Z}}$ of $\wt{\sigma}$ consists of the union $\wt{\mc{Z}}_+ \cup \wt{\mc{Z}}_-$ of the sections $\wt{\mc{Z}}_{\pm} = [\pm \i \sigma^{-1}] : \mc{M} \rightarrow \wt{\mc{M}}$. Off $\wt{\mc{Z}}$, $\wt{\sigma}$ determines a weakly half-Einstein metric $\wt{g} = \sec^2 \phi \cdot \wt{g}_{\theta, \wt{\xi}}$, where $\phi$ is the local fibre coordinate determined by $\sigma$ and $\wt{g}_{\theta, \wt{\xi}}$ is the perturbed Fefferman metric associated to the contact form $\theta = (\sigma \ol{\sigma})^{-1} \bm{\theta}$.
\end{thm}

\begin{proof}
	We omit the proof here as it is subsumed in that of Theorem \ref{thm:hlfEinsc} below.
\end{proof}

\subsection{Almost half-Einstein scales}
Next, suppose that $\wt{\sigma}$ is an almost half-Einstein scale, i.e.\ it satisfies \eqref{eq:alwkEins} and \eqref{eq:alcstRicSc}. So, in particular, Proposition \ref{prop:alwkhalfE} holds. By Proposition \ref{prop:scale2metric}, condition \eqref{eq:alcstRicSc} tells us that, off $\wt{\mc{Z}}$, the Ricci scalar is constant, which, by \cite[equation~(8.14)]{TaghaviChabert2022}, is shown to be equivalent to the second-order differential equation on the component $\wt{\lambda}_{0}$ in \eqref{eq:strongNS_2}, for $\phi \neq \tfrac{2k+1}{2} \pi$, $k \in \Z$,
\begin{multline}\label{eq:ode1}
\ddot{\wt{\lambda}}_{0} + 2 (2m+1) \tan \phi \dot{\wt{\lambda}}_{0} + \left( -4m (m+1) + 2(m+1)(2m+1) \sec^2 \phi \right) \wt{\lambda}_{0} \\
- 2 (m+1) \wt{\Lambda} \sec^2 \phi + 2m \Lambda = 0 \, ,
\end{multline}
where $\wt{\Lambda} = \tfrac{1}{2m+2} \wt{\Sc}$ is constant, $\wt{\Sc}$ being the Ricci scalar of $\wt{g}$, and   $\Lambda := \tfrac{1}{m} \left( \Sc - \| \Nh \|^2 \right)$, with $\Sc$ being the Webster--Ricci scalar of $\theta$.
Applying Proposition \ref{prop:alwkhalfE} with the solution $\wt{\lambda}_{0}$ of \eqref{eq:ode1} proves:
\begin{prop}\label{prop:al_half-Einstein} 
	Let $(\wt{\mc{M}},\wt{\mbf{c}},\wt{N},\wt{k}) \longrightarrow (\mc{M},H,J)$ be a twist-induced nearly Robinson geometry of dimension $2m+2>4$ with non-shearing congruence. Suppose that the Weyl tensor satisfies \eqref{eq:PetrovIIa2}. Let $\wt{\sigma}$ be an almost Lorentzian scale that satisfies \eqref{eq:Rickksc_a}, so that off the zero set $\wt{\mc{Z}}$ of $\wt{\sigma}$, $\wt{g} = \wt{\sigma}^{-2} \wt{\bm{g}}$ is a smooth metric. Then $\wt{\sigma}$ is an almost half-Einstein scale if and only if off $\wt{\mc{Z}}$, the metric $\wt{g}$ takes the form \eqref{eq:strongNS_2}, the CR scale $s$ corresponding to $\theta$ satisfies \eqref{eq:CRscalealK}, and $\wt{\lambda}_{0}$ is given by
		\begin{align}\label{eq:ODE1sol}
			\wt{\lambda}_0 & = \sum_{-m-1}^{m+1} \lambda_0^{(2k)} \e^{2\i k \phi} \, ,
		\end{align}
		where
		\begin{subequations}\label{eq:coeff_step3}
			\begin{align}
				\lambda_0^{(2m+2)} & = \tfrac{m!(m+1)!}{2(2m+2)!} \left( (2m+1) \Lambda - (2m+2) \wt{\Lambda} \right) + \mu \, , \\
				\lambda_0^{(2k)} & = \tfrac{2(2m+1)!}{(m+1-k)!(m+1+k)!} \left( k \lambda_0^{(2m+2)} + (m+1-k)\Re(\mu) \right) \, ,  & 1 \leq |k| \leq m \, , \\
				\lambda_0^{(0)} & =  \tfrac{1}{2m+2} \Lambda + \tfrac{2(2m+1)!}{m!(m+1)!} \Re(\mu) \, ,
			\end{align}
		\end{subequations}
		for some complex-valued function $\mu$ on $\mc{M}$, and real-valued function $\Lambda := \tfrac{1}{m} \left( \Sc - \| \Nh \|^2 \right)$ on $\mc{M}$, with $\Sc$ being the Webster--Ricci scalar of $\theta$, and $\Nh$ the Nijenhuis tensor of $(N,J)$.
\end{prop}

\begin{rem}
	We can weaken our assumption in Lemma \ref{lem:al_half-Einstein} and Propositions \ref{prop:alwkhalfE} and \ref{prop:al_half-Einstein} by replacing the twist-induced nearly Robinson geometry by an optical geometry $(\wt{\mc{M}},\wt{\mbf{c}},\wt{K})$ with twisting non-shearing congruence of null geodesics $\wt{\mc{K}}$, together with the Weyl curvature condition \eqref{eq:Wkvkv}. This follows from the fact that, by Theorem \ref{thm:non-shearing-hidim}, condition \eqref{eq:Wkvkv} implies that the twist of $\wt{\mc{K}}$ induces a nearly Robinson structure $(\wt{N},\wt{k})$ with distinguished  generator $\wt{k}$ of $\wt{\mc{K}}$, and $(\wt{\mc{M}},\wt{\mbf{c}},\wt{N},\wt{k})$ is locally fibred over an almost CR manifold $(\mc{M},H,J)$.
\end{rem}

By the above remark, and since $\wt{\lambda}_{\alpha} = 0$ and $\wt{\lambda}_{0}$ is of the form \eqref{eq:coeff_step3}, applying Lemma \ref{lem:NSh2Feff} results in the following:

\begin{cor}\label{cor:alt_char}
	Let $(\wt{\mc{M}},\wt{\mbf{c}},\wt{K})$ be an optical geometry of dimension $2m+2>4$ with twisting non-shearing congruence of null geodesics $\wt{\mc{K}}$. Suppose that the  Weyl tensor satisfies
	\begin{align*}
		\wt{\Weyl}_{a b c d} \wt{k}^{a} \wt{v}^{b} \wt{k}^{c} & = 0 \, , & \mbox{for any $\wt{k} \in \Gamma (\wt{K})$, $\wt{v} \in \Gamma (\wt{K}^\perp)$,}
	\end{align*}
	and $\wt{\mbf{c}}$ admits an almost half-Einstein scale. Then $(\wt{\mc{M}},\wt{\mbf{c}},\wt{K})$ is locally conformally isometric to a perturbed Fefferman space.
\end{cor}

The above results can easily be adapted to the case where the optical geometry under consideration is a perturbed Fefferman space. We leave the details for Proposition \ref{prop:alwkhalfE} to the reader, while for Proposition \ref{prop:al_half-Einstein}, we derive:
\begin{thm}\label{thm:hlfEinsc}
	Let $(\wt{\mc{M}},\wt{\mbf{c}}_{\wt{\xi}},\wt{k}) \longrightarrow (\mc{M},H,J)$ be a perturbed Fefferman space of dimension $2m+2>4$. Suppose that the Weyl tensor satisfies
	\begin{align}
		\wt{\Weyl}_{a b c d} \wt{k}^{a} \wt{v}^{b} \wt{k}^{c} & = 0 \, , & \mbox{for any $\wt{v} \in \Gamma (\langle\wt{k}\rangle^\perp)$,} \label{eq:PetrovIIa3} 
	\end{align}
	Then $\wt{\mbf{c}}_{\wt{\xi}}$ admits an almost half-Einstein scale $\wt{\sigma}$ if and only if there exists a nowhere vanishing density $\sigma \in \Gamma(\mc{E}(1,0))$ such that the CR scale $s := \sigma \ol{\sigma}$ satisfies
	 \begin{subequations}\label{eq:CRscalealK_B}
	 \begin{align}
	 	& \nabla_{(\alpha} \nabla_{\beta)} s + \i \WbA_{\alpha \beta} s + \Nh_{\gamma (\alpha \beta)} \nabla^{\gamma} s = 0 \, , \\
	 	& \nabla^{\gamma} s \Nh_{\gamma (\alpha \beta)} - \tfrac{1}{m} \nabla^{\gamma} \Nh_{\gamma (\alpha \beta)} s = 0 \, ,
	 \end{align}
	 \end{subequations}
	 and $\wt{\xi}$ is determined\label{item:2A} by the CR data $\left(\bm{\xi}_{\alpha}^{(0)}, [\nabla , \bm{\xi}_{0}^{(0)}], \bm{\xi}_{0}^{(2k)} \right)_{k=1,\ldots, m+1}$ given by
	 \begin{subequations}\label{eq:xi0}
	 	\begin{align*}
	 		\bm{\xi}_{\alpha}^{(0)}  & = - \i \sigma^{-1} \nabla_{\alpha} \sigma  \, , \\
	 		\bm{\xi}_0^{(2m+2)} & = \left( \tfrac{m!(m+1)!}{2(2m+2)!} \left( (2m+1) \Lambda - (2m+2) \wt{\Lambda} \right) + \mu \right) \sigma^{m} \ol{\sigma}^{-m-2} \, , \\
	 		\bm{\xi}_0^{(2k)} & = \tfrac{2(2m+1)!}{(m+1-k)!(m+1+k)!} \left( k \bm{\xi}_0^{(2m+2)} + (m+1-k)\Re(\mu) \right) \sigma^{k-1} \ol{\sigma}^{-k-1} \, ,  & 1 \leq |k| \leq m \, , \\
	 		\bm{\xi}_0^{(0)} & =  \tfrac{\i}{m} \left( \nabla_{\alpha} \bm{\xi}^{\alpha}_{(0)} - \nabla^{\alpha} \bm{\xi}_{\alpha}^{(0)} \right) + \tfrac{2(2m+1)!}{m!(m+1)!} \Re(\mu) \sigma^{-1} \ol{\sigma}^{-1} \, ,
	 	\end{align*}
	 \end{subequations}
	 for some real constant $\wt{\Lambda}$, real-valued function $\Lambda := \tfrac{1}{m} \left( \Sc - \| \Nh \|^2 \right)$, and complex-valued function $\mu$ on $(\mc{M},H,J)$.

Moreover, the zero set $\wt{\mc{Z}}$ of $\wt{\sigma}$ consists of the union $\wt{\mc{Z}}_+ \cup \wt{\mc{Z}}_-$ of the sections $\wt{\mc{Z}}_{\pm} = [\pm \i \sigma^{-1}] : \mc{M} \rightarrow \wt{\mc{M}}$. Off $\wt{\mc{Z}}$, $\wt{\sigma}$ determines a half-Einstein metric $\wt{g} = \sec^2 \phi \cdot \wt{g}_{\theta, \wt{\xi}}$, with Ricci scalar $(2m+2) \wt{\Lambda}$, where $\phi$ is the local fibre coordinate determined by $\sigma$ and $\wt{g}_{\theta, \wt{\xi}}$ is the perturbed Fefferman metric associated to the contact form $\theta = (\sigma \ol{\sigma})^{-1} \bm{\theta}$.
\end{thm}

\begin{proof}
	This is really a corollary of Proposition \ref{prop:al_half-Einstein}, where we take our optical geometry to be a perturbed Fefferman space in the first place. We begin by interpreting the coordinate $\phi$ as being defined by a distinguished nowhere vanishing density $\sigma$ of weight $(1,0)$ as in Proposition \ref{prop:density_Ric}, which is related to the CR scale as $s = \sigma \ol{\sigma}$. We can determine the components of the perturbation one-form $\wt{\xi}$ by expressing the perturbed Fefferman metric as \eqref{eq:metNSFeff}, and we finally turn these into densities by applying Lemma \ref{lem:CR_data} --- see relations \eqref{eq:Fourier_coef}. The converse works by reversing the steps.
\end{proof}

\subsection{Almost Einstein scales}
Our next step would be to consider almost pure radiation scales, however, the following lemma, which is a straightforward reformulation of a result in \cite{TaghaviChabert2022}, will give us a short-cut to almost Einstein scales.
\begin{lem}\label{lem:pertFeffpurad}
	An almost pure radiation scale on a perturbed Fefferman space of dimension $2m+2>4$ whose Weyl tensor satisfies \eqref{eq:PetrovIIa3} is necessarily an almost Einstein scale.
\end{lem}

We can presently finish this section with a description of perturbed Fefferman spaces admitting almost Einstein scales.
\begin{thm}\label{thm:pertFeffpurad}
	Let $(\wt{\mc{M}},\wt{\mbf{c}}_{\wt{\xi}},\wt{k}) \longrightarrow (\mc{M},H,J)$ be a perturbed Fefferman space of dimension $2m+2>4$. Suppose that the Weyl tensor satisfies \eqref{eq:PetrovIIa3}. Then $\wt{\mbf{c}}_{\wt{\xi}}$ admits an almost Einstein scale $\wt{\sigma}$ if and only if there exists a nowhere vanishing density $\sigma \in \Gamma(\mc{E}(1,0))$ such that $s := \sigma \ol{\sigma}$ is an CR--Einstein scale, i.e.\ $s$ satisfies
	\begin{subequations}\label{eq:CR-E-inv_2}
		\begin{align}
			& \nabla_{(\alpha} \nabla_{\beta)} s + \i \WbA_{\alpha \beta} s + \Nh_{\gamma (\alpha \beta)} \nabla^{\gamma} s = 0 \, ,  \\
			& \nabla^{\gamma} s \Nh_{\gamma (\alpha \beta)} - \tfrac{1}{m} \nabla^{\gamma} \Nh_{\gamma (\alpha \beta)} s = 0 \, ,  \\
			& \left( s \nabla_{\bar{\beta}} \nabla_{\alpha} s - \nabla_{\alpha} s  \nabla_{\bar{\beta}} s + \Rho_{\alpha \bar{\beta}} s^2 - \tfrac{1}{m+2} \Nh_{\alpha \gamma \delta} \Nh_{\bar{\beta}}{}^{\gamma \delta} s^2 \right)_\circ = 0 \, ,
		\end{align}
	\end{subequations}
	 and $\wt{\xi}$ is determined by the CR data $\left(\bm{\xi}_{\alpha}^{(0)},  [\nabla,\bm{\xi}_{0}^{(0)}], \bm{\xi}_{0}^{(2k)} \right)_{k = 1, \ldots, m+1}$ given by
	\begin{subequations}\label{eq:xi0b}
	\begin{align}
		\bm{\xi}_{\alpha}^{(0)}  & = - \i \sigma^{-1} \nabla_{\alpha} \sigma  \, , \label{eq:xi0balf} \\
		\bm{\xi}_0^{(2m+2)} & = \left( \tfrac{m!(m+1)!}{2(2m+2)!} \left( (2m+1) \Lambda - (2m+2) \wt{\Lambda} \right) + \mu \right) \sigma^{m} \ol{\sigma}^{-m-2} \, , \\
		\bm{\xi}_0^{(2k)} & = \tfrac{2k(2m+1)!}{(m+1-k)!(m+1+k)!} \bm{\xi}_0^{(2m+2)} \sigma^{k-1} \ol{\sigma}^{-k-1} \, ,  & 1 \leq |k| \leq m \, , \\
		\bm{\xi}_0^{(0)} & =  \tfrac{\i}{m} \left( \nabla_{\alpha} \bm{\xi}^{\alpha}_{(0)} - \nabla^{\alpha} \bm{\xi}_{\alpha}^{(0)} \right) \, , \label{eq:xi0b0}
	\end{align}
\end{subequations}
	for some real constants $\wt{\Lambda}$ and $\Lambda := \tfrac{1}{m} \left( \Sc - \| \Nh \|^2 \right)$, and complex constant $\mu$ on $(\mc{M},H,J)$ satisfying $\ol{\mu}= -\mu$.

	Moreover, the zero set $\wt{\mc{Z}}$ of $\wt{\sigma}$ consists of the union $\wt{\mc{Z}}_+ \cup \wt{\mc{Z}}_-$ of the sections $\wt{\mc{Z}}_{\pm} = [\pm \i \sigma^{-1}] : \mc{M} \rightarrow \wt{\mc{M}}$. Off $\wt{\mc{Z}}$, $\wt{\sigma}$ determines an Einstein metric $\wt{g} = \sec^2 \phi \cdot \wt{g}_{\theta, \wt{\xi}}$, with Ricci scalar $(2m+2) \wt{\Lambda}$, where $\phi$ is the local fibre coordinate determined by $\sigma$ and $\wt{g}_{\theta, \wt{\xi}}$ is the perturbed Fefferman metric associated to the contact form $\theta = (\sigma \ol{\sigma})^{-1} \bm{\theta}$.
\end{thm}

\begin{proof}
	Our assumption that $\wt{\sigma}$ is an almost Einstein scale clearly allows us to re-use the setting of Proposition \ref{prop:al_half-Einstein} and Theorem \ref{thm:hlfEinsc} to which this proof will refer. The only difference is that there are additional constraints on the Ricci tensor of the metric $\wt{g} = \wt{\sigma}^{-2} \wt{\bm{g}}$. First, the condition $\wt{\Ric}(\wt{v},\wt{w}) = \wt{\Lambda} \wt{g}(\wt{v},\wt{w})$ for any $\wt{v}, \wt{w}  \in \Gamma(\langle \wt{k} \rangle^\perp)$ has the following two consequences:
	\begin{itemize}
		\item The pseudo-Hermitian structure $\theta$ is CR--Einstein, and so, by Theorem \ref{thm:alCRE}, the corresponding CR scale $s$ must satisfy \eqref{eq:CR-E-inv_2} in agreement with Theorem \ref{thm:hlfEinsc}. This also means that $\Lambda := \tfrac{1}{m} \left( \Sc - \| \Nh \|^2 \right)$ is now constant.
		\item The complex-valued function $\wt{\lambda}_{0}$ is subject to the ordinary differential equation
		\begin{align}\label{eq:ode2}
			\tan \phi \dot{\wt{\lambda}}_{0} - \left(2 (m+1) - (2m+1)\sec^2 \phi \right) \wt{\lambda}_{0} - \wt{\Lambda} \sec^2 \phi + \Lambda = 0 \, ,
		\end{align}
		in addition to \eqref{eq:ode1} --- see \cite[equation~(8.18)]{TaghaviChabert2022}. Plugging the solution \eqref{eq:ODE1sol} of \eqref{eq:ode1} with \eqref{eq:coeff_step3} into \eqref{eq:ode2} determines $\mu$ as a purely imaginary function, i.e.\ $\Re(\mu) = 0$, which in comparison with the form of the CR data given in Theorem \ref{thm:hlfEinsc}, yields \eqref{eq:xi0b}.
	\end{itemize}
	The next condition, $\wt{\Ric}(\wt{v},\cdot) = \wt{\Lambda} \wt{g}(\wt{v},\cdot)$ for any $\wt{v} \in \Gamma(\langle{\wt{k}}\rangle^\perp)$, then forces $\mu$ to be a constant as shown in \cite{TaghaviChabert2022}. This is the equation characterising $\wt{g}$ as a pure radiation metric, but we know by Lemma \ref{lem:pertFeffpurad} that this is equivalent to $\wt{g}$ being an Einstein metric. At this stage, we have exhausted all the conditions. The proof is finally completed by following the same line of arguments as in Proposition \ref{prop:al_half-Einstein} and Theorem \ref{thm:hlfEinsc}.
\end{proof}

As a corollary, we find:
\begin{cor}\label{cor:Einstein_Fefferman}
	Let $(\wt{\mc{M}},\wt{\mbf{c}}^{(\alpha)},\wt{k}) \longrightarrow (\mc{M},H,J)$ be an $\alpha$-Fefferman space that admits an almost Einstein scale $\wt{\sigma}$. Then $\alpha=1$.
	
	This being the case, the almost CR manifold $(\mc{M},H,J)$ admits a CR--Einstein structure that arises from a nowhere vanishing density $\sigma \in \Gamma(\mc{E}(1,0))$ that satisfies \eqref{eq:spalCREall}.
	
	In addition, the Ricci scalar of the Einstein metric defined by $\wt{\sigma}$ off its zero set is given by $\wt{\Sc} = \tfrac{2m+1}{m} \left( \Sc - \| \Nh \|^2 \right)$, where $\Sc$ is the Webster--Ricci scalar of the pseudo--Hermitian structure defined by the CR scale $\sigma \ol{\sigma}$ and $\Nh$ is the Nijenhuis tensor in that scale.
\end{cor}

\begin{proof}
	We first note that any $\alpha$-Fefferman metric $\wt{g}_{\theta}^{(\alpha)}$ can be expressed as a perturbed $1$-Fefferman metric, i.e.\
	\begin{align*}
		\wt{g}_{\theta}^{(\alpha)} & = \wt{g}_{\theta}^{(1)} + \wt{\xi} =: \wt{g}_{\theta,\xi} \, , & \mbox{where} & & \wt{\xi} & = \tfrac{1-\alpha}{2m(2m+1)} \| \Nh \|^2 \theta \, .
	\end{align*}
	Now, let $\wt{g}$ be the Einstein metric defined off the zero set of the almost Einstein scale $\wt{\sigma}$. Then, by Theorem \ref{thm:pertFeffpurad}, there exists a nowhere vanishing density $\sigma \in \Gamma(\mc{E}(1,0))$ satisfying equations \eqref{eq:CR-E-inv_2} and such that $\wt{g} = \sec^2 \phi \cdot \wt{g}_{\theta,\wt{\xi}}$ where $\wt{g}_{\theta,\wt{\xi}}$ is the perturbed $1$-Fefferman metric associated to it. From the form of $\wt{\xi}$, we already know that $\xi_{\alpha}^{(0)} = 0$, and so, by \eqref{eq:xi0balf}, $\nabla_{\bar{\alpha}} \sigma = 0$. By Theorem \ref{thm:spalCRE}, this already tells us that $\sigma$ in fact satisfies \eqref{eq:spalCREall}.
	
	Next, by \eqref{eq:xi0b0}, we have $\xi_{0}^{(0)} = 0$ too. But this implies $\tfrac{1-\alpha}{2m(2m+1)} \| \Nh \|^2 = 0$, i.e.\ $\alpha = 1$. The vanishing of the other coefficients implies $(2m+1) \Lambda - (2m+2) \wt{\Lambda} = \mu = 0$, which determines the relation between $\wt{\Sc}$, $\Sc$ and $\|\Nh\|^2$.
\end{proof}

\begin{rem}	
	For consistency with Corollary \ref{cor:Einstein_Fefferman}, we can compute the integrability condition \eqref{eq-Wkk} when $\wt{\mbf{c}}^{(\alpha)}$ admits an almost Einstein scale. Using \cite[equation~(8.24f)]{TaghaviChabert2022}, we find
	\begin{align*}
		\tfrac{1}{2m+1} \left( \Lambda - \tfrac{2m+2}{2m+1} \wt{\Lambda} \right) & = \tfrac{\alpha - 1}{2m+1} \| \Nh \|^2 \, ,
	\end{align*}
	which can only be satisfied when $\alpha=1$.
\end{rem}

\section{Properties of the zero set of almost (half-)Einstein scales}\label{sec:asymptotics}
Let $\wt{\sigma}$ be an almost Lorentzian scale on a perturbed Fefferman space $(\wt{\mc{M}},\wt{\mbf{c}}_{\wt{\xi}},\wt{k}) \longrightarrow (\mc{M},H,J)$ of dimension $2m+2\geq4$. Then, for each choice of Levi-Civita connection $\wt{\nabla}$, we have a weighted one-form
\begin{align}\label{eq:nu}
	\wt{\bm{\nu}}_{a} & := \wt{\nabla}_{a} \wt{\sigma} \in \Gamma( \wt{\mc{E}}_{a}[1]) \, .
\end{align}
This depends on the choice of metric in $\wt{\mbf{c}}_{\wt{\xi}}$ since under a conformal change, $\wt{\bm{\nu}}_{a}$ transforms as $\wh{\wt{\bm{\nu}}}_{a} = \wt{\bm{\nu}}_{a} + \wt{\Upsilon}_{a} \wt{\sigma}$. However, on restriction to $\wt{\mc{Z}}$, any of these weighted forms coincide with the weighted normal covector of $\wt{\mc{Z}}$ \cite{Curry2018}.

Let us assume further that $\wt{\sigma}$ is an almost half-Einstein scale. By Theorem \ref{thm:hlfEinsc}, $\wt{\sigma}$ is determined by some nowhere vanishing $\sigma \in \Gamma(\mc{E}(1,0))$ trivialising $\wt{\mc{M}} \rightarrow \mc{M}$ with fibre coordinate $\phi$ and with contact form $\theta = (\sigma \ol{\sigma})^{-1} \bm{\theta}$, associated Fefferman scale $\wt{\sigma}_{\theta,\wt{\xi}}$ so that $\wt{\sigma} = \cos \phi \cdot \wt{\sigma}_{\theta,\wt{\xi}}$. Its zero set is given by $\wt{\mc{Z}} = \wt{\mc{Z}}_+  \cup \wt{\mc{Z}}_-$ where $\wt{\mc{Z}}_\pm := [\pm \i \sigma^{-1}] : \mc{M} \rightarrow \wt{\mc{M}}$. This means two things: First, for any choice of metric $\wt{g}$ in $\wt{\mbf{c}}$ with Levi-Civita connection $\wt{\nabla}$, the weighted one-form \eqref{eq:nu} reads as
\begin{align*}
	\wt{\bm{\nu}} & = - \wt{\sigma}_{\theta,\wt{\xi}} \cdot \sin \phi \cdot \d \phi   + \cos \phi \cdot \wt{\nabla} \wt{\sigma}_{\theta,\wt{\xi}} \, .
\end{align*}
Second, we may choose our metric $\wt{g}$ to be such that $\wt{\nabla}$ preserves the distinguished perturbed Fefferman scale $\wt{\sigma}_{\theta,\wt{\xi}}$. This singles out the one-form of weight $1$ and vector field of weight $-1$
\begin{align*}
	\accentset{(\theta)}{\wt{\bm{\nu}}} & = - \wt{\sigma}_{\theta,\wt{\xi}} \cdot \sin \phi \cdot \d \phi \, , & 	\accentset{(\theta)}{\wt{\bm{n}}} & = \wt{\bm{g}}^{-1} ( \accentset{(\theta)}{\wt{\bm{\nu}}} , \cdot ) \, .
\end{align*}
Computing the squared length of $\accentset{(\theta)}{\wt{\bm{n}}}$ on restriction to $\wt{\mc{Z}}$, we find 
\begin{align*}
	\wt{\bm{g}} (\accentset{(\theta)}{\wt{\bm{n}}}, 	\accentset{(\theta)}{\wt{\bm{n}}}) |_{\wt{\mc{Z}}} & = -\tfrac{1}{2m+1} \wt{\Lambda} \, .
\end{align*}

Now, recall that $\wt{\mc{Z}}$ is said to be \emph{spacelike}, \emph{timelike} or \emph{null} if its normal is timelike, spacelike or null respectively. Hence, the previous display 
immediately yields the following proposition --- see also \cite{Penrose1965,Penrose1986} for the four-dimensional case.
\begin{prop}\label{prop:redE_hypsrf}
	Let $(\wt{\mc{M}},\wt{\mbf{c}}_{\wt{\xi}},\wt{k}) \longrightarrow (\mc{M},H,J)$ be a perturbed Fefferman space of dimension $2m+2 \geq 4$ that admits an almost half-Einstein scale $\wt{\sigma}$, so that off the zero set $\wt{\mc{Z}}$ of $\wt{\sigma}$, the metric $\wt{g} = \wt{\sigma}^{-2} \wt{\bm{g}}$ is half-Einstein with constant Ricci scalar $(2m+2) \wt{\Lambda}$. Then
	\begin{itemize}
		\item $\wt{\mc{Z}}$ is null if and only if $\wt{\Lambda} = 0$;
		\item $\wt{\mc{Z}}$ is spacelike if and only if $\wt{\Lambda} > 0$;
		\item $\wt{\mc{Z}}$ is timelike if and only if $\wt{\Lambda} < 0$.
	\end{itemize}
\end{prop}

\begin{rem}
	In general, while $\wt{\mc{Z}}$ is expected to have merely a conformal, rather than metric, structure --- see for instance \cite{Curry2018} --- in our case, by virtue of the almost half-Einstein equations, we do in fact have a distinguished regular metric on $\wt{\mc{Z}}$, namely, the perturbed Fefferman metric restricted to $\wt{\mc{Z}}$:
	\begin{align*}
		\wt{g}_{\theta,\wt{\xi}} |_{\wt{\mc{Z}}} & = 4 \theta \odot \left( \tfrac{\i}{2} \left( \sigma^{-1} \nabla \sigma - \ol{\sigma}^{-1} \nabla \ol{\sigma} \right)   - \left( \tfrac{1}{m+2} \Rho + \tfrac{1}{2m(m+1)} \| \Nh \|^2 \right) \theta  + \wt{\xi}|_{\wt{\mc{Z}}}  \right) + h \, .
	\end{align*}
	This clearly depends on the pseudo-Hermitian data.
	Proposition \ref{prop:redE_hypsrf} can then alternatively be derived by computing the determinant $\det \left( \left.\wt{g}_{\theta,\wt{\xi}} \right|_{\wt{\mc{Z}}} \right) = -\tfrac{4}{2m+1} \wt{\Lambda}$. When $\accentset{(\theta)}{\wt{\bm{n}}}$ is null, it is tangent to $\wt{\mc{Z}}$ and the metric $\wt{g}_{\theta,\wt{\xi}}$ is degenerate. 
\end{rem}

We end this section with a result on the curvature properties on $\wt{\mc{Z}}$:
\begin{thm}\label{thm:conformal_flat_asym}
	Let $(\wt{\mc{M}},\wt{\mbf{c}}_{\wt{\xi}},\wt{k}) \rightarrow (\mc{M},H,J)$ be a perturbed Fefferman space of dimension $2m+2>4$ that admits an almost Einstein scale $\wt{\sigma}$ with zero set $\wt{\mc{Z}}$, so that off $\wt{\mc{Z}}$, the metric $\wt{g} = \wt{\sigma}^{-2} \wt{\bm{g}}$ is Einstein with constant Ricci scalar $(2m+2) \wt{\Lambda}$. Then $\wt{\mbf{c}}_{\wt{\xi}}$ is conformally flat on $\wt{\mc{Z}}$ if and only if $(\mc{M},H,J)$ is CR flat and the CR--Einstein structure has Webster--Ricci scalar $\tfrac{m(2m+2)}{2m+1} \wt{\Lambda}$.
\end{thm}

\begin{proof}
	This follows from a direct computation.	The Weyl curvature of the perturbed Fefferman metric is given by \cite[equations~(8.24)]{TaghaviChabert2022}. Evaluating these at $\phi = \pm \tfrac{\pi}{2}$ gives us the Weyl curvature on $\wt{\mc{Z}}$, and it is then straightforward, if  somewhat tedious, to check that the Weyl tensor vanishes on $\wt{\mc{Z}}$ if and only if $\Nh_{\alpha \beta \gamma} = 0$, $\mathsf{S}_{\alpha \bar{\beta} \gamma \bar{\delta}} =0$, i.e.\ $(\mc{M},H,J)$ is CR flat, and $\frac{1}{2m+2} \Lambda - \frac{1}{2m+1} \wt{\Lambda}=0$.
\end{proof}

\begin{rem}
	The above result should be contrasted with \cite[Theorem~7.13]{TaghaviChabert2023} in dimension four.
\end{rem}

\section{Comments on conformal symmetries}\label{sec:sym}
Given a perturbed Fefferman space $(\wt{\mc{M}},\wt{\mbf{c}}_{\wt{\xi}},\wt{k}) \longrightarrow (\mc{M},H,J)$, one may naturally ask how conformal symmetries of $(\wt{\mc{M}},\wt{\mbf{c}}_{\wt{\xi}},\wt{k})$ relate to solutions to CR invariant differential equations on its base $(\mc{M},H,J)$. While this question goes beyond the scope of this article, we shall nevertheless outline some ideas towards the answer.

Let us first review the situation in the unperturbed case, i.e.\ $\wt{\xi} = 0$ and when $(\mc{M},H,J)$ is a CR manifold. It is shown in \cite{Cap2008} that the space of all conformal Killing fields on a Fefferman space $(\wt{\mc{M}},\wt{\mbf{c}},\wt{k})$ splits into a direct sum of three spaces
\begin{align}\label{eq:CK_decomp}
	\wt{\mc{A}} \oplus \wt{\mc{B}} \oplus \langle \wt{k} \rangle \, ,
\end{align}
where
\begin{enumerate}
	\item any section of $\wt{\mc{A}}$ is transverse to $\langle \wt{k} \rangle^\perp$ and arises from a real-valued solution $s \in \Gamma(\mc{E}(1,1))$ to
	\begin{align*}
		\nabla_{\alpha} \nabla_{\beta} s + \i \WbA_{\alpha \beta} s & = 0 \, ,
	\end{align*}
	and a nowhere vanishing solution gives rise to transverse infinitesimal CR symmetries as we have already seen in Section \ref{sec:CR_geom};
	\item any section of $\wt{\mc{B}}$ is tangent to $\langle \wt{k} \rangle^\perp$, but not to $\langle \wt{k} \rangle$, and arises from a solution $w^{\alpha} \in \Gamma(\mc{E}^{\alpha}(-1,1))$ to the CR invariant differential equations
	\begin{align*}
		\nabla^{(\alpha} w^{\beta)} & = 0 \, , & \nabla_{\alpha} w^{\beta} & = \tfrac{1}{m}\delta_{\alpha}^{\beta} \nabla_{\gamma} w^{\gamma} \, .
	\end{align*}
	In dimension three, however, the second condition is vacuous, while the first one reduces to $\nabla^{\alpha} w^{\beta} = 0$.
\end{enumerate}
For conformally flat $(\wt{\mc{M}},\wt{\mbf{c}},\wt{k})$, the dimensions of $\wt{\mc{A}}$ and $\wt{\mc{B}}$ are $(m+1)(m+3)$ and $(m+1)(m+2)$ respectively, the former being the maximal dimension of the automorphism group of $(\mc{M},H,J)$ as expected. Together with $\langle \wt{k} \rangle$, these indeed add up to $(m+2)(2m+3)$, the maximal dimension of the automorphism group of $(\wt{\mc{M}},\wt{\mbf{c}},\wt{k})$.

There are two ways in which a perturbation of $(\wt{\mc{M}},\wt{\mbf{c}},\wt{k})$ by a semi-basic one-form $\wt{\xi}$ affects the decomposition \eqref{eq:CK_decomp}: first, $\wt{k}$ itself will cease to be conformal Killing (unless $\wt{k} \hook \d \wt{\xi} = 0$). Second, the perturbation one-form will affect the conformal curvature, and in particular, the dimension of each summand in \eqref{eq:CK_decomp} will not be preserved.

Interestingly, if we start with a conformally flat $(\wt{\mc{M}},\wt{\mbf{c}},\wt{k})$, we then end up with non-conformally flat perturbed Fefferman spaces over a flat CR manifold. To illustrate this phenomenon in dimension four, and following \cite{Trautman2002}, the hyperquadric, locally CR equivalent to the CR three-sphere, gives rise to three distinct perturbed Fefferman spaces, each admitting an almost Einstein scale $\wt{\sigma}$ with zero set $\wt{\mc{Z}}$:
\begin{itemize}
	\item the canonical conformally flat Fefferman conformal structure that contains the Minkowski metric off $\wt{\mc{Z}}$;
	\item a perturbed Fefferman conformal structure that contains the Petrov type D Taub-NUT metric \cite{Taub1951,Newman1963} off $\wt{\mc{Z}}$;
	\item a perturbed Fefferman conformal structure that contains the Petrov type N Hauser metric \cite{Hauser1974} off $\wt{\mc{Z}}$.
\end{itemize}

Incidentally, since in dimension four, the `unperturbed' Fefferman conformal structure admits an almost Einstein scale if and only if it is conformally flat (or equivalently its underlying CR structure is flat), the main point of these perturbations is to enlarge the range of possible almost Einstein scales.

In fact, a four-dimensional conformally flat space $(\wt{\mc{M}},\wt{\mbf{c}})$ can be viewed as a (perturbed) Fefferman space in `infinitely many' ways: The space of all null geodesics in $\wt{\mc{M}}$ is a contact CR five-fold $\mbb{PN}$ of signature $(1,1)$, and any contact CR submanifold of $\mbb{PN}$ of dimension three gives rise to a twisting non-shearing congruences of null geodesics on $\wt{\mc{M}}$. In the analytic category, the so-called \emph{Kerr theorem} tells us that such CR three-manifolds can be constructed as the intersection of $\mbb{PN} \subset \CP^3$ with a complex surface in $\CP^3$, the \emph{twistor space} of $(\mc{M},\wt{\mbf{c}})$ \cite{Penrose1967,Penrose1986}. For instance, the `massless' Kerr metric is flat, but its underlying CR three-manifold admits a two-dimensional group of symmetry.

\begin{rem}
	In  \cite{Lewandowski1990}, the authors consider algebraically special pure radiation metrics, which can therefore be treated as almost pure radiation scales on a perturbed Fefferman space, with underlying contact CR three-manifold admitting an automorphism group of \emph{submaximal} dimension (in this case, three).
\end{rem}

\begin{rem}
When $(\mc{M},H,J)$ is a strictly almost CR manifold, we would expect the decomposition \eqref{eq:CK_decomp} to carry over, but we would certainly have to change the interpretation of the bundle $\wt{\mc{A}}$, and possibly $\wt{\mc{B}}$. There is also the added complication of the parameter $\alpha$ in Definitions \ref{defn-alpha-Fefferman} and \ref{defn-alpha-Fefferman-space}. We shall not attempt to answer these questions at this stage.
\end{rem}

\appendix

\section{A strictly almost CR manifold admitting a CR--Einstein structure}\label{app:NP}
	According to \cite[Corollary~4.1 and Proposition~4.4]{TaghaviChabert2022}, locally, the anti-canonical circle bundle of any almost K\"{a}hler--Einstein manifold admits a CR--Einstein structure, and any almost CR manifold admitting a CR--Einstein structure locally arises in this way --- see also Proposition \ref{prop:CRE-KE}. This allows us to produce examples of such almost CR manifolds. For the integrable case, see \cite{Leitner2007}.  We now construct an example of a CR--Einstein structure on a strictly almost CR manifold, based on the strictly almost K\"{a}hler--Einstein four-manifold explicitly given by Nurowski and Przanowski \cite{Nurowski1999}.
	
	We let $\mc{M}$ be the subset of $\R^5$ with coordinates $(u,z^\alpha, \ol{z}^{\bar{\alpha}})_{\alpha=1,2}$ subject to the condition $z^1 + \ol{z}^{\bar{1}} - 2 z^2 \ol{z}^{\bar{2}}>0$. Define
	\begin{align*}
		\theta & =
			\d u + \i z^1 \d z^2 - \i z^2 \d z^1  - \i \ol{z}^{\bar{1}} \d \ol{z}^{\bar{2}} + \i \ol{z}^{\bar{2}} \d \ol{z}^{\bar{1}} \, , \\
		\theta^1 & = f^{-\tfrac{1}{4}} \cdot \left( \d z^1 - 2 \ol{z}^{\bar{2}} \cdot \d z^2 + f^{\tfrac{1}{2}} \cdot \d \ol{z}^{\bar{2}} \right)  \, , & \ol{\theta}{}^{\bar{1}} & = \ol{\theta^{1}} \, , \\
		\theta^2 & = f^{-\tfrac{1}{4}} \cdot \left( - \d \ol{z}^{\bar{1}} + 2 z^{2} \cdot \d \ol{z}^{\bar{2}} + f^{\tfrac{1}{2}} \cdot \d z^2 \right) \, , & \ol{\theta}{}^{\bar{2}} & = \ol{\theta^{2}} \, , 
	\end{align*}
	where
	\begin{align*}
		f & := 4 \cdot (z^1 + \ol{z}^{\bar{1}} - 2 z^2 \ol{z}^{\bar{2}}) \, .
	\end{align*}
	Note that
	\begin{align*}
		\d f & = 2 f^{\tfrac{1}{4}} \cdot \left( \theta^1 - \theta^2 + \ol{\theta}{}^{\bar{1}} - \ol{\theta}{}^{\bar{2}} \right) \, . 
	\end{align*}
	Then
	\begin{align*}
		\d \theta & = \i \theta^1 \wedge \ol{\theta}{}^{\bar{1}} + \i \theta^2 \wedge \ol{\theta}{}^{\bar{2}} \, ,\\
		\d \theta^1 & = f^{-\tfrac{3}{4}} \left( - \tfrac{1}{2} \theta^1 \wedge \theta^2 - \tfrac{1}{2} \theta^1 \wedge \ol{\theta}{}^{\bar{1}} + \tfrac{1}{2}  \theta^1 \wedge \ol{\theta}{}^{\bar{2}} + \ol{\theta}{}^{\bar{1}} \wedge \ol{\theta}{}^{\bar{2}} \right) \, , \\
		\d \theta^2 & = f^{-\tfrac{3}{4}} \left( - \tfrac{1}{2} \theta^1 \wedge \theta^2 - \tfrac{1}{2} \theta^2 \wedge \ol{\theta}^{\bar{1}} + \tfrac{1}{2}  \theta^2 \wedge \ol{\theta}^{\bar{2}} + \ol{\theta}^{\bar{1}} \wedge \ol{\theta}^{\bar{2}} \right) \, ,
	\end{align*}
	so that $(\theta,\theta^{\alpha},\ol{\theta}{}^{\bar{\alpha}})$ define a strictly almost CR structure $(H, J)$ on $\mc{M}$ with positive definite Levi form. The one-form of the Webster connection $\nabla$ associated to the pseudo-Hermitian structure $\theta$ is given by
	\begin{align*}
		\Gamma_{1}{}^{1} & = \tfrac{1}{2} f^{-\tfrac{3}{4}} \left( \theta^1 - \theta^2 - \ol{\theta}{}^{\bar{1}} + \ol{\theta}{}^{\bar{2}} \right) \, , &
		\Gamma_{2}{}^{2} & = \tfrac{1}{2} f^{-\tfrac{3}{4}} \left( \theta^1 - \theta^2 - \ol{\theta}{}^{\bar{1}} + \ol{\theta}{}^{\bar{2}} \right) \, , &
		\Gamma_{1}{}^{2} & = \Gamma_{2}{}^{1} = 0 \, ,
	\end{align*}
	and the Webster torsion and Nijenhuis tensor by
	\begin{align}
		\WbA_{\alpha \beta} & = 0 \, , \label{eq:A_NPrz} \\
		\Nh_{1 2 1} = \Nh_{1 2 2} & = - f^{-\tfrac{3}{4}} \, , \label{eq:N_NPrz}
	\end{align}
	respectively. Equation \eqref{eq:A_NPrz} tells us that the Reeb vector field $\parderv{}{u}$ generates a transverse symmetry of $(\mc{M},H,J)$, and following \cite[Section~4.3]{TaghaviChabert2022}, the four-dimensional leaf space $\ul{\mc{M}}$ of this Reeb foliation inherits a strictly almost K\"{a}hler structure $(\ul{h},\ul{J})$, which is precisely the one given in \cite{Nurowski1999}. The contact form $\theta$ can then be identified as the induced connection one-form on the anti-canonical circle bundle over $(\ul{\mc{M}},\ul{h},\ul{J})$ with local fibre coordinate $u$.
	
	On the other hand, equation \eqref{eq:N_NPrz} implies
\begin{subequations}\label{eq:N2NP}
	\begin{align}
		\Nh_{\gamma \delta 1} \Nh^{\gamma \delta}{}_{\bar{1}} & = \Nh_{\gamma \delta 1} \Nh^{\gamma \delta}{}_{\bar{2}} = \Nh_{\gamma \delta 2} \Nh^{\gamma \delta}{}_{\bar{1}} = \Nh_{\gamma \delta 2} \Nh^{\gamma \delta}{}_{\bar{2}} = 2 f^{-\tfrac{3}{2}} \, , \\
		\Nh_{\alpha \gamma \delta} \Nh_{\bar{\beta}}{}^{\gamma \delta} & = 2 f^{-\tfrac{3}{2}} h_{\alpha \bar{\beta}} \, , \label{eq:N2NPab} \\
		\Nh_{\alpha \beta \gamma} \Nh^{\alpha \beta \gamma} & = 4 f^{-\tfrac{3}{2}} \, ,
	\end{align}
\end{subequations}
	and a direct computation will show
	\begin{align}\label{eq:DN_NPrz}
		\nabla_{\bar{\delta}} \Nh_{\gamma \alpha \beta} & = 0 \, .
	\end{align}

	To compute the curvature two-form $\Omega_{\alpha}{}^{\beta}$ of $\Gamma_{\alpha}{}^{\beta}$, it is convenient to use the identity
	\begin{align*}
		\d ( \theta^1 - \theta^2) & = -\tfrac{1}{2} f^{-\tfrac{3}{4}} ( \theta^1 - \theta^2) \wedge ( \ol{\theta}{}^{\bar{1}} - \ol{\theta}{}^{\bar{2}} ) \, ,
	\end{align*}
	from which we easily obtain
	\begin{align*}
		\Omega_{2}{}^{2} = \Omega_{1}{}^{1}
		& = f^{-\tfrac{3}{2}} ( \theta^1 - \theta^2) \wedge ( \ol{\theta}{}^{\bar{1}} - \ol{\theta}{}^{\bar{2}} ) \, , 
		& \Omega_{2}{}^{1} = \Omega_{2}{}^{1} & = 0 \, ,
	\end{align*}
	i.e.\
	\begin{align*}
		\Riem_{1 \bar{1} \alpha}{}^{\beta} = \Riem_{2 \bar{2} \alpha}{}^{\beta} & = f^{-\tfrac{3}{2}} \delta_{\alpha}^{\beta} \, , &
		\Riem_{1 \bar{2} \alpha}{}^{\beta} = \Riem_{2 \bar{1} \alpha}{}^{\beta} & = - f^{-\tfrac{3}{2}} \delta_{\alpha}^{\beta} \, ,
	\end{align*}
	and all other components vanish. In particular, the Webster--Ricci tensor is given by
	\begin{align*}
		\Ric_{\alpha \bar{\beta}} & = 2 f^{-\tfrac{3}{2}} h_{\alpha \bar{\beta}} \, , & \Sc & = 4 f^{-\tfrac{3}{2}} \, ,
	\end{align*}
	so that using \eqref{eq:N2NPab}, we find
	\begin{align}\label{eq:CRE-Nprz}
		\Ric_{\alpha \bar{\beta}} - \Nh_{\alpha \gamma \delta} \Nh_{\bar{\beta}}{}^{\gamma \delta} & = 0 \, .
	\end{align}
	By \eqref{eq:A_NPrz}, \eqref{eq:DN_NPrz} and \eqref{eq:CRE-Nprz}, we can now conclude that $(\mc{M},H,J)$ as defined above admits a CR--Einstein structure --- see equations \eqref{eq:CR-E}. Again, with reference to Proposition \ref{prop:CRE-KE} or \cite[Section~4.3]{TaghaviChabert2022}, this means that the strictly almost K\"{a}hler structure on the leaf space of the Reeb foliation generated by $\parderv{}{u}$ is in fact almost K\"{a}hler--Einstein in agreement with the construction of \cite{Nurowski1999}. 
	
	Let us now turn our attention to the existence of distinguished complex-valued densities on $(\mc{M},H,J)$. Since by \eqref{eq:N2NP},
	\begin{align*}
		2 \Nh_{1 \gamma \delta} \Nh_{\bar{2}}{}^{\gamma \delta} -  \Nh_{\gamma \delta 1} \Nh^{ \gamma \delta}{}_{\bar{2}}  = - 2 f^{-\tfrac{3}{2}} \, , 
	\end{align*}
	condition \eqref{eq:NspalCRE} is not satisfied, which means that the CR--Einstein scale $s$ cannot be expressed as $s = \sigma \ol{\sigma}$ where $\sigma$ is a density of weight $(1,0)$ that satisfies $\nabla_{\bar{\alpha}} \sigma = 0$. We shall nevertheless show that there exists a density $\wh{\sigma}$ of weight $(1,0)$ that satisfies $\nabla_{\bar{\alpha}} \wh{\sigma} = 0$ but $s \neq \wh{\sigma} \ol{\wh{\sigma}}$. To this end, we define
	\begin{align*}
		\sigma & := \left( \theta \wedge \theta^1 \wedge \theta^2 \right)^{-\tfrac{1}{4}} \, ,
	\end{align*}
	which clearly satisfies $s = \sigma \ol{\sigma}$, and compute
	\begin{align*}
		\nabla \sigma & = \i \xi \otimes \sigma \, , & \mbox{where} & & 
		\xi & = - \tfrac{\i}{4} f^{-\tfrac{3}{4}} \left( \theta^1 - \theta^2 - \ol{\theta}{}^{\bar{1}} + \ol{\theta}{}^{\bar{2}} \right) \, .
	\end{align*}
	In particular, since
	\begin{align*}
		\d \xi & = - \i f^{-\tfrac{3}{2}} \left( \theta^1 \wedge \ol{\theta}{}^{\bar{1}} - \theta^1 \wedge \ol{\theta}{}^{\bar{2}} + \theta^2 \wedge \ol{\theta}{}^{\bar{1}} + \theta^2 \wedge \ol{\theta}{}^{\bar{2}} \right) \, ,
	\end{align*}
	is nowhere vanishing, the one-form $\xi$ cannot possibly be exact, which means that there is no density $\sigma'$ of weight $(1,0)$ such that $\sigma' \ol{\sigma}'$ is the CR--Einstein scale $s$ and $\nabla_{\bar{\alpha}} \sigma' = 0$ as claimed earlier.
		
	However, for any $k \in \R$, we have
	\begin{align*}
		\nabla (f^k \sigma) & = f^{k-\tfrac{3}{4}} \left( (2k + \tfrac{1}{4}) (\theta^1 - \theta^2) + (2k - \tfrac{1}{4}) (\ol{\theta}^{\bar{1}} - \ol{\theta}^{\bar{2}}) \right) \otimes \sigma \, .
	\end{align*}
	Hence, taking $k=\tfrac{1}{8}$, we find that $\wh{\sigma} := f^{\tfrac{1}{8}} \sigma$ satisfies
	\begin{align*}
		\nabla \wh{\sigma} & = \tfrac{1}{2} f^{-\tfrac{3}{4}} (\theta^1 - \theta^2) \otimes \wh{\sigma} \, ,
	\end{align*}
	i.e.\
	\begin{align*}
		\nabla_{\bar{\alpha}} \wh{\sigma} & = 0 \, , & s & = f^{-\tfrac{1}{4}} \wh{\sigma} \ol{\wh{\sigma}} \, ,
	\end{align*}
	as required. It is also interesting to note that $\wh{\sigma} \ol{\wh{\sigma}}$ is distinct from $\|\Nh\|^{-2}$.
	
	This example thus shows that strictly almost CR manifolds may admit CR densities, i.e.\ densities of weight $(1,0)$ annihilated by the distribution $H^{(0,1)}$.

\section{Proof of Theorem \ref{thm:Fefferman-CRA}}\label{app:proofchi}
	We essentially follow the proof given in \cite{Graham1987} for the involutive case. We first recall that a null conformal Killing field $\wt{k}$ on $(\wt{\mc{M}},\wt{\mbf{c}})$ can be equivalently expressed in terms of the weighted one-form $\wt{\bm{\kappa}}_{a} = \wt{\bm{g}}_{a b} \wt{k}^{b}$ solving the conformally invariant equation
	\begin{align}\label{eq-CKf}
		\nabla_a \wt{\bm{\kappa}}_b - \wt{\bm{\tau}}_{ab} - \wt{\epsilon} \, \wt{\bm{g}}_{ab} & = 0 \, ,
	\end{align}
	where $\wt{\bm{\tau}}_{ab} = \wt{\bm{\tau}}_{[ab]} = \wt{\nabla}_{[a} \wt{\bm{\kappa}} _{b]}$ and $\wt{\epsilon} = \tfrac{1}{n+2} \wt{\nabla}_c \wt{k}^c$.
In the course of the proof, we shall make use of the additional assumptions
\begin{subequations}
	\begin{align}
	& \tfrac{1}{n^2} (\wt{\nabla}_{a} \wt{k}^{a} )^2 - \wt{\Rho}_{a b} \wt{k}^a \wt{k}^b - \tfrac{1}{n} \wt{k}^a \wt{\nabla}_a \wt{\nabla}_b \wt{k}^b <  0\, , \label{eq:Rickk_neg} \\
	& \wt{k}^a \wt{\Weyl}_{a b c d} \wt{k}^d  = 2 s \wt{\bm{\kappa}}_{b} \wt{\bm{\kappa}}_{c} \| \wt{\Weyl}(\wt{k}) \|^2 \, ,  \label{eq:kkW}\\
	& \wt{k}^a \wt{\Cot}_{a b c} \wt{k}^c  = s \wt{\bm{\kappa}}_{b} \wt{k}^{c} \wt{\nabla}_{c} \| \wt{\Weyl}(\wt{k}) \|^2  \, , \label{eq:kkY}
\end{align}
and
\begin{multline}\label{eq:cond_Weyl}
	\wt{\Weyl}_{a b}{}^{c d} \wt{\bm{\tau}}_{c d} - 2 \wt{k}^c \wt{\Cot}_{c a b} - \tfrac{1}{2} \left( \wt{\bm{\tau}}_{c[a} \wt{k}^{d} \wt{\Weyl}_{b] d}{}^{e f} \wt{\Weyl}_{e f g}{}^{c} \wt{k}^g  + \wt{\bm{\kappa}}_{[a} \wt{k}^{c} \wt{\Weyl}_{b] c}{}^{d e} \wt{\Cot}_{f d e} \wt{k}^f \right) \\
	= t \wt{\nabla}_{[a}  \left( \wt{\bm{\kappa}}_{b]} \| \wt{\Weyl}(\wt{k}) \|^2 \right) \, ,
\end{multline}
\end{subequations}
for some real constants $s$ and $t$, and where $\| \wt{\Weyl}(\wt{k}) \|^2 := \wt{k}^{a} \wt{\Weyl}_{a b c d} \wt{k}^e \wt{\Weyl}_{e}{}^{b c d}$. Conditions \eqref{eq:Rickk_neg}, \eqref{eq:kkW}, \eqref{eq:kkY} and \eqref{eq:cond_Weyl} are none other than the respective hypotheses  \eqref{eq-Rho_sc}, \eqref{eq-Wkk}, \eqref{eq-Ykk} and \eqref{eq-int_cond} of Theorem \ref{thm:Fefferman-CRA} except that the former depend on two parameters, $s$ and $t$, while the latter on only one, $\alpha$. This is for convenience, and an algebraic relation between $s$, $t$ and $\alpha$ will emerge towards the end of the proof.

Let us first assume that equation \eqref{eq-CKf} holds. We can always choose a conformal scale $\wt{\sigma}$ such that $\wt{k}$ is Killing with respect to $\wt{g}_{a b} = \sigma^{-2} \wt{\bm{g}}_{a b}$, , i.e.\ $\mathsterling_{\wt{k}} \wt{g}_{a b} = 0$. In this case, the prolongation of equation \eqref{eq-CKf} (see, e.g.\ \cite{Gover2008}) then reduces to:
 \begin{subequations}
	\begin{align}
		\wt{\nabla}_a \wt{\kappa}_b - \wt{\tau}_{ab} & = 0 \, , \label{eq-prlg_1} \\
		\wt{\nabla}_a \wt{\tau}_{b c} - 2 \, \wt{g}_{a[b} \wt{\psi}_{c]} + 2\,  \wt{\Rho} {_{a[b}} \wt{\kappa}_{c]} - \wt{k}^d \wt{\Weyl}_{dabc} & = 0 \, , \label{eq-prlg_2} \\
		\wt{\Rho}_a{}^b \wt{\kappa}_b + \wt{\psi}_a & = 0 \, , \label{eq-prlg_3} \\
		\wt{\nabla}_a \wt{\psi}_b - \wt{\Rho}_a{}^c \wt{\tau}_{cb} - \wt{\Cot}_{abc} \wt{k}^c & = 0 \, , \nonumber
	\end{align}
	where $\wt{\kappa}_{a} = \wt{g}_{a b} \wt{k}^b$, $\wt{\tau}_{a b} = \wt{\nabla}_{[a} \wt{\kappa}_{b]}$ and $\wt{\psi}_{a}$ is simply defined by \eqref{eq-prlg_3}.
	Upon skew-symmetrisation, the last display becomes
	\begin{align}
		\wt{\nabla}_{[a} \wt{\psi}_{b]} - \wt{\Rho}_{[a}{}^{c} \wt{\tau}_{b]c} + \tfrac{1}{2} \wt{k}^c \wt{\Cot}_{c a b}  & = 0 \, . \label{eq-prlg_4} 
	\end{align}
\end{subequations}

	We shall show that $\wt{c} := \wt{\psi}_c \wt{k}^c$ is constant, $\wt{\tau}_{a b}$ is annihilated by both $\wt{k}^a$ and $\wt{\psi}^a$, and $\wt{\psi}^a$ is null. In this way, we will be able to rescale $\wt{k}^a$ by some constant, so that $\wt{\tau}_{a b}$ will play the r\^{o}le of a bundle Hermitian structure on the screen bundle of $\wt{k}$. Contracting \eqref{eq-prlg_1} with $\wt{k}^a$ already gives
	\begin{align}\label{eq-tau_k}
		\wt{\tau}_{a b} \wt{k}^b & = 0 \, ,
	\end{align}
	which also implies that $\wt{k}^b \wt{\nabla}_b \wt{k}^a = 0$. Contracting \eqref{eq-prlg_2} with $\wt{k}^d$, and using \eqref{eq-prlg_1} and \eqref{eq-prlg_3} yields
	\begin{align}
		\wt{k}^a \wt{\Weyl}_{abcd} \wt{k}^d & = \wt{{\tau}}_a{}^c \wt{{\tau}}_{c b} - (\wt{\psi}_c \wt{k}^c ) \, \wt{{g}}_{ab} +  2 \, \wt{{\kappa}}_{(a} \wt{\psi}_{b)} \, . \label{eq-prlg_2_k}
	\end{align}
	Equation \eqref{eq-prlg_3} and the trace of \eqref{eq-prlg_2_k} give, with $\wt{c} = \wt{k}^{a} \wt{\psi}_{a}$,
	\begin{align}
		\wt{c} &  = - \wt{\Rho}_{a b} \wt{k}^a \wt{k}^b \, , \label{eq-prlg_3_k} \\
		\wt{c} & =  - \tfrac{1}{n} \wt{{\tau}}_{a b} \wt{{\tau}}^{a b} \, , \label{eq-prlg_2_k_tr}
	\end{align}
	respectively, which upon differentiation, yield
	\begin{align}
		\wt{\nabla}_a \wt{c}
		& = - 2 \wt{k}^b \wt{\Cot}_{b a c} \wt{k}^c + \wt{{\tau}}_{a c} \wt{\psi}^c \, ,  \label{eq-nab_c1}\\
		\wt{\nabla}_a \wt{c}
		& =  -\tfrac{1}{n} \left( 4 \wt{{\tau}}_{a b} \wt{\psi}^b + 2 \wt{k}^{d} \wt{\Weyl}_{d a b c} \wt{{\tau}}^{b c} \right) \, . \label{eq-nab_c}
	\end{align}
	By the Leibniz rule, equation \eqref{eq-prlg_1} and the property of $\wt{\Cot}_{a b c}$, we have
	\begin{align}
		\wt{\nabla}^b (\wt{k}^a \wt{\Weyl}_{abcd} \wt{k}^d ) 
		& = - (n-1)  \wt{k}^a \wt{\Cot}_{acd} \wt{k}^d + \tfrac{3}{2}  \wt{k}^d \wt{\Weyl}_{d c a b} \wt{\tau}^{a b} \, . \label{eq:DkWk}
	\end{align}

	Henceforth, we impose the assumptions \eqref{eq:kkW} and \eqref{eq:kkY}. Recall that the integrability condition for $\wt{k}$ to be conformal Killing is that $\mathsterling_{\wt{k}} \wt{\Weyl}_{abc}{}^{d} = 0$. This means that with our choice of scale for which $\wt{k}^{a}$ is Killing, we have that $\mathsterling_{\wt{k}} \| \wt{\Weyl} (\wt{k}) \|^2 = \wt{k}^{a} \wt{\nabla}_{a} \| \wt{\Weyl} (\wt{k}) \|^2 = 0$. It immediately follows from \eqref{eq:kkY} that  $\wt{k}^a \wt{k}^c \wt{\Cot}_{a b c} = 0$. In addition, since $\wt{\nabla}^{a} \wt{\kappa}_{a} = 0$ by \eqref{eq-prlg_1}, and substituting \eqref{eq:kkW} into the LHS of \eqref{eq:DkWk}, it follows that the LHS of \eqref{eq:DkWk} is identically zero. Hence, we find that \eqref{eq:DkWk} 
	implies
	\begin{align*}
		\wt{k}^{d} \wt{\Weyl}_{d a b c} \wt{{\tau}}^{b c}
		& =  0 \, .
	\end{align*}	
	Thus, \eqref{eq-nab_c1} and \eqref{eq-nab_c} immediately gives $\wt{\nabla}_a \wt{c} =  \wt{{\tau}}_{a c} \wt{\psi}^c =  -\tfrac{4}{n} \wt{{\tau}}_{a b} \wt{\psi}^b$, i.e.\ $\wt{\nabla}_{a} \wt{c} = 0$ and $\wt{{\tau}}_{a b} \wt{\psi}^b = 0$. In particular, $\wt{c}$ is a constant. Now assuming \eqref{eq:Rickk_neg}, we conclude that $\wt{c}<0$. This allows us to rescale $\wt{k}^a$ by $-\tfrac{1}{\wt{c}}$ so that $\wt{\psi}_a \wt{k}^a = -1$. Then, by assumption \eqref{eq:kkW} and equation \eqref{eq-prlg_2_k}, we get
	\begin{align}\label{eq-prlg_2_k'}
		\wt{{\tau}}_a{}^c \wt{{\tau}}_{c b} & = - \wt{{g}}_{ab} +  2 \, \wt{{\kappa}}_{(a} \wt{\lambda}_{b)} \, ,
	\end{align}	
	where we have set $\wt{\lambda}_a = - \wt{\psi}_{a} + s \| \wt{\Weyl} (\wt{k}) \|^2 \wt{\kappa}_{a}$. We also define $\wt{\ell}^a = \wt{g}^{a b} \wt{\lambda}_b$.

		Since $\wt{\ell}^a \wt{{\tau}}_{a b} = 0$, we immediately conclude that $\wt{\ell}^a$ is null. Hence, \eqref{eq-prlg_2_k'} tells us that $\wt{{\tau}}_{a b}$ defines a bundle Hermitian structure on the screen bundle $\wt{H}=\langle \wt{k} \rangle^\perp / \langle \wt{k} \rangle$ of $\langle \wt{k} \rangle$, which is isomorphic to $\langle \wt{k} \rangle^\perp \cap \langle \wt{\ell} \rangle^\perp$. By virtue of the fact that $\wt{k}$ is null and conformal Killing, $\wt{k}$ generates a non-shearing congruence of null geodesics $\wt{\mc{K}}$. Furthermore, the algebraic condition \eqref{eq-prlg_2_k'} tells us that the twist of $\wt{\mc{K}}$, identified with $\wt{{\tau}}_{a b}$, induces an almost Robinson structure $(\wt{N},\langle \wt{k} \rangle)$ on $\wt{\mc{M}}$. Hence, by \cite{TaghaviChabert2022,Fino2023}, $(\wt{N},\langle \wt{k} \rangle)$ induces a partially integrable almost contact CR structure $(H,J)$ on the local leaf space $\mc{M}$ of $\wt{\mc{K}}$, and we shall denote by $\varpi$ is the projection from $\wt{\mc{M}}$ to $\mc{M}$. With respect to our choice of metric $\wt{g}$ for which $\wt{k}$ is Killing, $\tfrac{1}{2}\wt{\kappa}$ and $\wt{\tau}$ are the respective pullbacks of the pseudo-Hermitian form $\theta$ and its Levi form $h$ from $\mc{M}$, i.e.\ $\wt{\kappa} = 2 \varpi^* \theta$ and $\wt{\tau} = \varpi^* h$.

Let us now impose condition \eqref{eq:cond_Weyl}. Plugging $\wt{\Riem}_{a b}{}^{c d} \wt{\tau}_{b c}
= \wt{\Weyl}_{a b}{}^{c d} \wt{\tau}_{b c} - 4 \wt{\Rho}_{[a}{}^{c} \wt{\tau}_{b] c}$ into \eqref{eq-prlg_4} and using \eqref{eq:cond_Weyl} yields
	\begin{multline}
		\wt{\nabla}_{[a} \wt{\psi}_{b]}
		 = - \tfrac{1}{4} \wt{R}_{a b}{}^{c d} \wt{\tau}_{c d} + \tfrac{1}{4} \left( \tfrac{1}{2} \left( \wt{{\tau}}_{c[a} \wt{k}^{d} \wt{\Weyl}_{b] d}{}^{e f} \wt{\Weyl}_{e f g}{}^{c} \wt{k}^g  + \wt{{\kappa}}_{[a} \wt{k}^{c} \wt{\Weyl}_{b] c}{}^{d e} \wt{\Cot}_{f d e} \wt{k}^f \right) \right) \\+ \tfrac{1}{4} t \wt{\nabla}_{[a}  \left( \wt{{\kappa}}_{b]} \| \wt{\Weyl}(\wt{k}) \|^2 \right) \, . \label{eq-pre-pre-crucial}
	\end{multline}
	In addition, \eqref{eq-prlg_4} gives
	\begin{align}\label{eq-pre-crucial}
		\wt{{\kappa}}_{[a} \wt{k}^{c} \wt{\Weyl}_{b] c}{}^{d e} \wt{\Cot}_{f d e} \wt{k}^f & = 2 \, \wt{{\kappa}}_{[a} \wt{k}^{c} \wt{\Weyl}_{b] c}{}^{d e} \left( - \wt{\nabla}_d \wt{\psi}_e + \wt{\Rho}_d{}^f \wt{{\tau}}_{e f} \right) \, .
	\end{align}
		Since $\wt{k}$ generates a twisting non-shearing congruence of null geodesics, non-expanding with respect to the chosen $\wt{g}$, we may use the computation of the curvature given in \cite[Appendix~A]{TaghaviChabert2022}. It is a tedious computational matter to show that $\wt{{\kappa}}_{[a} \wt{k}^{c} \wt{\Weyl}_{b] c}{}^{d e}  \wt{\Rho}_d{}^f \wt{{\tau}}_{e f} = 0$.
	
	At this stage, let us write
	\begin{align}
		\wt{\nabla}_{[a} \wt{\lambda}_{b]} = -  \wt{B}_{a b} + \wt{{\kappa}}_{[a} \wt{C}_{b]} \,  ,
	\end{align}
	for some tensors $\wt{B}_{ab} = \wt{B}_{[a b]}$ and $\wt{C}_{a}$. Since $\wt{k}$ is conformal Killing, we have $\wt{k}^{a} \wt{B}_{a b} = \wt{k}^{a} \wt{C}_{a} = 0$ and $\wt{\ell}^{a} \wt{B}_{a b} = \wt{\ell}^{a} \wt{C}_{a} = 0$. Using $\wt{\psi}_{a} = - \wt{\lambda}_{a} + s \| \wt{\Weyl}(\wt{k}) \|^2 \wt{\kappa}_{a}$, combining the previous display with \eqref{eq-pre-crucial}, we obtain
	\begin{multline}\label{eq-crucial}
		\wt{B}_{a b} -  \wt{{\kappa}}_{[a} \wt{C}_{b]} 
		=  - \tfrac{1}{4} \wt{\Riem}_{a b}{}^{c d} \wt{\tau}_{c d} + \tfrac{1}{8} \left( \wt{{\tau}}_{c[a} \wt{k}^{d} \wt{\Weyl}_{b] d}{}^{e f} \wt{\Weyl}_{e f g}{}^{c} \wt{k}^g  - 2 \, \wt{{\kappa}}_{[a} \wt{k}^{c} \wt{\Weyl}_{b] c}{}^{d e} \wt{B}_{d e}  \right) \\
		+ \tfrac{1}{4} (t - 4 s) \left( \| \wt{\Weyl}(\wt{k}) \|^2 \wt{\tau}_{a b} - \wt{{\kappa}}_{[a} \wt{\nabla}_{b]} \| \wt{\Weyl}(\wt{k}) \|^2 \right) \, .
	\end{multline}
	We take the components of $\wt{B}_{a b}$ with respect to the splitting of $\langle \wt{k} \rangle^\perp \cap \langle \wt{\ell} \rangle^\perp$ into the eigenbundles of the bundle complex structure defined by $\wt{\tau}_{a b}$. In the obvious index notation, we find
	\begin{align}
		\wt{B}_{\alpha \beta} & =  - \tfrac{1}{4} \wt{\Riem}_{\alpha \beta}{}^{c d} \wt{\tau}_{c d}\, , \label{eq-crucial-Bp}  \\
		\wt{B}_{\alpha \bar{\beta}} & =  - \tfrac{1}{4} \wt{\Riem}_{\alpha \bar{\beta}}{}^{c d} \wt{\tau}_{c d} - \tfrac{\i}{8}  \wt{k}^{d} \wt{\Weyl}_{\alpha d}{}^{e f} \wt{\Weyl}_{e f g \bar{\beta}} \wt{k}^g + \tfrac{1}{4} (t - 4 s) \i \| \wt{\Weyl}(\wt{k}) \|^2 h_{\alpha \bar{\beta}}  \, , \label{eq-crucial-Bm} \\
		\wt{C}_{\alpha} & =  - \tfrac{1}{2} \wt{\Riem}_{\alpha 0 \bar{\beta}}{}^{c d} \wt{\tau}_{c d} + \tfrac{1}{4} \left( - \tfrac{\i}{2} \wt{k}^{d} \wt{\Weyl}_{0 d}{}^{e f} \wt{\Weyl}_{e f g \alpha} \wt{k}^g  + \wt{k}^{c} \wt{\Weyl}_{\alpha c}{}^{d e} \wt{B}_{d e} \right) + \tfrac{1}{4} (t - 4 s) \wt{\nabla}_{\alpha} \| \wt{\Weyl}(\wt{k}) \|^2 \, . \label{eq-crucial-C}
	\end{align}
	Tracing \cite[equations~(A.10), (A.11) and (A.14)]{TaghaviChabert2022} over the first two indices yields
	\begin{align}
		\wt{\Riem}_{\gamma}{}^{\gamma}{}_{\alpha \beta} & =   2 ( m + 1) \, \i \wt{B}_{\alpha \beta}   -  \nabla^{\gamma} \Nh_{\alpha \beta \gamma} \, , \label{eq:R-DbT-Tr}\\
		\wt{\Riem}_{\gamma}{}^{\gamma}{}_{\bar{\beta} \alpha} & =  \Ric_{\alpha \bar{\beta}}  -  2 (m + 1) \, \i \wt{B}_{\alpha \bar{\beta}} 
		- 2 \, \i \wt{B}_{\gamma}{}^{\gamma} h_{\alpha \bar{\beta}}  
		+  \Nh_{\alpha \epsilon \gamma}  \Nh{}_{\bar{\beta}}{}^{\epsilon \gamma} \, , \label{eq:R-ST-Tr} \\
		\wt{\Riem}_{\beta}{}^{\beta}{}_{\alpha 0} & =  - \nabla^{\beta} \wt{B}_{\alpha \beta}  + \nabla_{\beta} \wt{B}_{\alpha}{}^{\beta}   +  2 m \, \i \wt{C}_{\alpha} + \wt{B}^{\beta \delta} \Nh_{\delta \alpha \beta}    - \tfrac{1}{2} \nabla^{\beta} \WbA_{\alpha \beta} - \tfrac{1}{2} \WbA^{\beta \delta} \Nh_{\delta \alpha \beta}  \, . \label{eq:R-D0B-Tr}
	\end{align}
	Here, $\nabla$ is the Webster connection corresponding to $\theta$, $\Ric_{\alpha \bar{\beta}}$ its Webster--Ricci tensor and $\Nh_{\alpha \beta \gamma}$ the Nijenhuis tensor. Similarly, using \cite[equations~(A.3) and (A.6)]{TaghaviChabert2022}, we compute
	\begin{align}
		\wt{k}^{d} \wt{\Weyl}_{\beta d}{}^{e f} \wt{\Weyl}_{e f g}{}^{\alpha} \wt{k}^g 
		& =  - 4 \Nh_{\gamma \delta \beta} \Nh^{\gamma \delta \alpha} \, , \label{eq:WWm} \\
		\wt{k}^{d} \wt{\Weyl}_{0 d}{}^{e f} \wt{\Weyl}_{e f g \alpha} \wt{k}^g 
		& =  -  4 \Nh_{\gamma \delta \alpha}  \wt{B}^{\gamma \delta} \, ,  \label{eq:WW0} \\
		\wt{k}^{c} \wt{\Weyl}_{\beta c}{}^{d e} \wt{B}_{d e} & =  2 \i \Nh_{\gamma \delta \beta} \wt{B}^{\gamma \delta} \, .  \label{eq:WB} 
	\end{align}
	This means that
	\begin{align}
	\| \wt{\Weyl}(\wt{k}) \|^2 & :=	\wt{k}^{a} \wt{\Weyl}_{a b c d} \wt{k}^e \wt{\Weyl}_{e}{}^{b c d} =  8 \| \Nh \|^2  \, . \label{eq:WWmC}
	\end{align}
	We now note $\wt{\Riem}_{a b}{}^{c d} \wt{\tau}_{b c}
	= - 2 \i \wt{\Riem}_{a b \gamma}{}^{\gamma}$.		
	Hence, plugging \eqref{eq:R-DbT-Tr} into \eqref{eq-crucial-Bp} and solving for $B_{\alpha \beta}$ gives
	\begin{align}\label{eq:Bab}
		\wt{B}_{\alpha \beta}  & =   -  \tfrac{\i}{2(m+2)} \nabla^{\gamma} \Nh_{\alpha \beta \gamma} \, .
	\end{align}
	Similarly, plugging \eqref{eq:R-ST-Tr} and \eqref{eq:WWm} into \eqref{eq-crucial-Bm} and solving for $\wt{B}_{\alpha \bar{\beta}}$ leads to
	\begin{align*}
		\wt{B}_{\alpha \bar{\beta}}	& = - \tfrac{\i}{2(m+2)} \left( \Ric_{\alpha \bar{\beta}} - \Nh_{\gamma \delta \alpha} \Nh^{ \gamma \delta}{}_{\bar{\beta}} + \Nh_{\alpha \gamma \delta} \Nh_{\bar{\beta}}{}^{\gamma \delta} \right) 
		- \tfrac{1}{m+2} \wt{B}_{\gamma}{}^{\gamma} h_{\alpha \bar{\beta}} + \tfrac{2(t - 4 s)}{m+2}  \i \| \Nh \|^2  h_{\alpha \bar{\beta}} \, .
	\end{align*}
	Taking the trace and solving for $\wt{B}_{\gamma}{}^{\gamma}$ gives
	\begin{align*}
		\wt{B}_{\gamma}{}^{\gamma} & = - \tfrac{1}{2} \i \Rho + (t - 4 s) \tfrac{m}{m+1} \i \| \Nh \|^2  \, .
	\end{align*}
	By substituting back into the last but one equation, we obtain
	\begin{align*}
		\wt{B}_{\alpha \bar{\beta}} & = - \tfrac{\i}{2(m+2)} \left( \Ric_{\alpha \bar{\beta}} - \Nh_{\gamma \delta \alpha} \Nh^{ \gamma \delta}{}_{\bar{\beta}} + \Nh_{\alpha \gamma \delta} \Nh_{\bar{\beta}}{}^{\gamma \delta}  - \Rho h_{\alpha \bar{\beta}} \right) + \tfrac{(t - 4 s) }{m+1} \i \| \Nh \|^2 h_{\alpha \bar{\beta}} \, .
	\end{align*}
	On the other hand, condition \eqref{eq:kkW} also gives us, with reference to \cite[equations~(A.8) and (A.21)]{TaghaviChabert2022},
	\begin{align*}
		16 s \| \Nh \|^2 & =  \tfrac{1}{m(2m+1)} \left( 4m (t - 4 s) -1 \right) \| \Nh \|^2  \, .
	\end{align*}
	Assuming that $\| \Nh \|^2$ nowhere vanishes, this implies
	\begin{align*}
		t - 4 s & = \tfrac{1 + 16 (2m+1) s}{4m} \, , & \mbox{i.e.} & & 32m(m+1)s & = 4m t -1 \, . 
	\end{align*}
	Thus,
		\begin{align}\label{eq:Babb}
		\wt{B}_{\alpha \bar{\beta}} & = - \tfrac{\i}{2(m+2)} \left( \Ric_{\alpha \bar{\beta}} - \Nh_{\gamma \delta \alpha} \Nh^{ \gamma \delta}{}_{\bar{\beta}} + \Nh_{\alpha \gamma \delta} \Nh_{\bar{\beta}}{}^{\gamma \delta}  - \Rho h_{\alpha \bar{\beta}} \right) + \tfrac{1 + 16 (2m+1) s}{4m(m+1)} \i \| \Nh \|^2  h_{\alpha \bar{\beta}} \, .
	\end{align}
	Finally, by plugging \eqref{eq:R-D0B-Tr}, \eqref{eq:WW0} and \eqref{eq:WB} into \eqref{eq-crucial-C} and solving for $\wt{C}_{\alpha}$, we find
	\begin{align}\label{eq:Ca}
		\wt{C}_{\alpha} & = \tfrac{1}{2}  T_{\alpha} +  \tfrac{1 + 16 (2m+1) s}{4m(m+1)} \nabla_{\alpha}  \| \Nh \|^2 \, ,
	\end{align}
	where $T_{\alpha}$ is given by \eqref{eq:CRT}.
	
	Let $(\wt{\mc{M}}',\wt{\mbf{c}}') \accentset{\varpi'}{\longrightarrow} (\mc{M},H,J)$ be the $\alpha$-Fefferman space of $(\mc{M},H,J)$. Let $\wt{\omega}_{\theta}$ be the induced Webster connection on $\wt{\mc{M}}$ compatible with the contact form $\theta$, and denote $\wt{g}'$ its corresponding Fefferman metric. We set
	\begin{align*}
		\wt{\lambda}' & = \wt{\omega}_{\theta} - \tfrac{1}{m+2}  \left( \Rho + \tfrac{\alpha(m+2)}{2m(m+1)} \| \Nh \|^2 \right) \theta  \, ,
	\end{align*}
	so that $\wt{g}'  = 4 \theta \odot \wt{\lambda}' + h$. We find that
	\begin{align*}
		\d \wt{\lambda}' & =
		  - \wt{B}'_{\alpha \beta} \theta^{\alpha} \wedge \theta^{\beta} - 2  \wt{B}'_{\alpha \bar{\beta}} \theta^{\alpha} \wedge \overline{\theta}{}^{\bar{\beta}} - \wt{B}'_{\bar{\alpha} \bar{\beta}} \overline{\theta}{}^{\bar{\alpha}} \wedge \overline{\theta}{}^{\bar{\beta}} - 2 \wt{C}'_{\bar{\alpha}} \overline{\theta}{}^{\bar{\alpha}} \wedge  \theta  - 2 \wt{C}'_{\alpha}  \theta^{\alpha} \wedge \theta \, ,
	\end{align*}
	where
	\begin{subequations}
	\begin{align}
		\wt{B}'_{\alpha \bar{\beta}} & = - \tfrac{\i}{2} {\Rho}_{\alpha \bar{\beta}} - \tfrac{\i}{2(m+2)} \left( {\Nh}_{\alpha \gamma \delta} {\Nh}_{\bar{\beta}}{}^{\gamma \delta} - {\Nh}_{\gamma \delta \alpha} {\Nh}^{ \gamma \delta}{}_{\bar{\beta}} \right) + \alpha  \tfrac{\i}{4m(m+1)} \| \Nh \|^2  h_{\alpha \bar{\beta}} \, ,  \label{eq:B'abb} \\
		\wt{B}'_{\gamma \delta} & =  - \tfrac{\i}{2(m+2)} {\nabla}^{\alpha} {\Nh}_{\gamma \delta \alpha} \, ,  \label{eq:B'ab} \\
		\wt{C}'_{\alpha} & = \tfrac{1}{2} \mathsf{T}_{\alpha} + \alpha \tfrac{1}{4m(m+1)} \nabla_{\alpha} \left( \| \Nh \|^2  \right)  \, , \label{eq:C'a}
	\end{align}
	\end{subequations}
Hence, setting
\begin{align*}
	s & = \tfrac{\alpha - 1}{16(2m+1)} \, , &
	t & = \tfrac{1}{4m} \left( 1 + \tfrac{2m(m+1)}{2m+1} (\alpha - 1) \right) \, ,
\end{align*}
so that $\alpha = 1 + 16 (2m+1) s$, we see that conditions \eqref{eq:Rickk_neg}, \eqref{eq:kkW}, \eqref{eq:kkY} and \eqref{eq:cond_Weyl} are none other than the respective hypotheses  \eqref{eq-Rho_sc}, \eqref{eq-Wkk}, \eqref{eq-Ykk} and \eqref{eq-int_cond} of Theorem \ref{thm:Fefferman-CRA}.

Now choose some nowhere vanishing density $\sigma \in \Gamma(\mc{E}(1,0))$ such that $\theta = (\sigma \ol{\sigma})^{-1} \bm{\theta}$ and with corresponding local coordinate $\phi'$ so that $\wt{\omega}_\theta$ takes the form
\begin{align*}
	\wt{\omega}_\theta & = \d \phi' + \tfrac{\i}{2} \left( \sigma^{-1} \nabla \sigma - \ol{\sigma}^{-1} \nabla \ol{\sigma} \right) \, .
\end{align*}
Define $\mr{\lambda} := \wt{\lambda}' - \d \phi'$. Then $\mr{\lambda}$ is a one-form on $\mc{M}$, and, on comparing \eqref{eq:Babb}, \eqref{eq:Bab} and \eqref{eq:Ca} with \eqref{eq:B'abb}, \eqref{eq:B'ab} and \eqref{eq:C'a}, we find that $\d \wt{\lambda} = \d \mr{\lambda}$, which locally implies that
	\begin{align*}
		\wt{\lambda} - \mr{\lambda} & = \d \phi \, .
	\end{align*}
	for some smooth function $\phi$ on $\wt{\mc{M}}$. This allows us to define a map $\iota$ from $(\wt{\mc{M}}, \wt{g})$ to $(\wt{\mc{M}}',\wt{g}')$ by $\iota(\wt{p})=(\varpi(\wt{p}),\e^{\i \phi(\wt{p})}\sigma^{-1}(\varpi(\wt{p})))$ for any $\wt{p} \in \wt{\mc{M}}$. Then $\iota$ is a bundle map with $\varpi' \circ \iota = \varpi$, which is also a local isometry, i.e.\ $\iota^* \wt{g}' = \wt{g}$, and extends to a local conformal isometry. It now follows that $(\wt{\mc{M}},\wt{\mbf{c}})$ and $(\wt{\mc{M}}',\wt{\mbf{c}}')$ are locally conformally isometric as claimed.

	%
	%
	%
	
	\bibliographystyle{abbrv}
	\bibliography{biblio}
	
	
\end{document}